\documentclass[12pt,twoside]{article}
\usepackage{etex}
\usepackage[dvipsnames]{xcolor} 
\usepackage{color}
\usepackage{tcolorbox}
\usepackage{booktabs}
\usepackage{amsthm}
\usepackage{amsfonts,amssymb,amsxtra,url,float} 
\allowdisplaybreaks[4] 
\usepackage[colorlinks,
  linkcolor=magenta, %
  anchorcolor=Periwinkle,
  citecolor=red,
  urlcolor=blue
  ]{hyperref} 
\usepackage{enumitem}
\usepackage{geometry} 
\usepackage{rotating} 
\usepackage{lscape} 
\usepackage{multirow}
\usepackage{graphicx} 
\usepackage{subfigure} 
\usepackage{tikz}
\usepackage{pgfplots}
\usepackage{tikz-3dplot}
\usetikzlibrary{patterns}
\usetikzlibrary{3d,calc}
\usetikzlibrary{decorations.pathreplacing,decorations.markings}
 \tikzset{
  on each segment/.style={
    decorate,
    decoration={
      show path construction,
      moveto code={},
      lineto code={
        \path [#1]
        (\tikzinputsegmentfirst) -- (\tikzinputsegmentlast);
      },
      curveto code={
        \path [#1] (\tikzinputsegmentfirst)
        .. controls
        (\tikzinputsegmentsupporta) and (\tikzinputsegmentsupportb)
        ..
        (\tikzinputsegmentlast);
      },
      closepath code={
        \path [#1]
        (\tikzinputsegmentfirst) -- (\tikzinputsegmentlast);
      },
    },
  },
  mid arrow/.style={postaction={decorate,decoration={
        markings,
        mark=at position 0.6 with {\arrow[#1]{stealth}} 
      }}},
}
\usetikzlibrary{arrows}
\usetikzlibrary{trees}

\usetikzlibrary{matrix}
\usetikzlibrary{patterns}
\usetikzlibrary{shadings} 
\usepackage{fancyhdr} 
\def\headertitle{Integral coefficient rings and homological dimensions of algebras}
\def\fstpage{1} 
\def\page{$\begin{matrix} {\color{white}0} \\ \thepage \end{matrix}$} 

\usepackage[all]{xy} 
\usepackage{dsfont} 
\usepackage{cite}
\usepackage{mathrsfs} 
\numberwithin{figure}{section}
\usepackage{marginnote} 
\usepackage{graphicx} 
\usepackage{multicol} 

\usepackage{enumitem}
\setenumerate[1]{itemsep=0pt,partopsep=0pt,parsep=\parskip,topsep=3pt}
\setitemize[1]{itemsep=0pt,partopsep=0pt,parsep=\parskip,topsep=3pt}
\setdescription{itemsep=0pt,partopsep=0pt,parsep=\parskip,topsep=3pt}
\setlist[itemize]{leftmargin=35pt}
\setlist[enumerate]{leftmargin=35pt}
\usepackage{changepage} 
\newcommand{\checks}[1]{{\color{cyan}{#1}}} 

\def\addtext{\color{black}}
\usepackage{contour} 
\usepackage{xcolor}
\contourlength{0.03em}
\contournumber{auto}

\newtheorem{theoremAlph}{Theorem}

\newtheorem{theorem}{Theorem}[section]
\newtheorem{lemma}[theorem]{Lemma}
\newtheorem{corollary}[theorem]{Corollary}
\newtheorem{main theorem}[theorem]{Main Theorem}
\newtheorem{proposition}[theorem]{Proposition}
\newtheorem{definition}[theorem]{Definition}

\newtheorem{remark}[theorem]{Remark}
\newtheorem{example}[theorem]{Example}

\newtheorem{question}[theorem]{Question}
\newtheorem{fact}[theorem]{Fact}
\newtheorem{assumption}[theorem]{Assumption}
\numberwithin{equation}{section}

\def\orcid{
\begin{tikzpicture}[baseline=-1mm]
\filldraw[Green!35] (0,0) circle (5pt);
\filldraw[white] (0,0) node{\tiny\textbf{iD}};
\end{tikzpicture}
}

\def\NN{\mathbb{N}} 
\def\ZZ{\mathbb{Z}} 
\def\QQ{\mathbb{Q}} 
\def\RR{\mathbb{R}} 
\def\CC{\mathbb{C}} 
\def\II{\mathbb{I}}
\def\FF{\mathbb{F}}

\newcommand{\Coker}{\operatorname{Coker}}
\newcommand{\Ker}{\operatorname{Ker}}
\newcommand{\Ima}{\operatorname{Im}}

\newcommand{\modcat}{\mathsf{mod}}
\newcommand{\Modcat}{\mathsf{Mod}}

\newcommand{\proj}{\mathsf{proj}}
\newcommand{\inj}{\mathsf{inj}}
\newcommand{\pandi}{\mathsf{proj}\text{-}\mathsf{inj}}
\newcommand{\pandimod}{{^P}\mkern-5.5mu\text{\rm\&}\mkern-4mu{_I}}

\newcommand{\Dcat}{\mathsf{D}}
\newcommand{\ind}{\mathsf{ind}}
\newcommand{\add}{\mathsf{add}}
\newcommand{\Q}{\mathcal{Q}} 
\def\I{\mathcal{I}}
\def\fcts{\mathfrak{s}}
\def\fctt{\mathfrak{t}}

\newcommand{\e}{\varepsilon}
\newcommand{\Hom}{\mathrm{Hom}} %
\newcommand{\End}{\mathrm{End}} %
\newcommand{\Aut}{\mathrm{Aut}} %
\newcommand{\Ext}{\mathrm{Ext}} %
\def\rad{\mathrm{rad}}
\def\soc{\mathrm{soc}}
\def\top{\mathrm{top}}
\def\rank{\mathrm{rank}}
\newcommand{\op}{\mathrm{op}}
\newcommand{\ol}[1]{\overline{#1}}
\newcommand{\w}[1]{\widehat{#1}}
\newcommand{\To}[1]{\mathop{-\!\!\!-\!\!\!\longrightarrow}\limits^{#1}}
\newcommand{\oT}[1]{\mathop{\longleftarrow\!\!\!-\!\!\!-}\limits^{#1}}

\def\pdim{\mathrm{proj.dim}}
\def\idim{\mathrm{inj.dim}}
\def\gldim{\mathrm{gl.dim}}
\def\findim{\mathrm{fin.dim}}
\def\Findim{\mathrm{Fin.dim}}

\def\bfS{\mathbf{S}}

\def\ave{\mathfrak{A}}
\def\id{\mathbf{1}}
\def\ident{\mathrm{id}}

\def\compos{\ \lower-0.2ex\hbox{\tikz\draw (0pt, 0pt) circle (.1em);} \ }
\def\ds{\le_{\oplus}}
\def\isods{{{\raisebox{0.26em}{$\lesssim$}}{\mkern-13.5mu}{\raisebox{-0.4em}{$=$}}{}_{\oplus}}}

\newcommand{\defines}[1]{{\it\color{blue!75}{#1}}}
\title{\bf Integral coefficient rings and homological dimensions of algebras
}

\vspace{5mm}

\author{
Yu-Zhe Liu$^{\ref{Author1}, \href{https://orcid.org/0009-0005-1110-386X}{\orcid}\ref{orcid1}~\ref{CorrespondingAuthor}}$
}
\date{ }
\begin{document}






\maketitle

\begin{enumerate}[label=\textbf{\color{red}\arabic*}] \footnotesize
  \item
    \begin{center}
      School of Mathematics and Statistics, Guizhou University,
      Guiyang 550025, Guizhou, China;

      E-mail: \url{liuyz@gzu.edu.cn} / \url{yzliu3@163.com} (Y.-Z Liu)
    \end{center} \label{Author1}
%
%
\end{enumerate}
\vspace{2mm}
\begin{enumerate}[label=\textbf{\color{red}$\dag$}]
  \item \footnotesize
    \begin{center}
      Corresponding author
    \end{center} \label{CorrespondingAuthor}
\end{enumerate}
\vspace{2mm}
\begin{enumerate} \footnotesize
    \item[{\orcid}] \centering  ORCID: \href{https://orcid.org/0009-0005-1110-386X}{0009-0005-1110-386X}
      \label{orcid1}
\end{enumerate}

\vspace{1mm}

\begin{adjustwidth}{1cm}{1cm}
  \noindent \footnotesize
  \textbf{Abstract}:
We define the integral profiles of all modules and introduce integral coefficient rings
${^{\mathscr{P}}\mkern-6.5mu\text{\it\&}\mkern-5.5mu{_\mathscr{I}}}(A)$
for all finite-dimensional complex algebras $A$.
The integral profile of a module is a matrix with parameters. We provide a classification theorem for modules, to be precise,
\begin{itemize}
  \item[(1)] two modules $M\cong N$ are isomorphic if and only if their integral profiles are similar,
    i.e., $M\cong N$ if and only if $\displaystyle \int M \sim \int N$.
\end{itemize}
That is, the integral profile is a complete invariant of finite-dimensional modules.
Furthermore, we show the following results in this paper:
\begin{itemize}
  \item[(2)] we introduce the central integrals of algebras and show that it is isomorphic to the center of algebras;
  \item[(3)] we provide a descriptions for some special modules;
  \item[(4)] integral coefficient ring of $A$ (with a compatible orthogonal fixed embedding system) has Morita invariance;
  \item[(5)] the global dimension of $A$ is finite if and only if the embedded integral profile of $\mathrm{top}(A)$ lies in ${^{\mathscr{P}}\mkern-6.5mu\text{\it\&}\mkern-5.5mu{_\mathscr{I}}}(A)[x]$;
  \item[(6)] the finitistic dimension of $A$ is finite if and only if each embedded integral profile of $M$ lying in ${^{\mathscr{P}}\mkern-6.5mu\text{\it\&}\mkern-5.5mu{_\mathscr{I}}}(A)[x]$ implies that its degree is less than or equal to a fixing integer $d\in\mathbb{N}^+$;
  \item[(7)] we provide two sufficient conditions, such that if a finite-dimensional complex algebra $A$ satisfies one of them, then its finitistic dimension is finite.
\end{itemize}
\vspace{1mm}


     \label{2020MSC}

\vspace{1mm}

     \label{Keywords}
\end{adjustwidth}

\newpage

\tableofcontents


\def\la{\langle} 
\def\ra{\rangle} 
\def\lala{\langle\!\langle}
\def\rara{\rangle\!\rangle}
\def\=<{\leqslant}
\def\>={\geqslant}
\def\vecd{\pmb{d}}
\def\vecg{\pmb{g}}
\def\spa{\mathrm{span}}
\def\Gal{\mathrm{Gal}}
\def\Mat{\mathbf{Mat}}
\def\bfE{\pmb{E}}
\def\op{\mathrm{op}}
\def\norm{\mathfrak{n}}
\def\bigboxplus{\mathop{ \raisebox{-0.3em}{ \scalebox{1.75}{\text{\(\boxplus\)}} } }}

\def\scrN{\mathscr{N}\!or}
\def\scrA{\mathscr{A}}
\def\Pc{\mathds{P}}
\def\Int{\mathrm{Int}}
\newcommand{\CentIntA}[1]{\mathfrak{I}_A{(#1)}}
\newcommand{\olCentIntA}[2]{\overline{\mathfrak{I}}_{A,#1}{(#2)}}

\def\invlim{\underleftarrow{\lim}}
\def\dirlim{\underrightarrow{\lim}}

\def\dimA{d_A}
\def\dimB{d_B}
\def\homo{\varsigma}
\def\itLamb{\mathit{\Lambda}}
\def\Verb{\pmb{|}}

\def\dd{\mathrm{d}}
\def\Daniel{\mathrm{D}}
\def\Bochner{\mathrm{B}}
\def\Lebesgue{\mathrm{L}}
\def\IntProf{\mathop{\text{\contour{black}{$\displaystyle\color{white}\int$}}}\nolimits}
\def\intProf{\mathop{\text{\contour{black}{$\color{white}\int$}}}}
\def\IntProfEmb{\mathop{\text{\contour{black}{$\displaystyle\color{white}\int$}} \mkern-18mu\text{\rotatebox{90}{$\color{cyan}\dag$}}}\nolimits}
\def\intProfEmb{\mathop{\text{\contour{black}{$\color{white}\int$}} \mkern-16mu\text{\rotatebox{90}{$\color{cyan}\dag$}}}\nolimits}
\def\fpipset{\mathfrak{F}^{\mathrm{proj}}}
\def\pandiring{{^{\mathscr{P}}\mkern-6.5mu\text{\it\&}\mkern-5.5mu{_\mathscr{I}}}}
\def\reconstr{\partial^{\mathscr{R}ec}}
\def\recring{{\mathscr{R}ec}}
\def\ricring{{\mathscr{R}ic}}
\def\compos{{\begin{smallmatrix}\circ\end{smallmatrix}}}

\newcommand{\basis}[1]{\mathfrak{B}_{#1}}
\newcommand{\cent}[1]{\mathcal{Z}{(#1)}}

\newcommand{\Sol}{\mathbf{Sol}}            
\newcommand{\LODE}{\mathbf{L}}
\newcommand{\EndL}{\End_{\itLamb}}         
\newcommand{\HomL}{\Hom_{\itLamb}}         
\newcommand{\Lip}{\mathrm{Lip}}            
\newcommand{\Dop}{D}                       
\newcommand{\TV}{\w{T}_V}                  
\newcommand{\DV}{\Dop_V}                   
\newcommand{\SV}{\bfS_{\homo}(\II_{\itLamb}, V)} 
\newcommand{\wSV}{\w{\bfS_{\homo}(\II_{\itLamb}, V)}} 
\newcommand{\wS}{\w{\bfS_{\homo}(\II_{\itLamb})}} 
\def\bfF{\mathbf{Fun}}
\def\ACI{\textsc{Aci}}
\def\Nat{\mathrm{Nat}}


\section{Introduction}\label{sec:intro}


\subsection{Backgrounds}

\textsl{The Backgrounds of this paper mainly lies in two aspects: categorical descriptions of integrations and homological dimensions of algebras.}

\paragraph{Categorical descriptions of integrations.}
It is well-known that classical vector-valued integration places integrable functions with values in Banach spaces in an analytic setting, and standard references include \cite[etc]{DunfordSchwartz1958, Folland1999, Ryan2002}.
On the algebraical side, \cite{Sag1965ana, CL2018Cart-int, CL2018int, Guo2013, Lei2023FA}
and \cite{BCS2015diff, Lemay2019, GP2020, APL2021diff, IL2023ana, Lemay2023, BCS2006diff} respectively provided
some algebraical descriptions of integrations and differentials by using categories,
where Leinster gave a categorical derivation of Lebesgue integration \cite{Lei2023FA}.
In \cite{LLHZ2025, Liu2025pre}, the authors proved that Leinster's work can be adapted to finite-dimensional algebras and normed tensor rings. Although the details of the proofs differ considerably from those in \cite{Lei2023FA}, the main results remain consistent.
These works provide the analytic and categorical environment in which the integral of a module representation can be defined.

\paragraph{Homological dimensions.}
Homological dimension, which originated from Alexandroff's work \cite{Alex1932, Alex1968}, is a central concept in algebra and topology. It is essentially a numerical invariant that characterizes the complexity of a geometric object, a topological object, or an algebraic object.
Homological dimension is also deeply connected with the more intuitive notion of the \v{C}ech--Lebesgue covering dimension of Tychonoff spaces (see \cite{Vcech1933})
\footnote{Tychonoff spaces are also referred to as completely regular Hausdorff spaces, or $T_{3\frac{1}{2}}$ spaces.},
and can be used to study the behaviour of spaces under continuous maps. Specifically, the topological property of \v{C}ech--Lebesgue covering dimension $\=< n$ can be equivalently described by using the language of algebraic topology and homology, and the notion of homological dimension arose precisely from this.
Since the foundational works of Cartan and Eilenberg in \cite{CE1956}, these invariants have become indispensable in homological algebra and in the representation theory of Artin algebras.
For a finite-dimensional algebra $A$, the projective dimension $\pdim_A M$ and injective dimension $\idim_A M$ of individual left/right $A$-modules, together with the global dimension $\gldim A$ of $A$, provide information about the homological properties of $A$ and its module category.
Auslander's classical study of global dimension also shows that this invariant can be computed by some smaller collection of modules \cite{Aus1955}.
Afterwards, homological dimensions of some important classes of algebras have also been computed,
such as \cite[etc]{GR2005, Bel2011, MT2015, ChX2017Finidim, Z2017globalcohomo, ZZ2020, FOZ2025,XinZhang2026}.

\subsection{Motivations}

Initially, Leinster pointed out that the Lebesgue integral of functions of one real variable can be translated into morphisms in a certain category of Banach modules, which he denoted by $\scrA^1$. This morphism starts at the initial object in this category, namely the classical $L^1$ space, and ends at the Banach space $\RR$, which is constrained by the constant function $\RR \to 1$ and the averaging map $\RR^{\oplus 2} \to \RR, (x,y)\mapsto\frac{x+y}{2}$, see \cite[Theorem 2.1, Proposition 2.2]{Lei2023FA}.
In \cite{Liu2025pre}, the author proved that, given an algebra homomorphism $\homo:A\to B$ between two finite-dimensional algebras,
for maps defined on $A$ and taking values in $B$, as long as they can be approximated by step functions,
one can assign to them the corresponding integrals and obtain an algebraic version of $L^p$-spaces ($p\>= 1$).
This $L^p$-space is described as an initial object in a category $\scrA_{\homo}^p$ consisting of Banach modules.
Moreover, when $p=1$, $\scrA_{\homo}^p$ reduces to $\scrA_{\homo}^1$, and the morphisms from $L^1$ (as the initial object) to other objects recover many conclusions in analytic settings,
among which is abstract integration (namely Daniell integral, see \cite{Dan1918, Dan1919-I, Dan1919-II}).
Then, A natural idea is whether one can construct the corresponding abstract $L^p$-spaces for homomorphisms of modules.
Houjun Zhang once made an attempt and defined abstract integrals for modules. This definition complies with Daniell's convention on abstract integration, so it is natural to obtain an equivalent characterization of module isomorphisms via integrals of homomorphisms.
To this end, he considered constructing $\scrA^p$ and $L^p$-spaces on graded algebras.
Although this construction eventually made some progress,
the integral characterization of graded module isomorphisms did not work very well,
so he judged that this definition of integral loses much information about modules.
Fortunately, $\homo \in L^1$, so we can apply to it by using the morphisms in $\scrA^1$ that start from $L^1$.
Although, in many contexts, a left $A$-module over an $\FF$-algebra $A$ is defined as a linear space $M$ equipped with a left $A$-linear action $A\times M \to M$ compatible with the scalar multiplication of the linear space,
it is a well-known fact that every such left $A$-module can be equivalently transformed into an algebra homomorphism $\rho_M:A\to \End_{\FF}(M)$,
where $\End_{\FF}(M)$ is the set of all endomorphisms of $M$, which is also an $\FF$-algebra and is called the endomorphism algebra of $M$.
This allows us to define the integral of the module $M$ as the integral of $\rho_M$, namely
\[ \int_{\II_A} M := (\scrA^1_{\rho_M}) \int_{\II_A} \rho_M \dd\mu_{\II_A}, \]
where $\II_A\subseteq A$ is a pre-selected integration domain and $\mu_{\II_A}$ is a pre-selected measure.
This fact became the starting point of this article.

\subsection{Outline and overview}

Let $\FF$ be a field. We know that every $\FF$-algebra $A$ can be equipped with a norm so that $A$ becomes a normed algebra.
Note that the construction of such a norm depends on the norm already defined on $\FF$.
Since linear spaces can always be normed (need Axiom of Choice), this does not pose a confusing issue.
According to the categorified integration method in \cite{Lei2023FA, LLHZ2025, Liu2025pre},
the prerequisite for defining an integral on an algebra is that the algebra is a Banach algebra.
Therefore, when considering integrals over an algebra $A$, it is natural to take its completion.
For simplicity, except for some basic concepts and definitions from representation theory and homological algebra,
all computations involving integrals in this paper will be carried out under the assumption that $\FF$ is the complex field $\CC$.
This paper is divided into two parts.

\begin{itemize}
  \item Part~I recalls special Banach module categories and the categorified integral, defines integrals of arbitrary modules, and computes them for direct sums, projective modules and injective modules in both probability and parameterized measure spaces.
  \item Part~II introduces integral profiles and homological dimensions.  We establish the matrix characterizations of projective and injective modules, define the total profiles attached to minimal (co)resolutions, construct the integral coefficient ring, and derive the characterizations of global and finitistic dimensions.
\end{itemize}

The theoretical outline of this paper is as follows.

First, we recall the category $\scrA^p_{\homo}$, representation theory, and relevant preliminaries in homological theory, and introduce the integral of modules.
We introduce the average central integral $\ACI(A)$ of $A$.
It is the normalized integral of a center-valued function over a measurable subset of a given region $\II_{\cent{A}}$ of integration.
we prove that the average central integral $\ACI(A)$ of a finite-dimensional complex algebra $A$, regarded as the regular module, is isomorphic to its center $\cent{A}$.

\begin{theoremAlph}\label{thm:main 2608152359}
Let $A$ be a finite-dimensional complex algebra.
For the correspondence
\[c \mapsto
  \frac{1}{\mu_{\II_{\cent{A}}}(\II_{\cent{A}})}
  \displaystyle\int_{\II_{\cent{A}}} c\id_{\II_{\cent{A}}} \dd\mu_{\II_{\cent{A}}},\]
the following statements hold.
\begin{enumerate}[label={\rm(\arabic*)}]
  \item {\rm(Proposition~\ref{prop:AveCentInt}, Corollary \ref{coro:frakZ})}
    The correspondence as above gives a bijection between $\cent{A}$ and $\ACI(A)$.
    Furthermore, it is an isomorphism of algebras.
    \label{thm:main 2608152359 1}
  \item {\rm(Theorem~\ref{thm:AveCentInt})}
    Each element in $\ACI(A)$ induces a natural transformation from the functor $\ident_{{_A\modcat}}: {_A\modcat} \to {_A\modcat}$ to itself.
    \label{thm:main 2608152359 2}
  \item {\rm(Theorem~\ref{thm:AveCentIntDer})}
    The natural transformation given in \ref{thm:main 2608152359 2} can be lifted to a natural transformation in the bounded derived category $\Dcat^b(A)$.
    \label{thm:main 2608152359 3}
\end{enumerate}
\end{theoremAlph}

Second, we consider the integrals of modules in probability space $(\II_A, \Sigma_{\II_A}, \mu_{\II_A})$.
Here, $\Sigma_{\II_A}$ is a $\sigma$-algebra and $\mu_{\II_A}$ is a measure such that $\mu_{\II_A}(\II_A)=1$.
We compute the integral of some special modules.

\begin{theoremAlph} \label{thm:main 2608160003}
Keep the notations from Theorem \ref{thm:main 2608152359}.
\begin{enumerate}[label={\rm(\arabic*)}]
  \item {\rm(Theorem \ref{thm:int-module iso})}
    If two modules $M$, $N \in {_A\modcat}$ are isomorphic, then the following matrix similarity holds:
    \[ \int_{\II_A} N  \sim \int_{\II_A} M. \]
    \label{thm:main 2608160003 1}
  \item {\rm(Proposition \ref{prop:int-dir sum} and Corollary \ref{coro:int-dir sum})}
    For a family left $A$-module $M_1$, $\ldots$, $M_n$, we have
    \[\int_{\II_A} \bigoplus_{i=1}^n M_i = \bigoplus_{i=1}^n \int_{\II_A} M_i.\]
    \label{thm:main 2608160003 2}
  \item {\rm(Proposition \ref{prop:int-proj mod})}
    Fixing a basis $\basis{A}=\{a_1,\ldots,a_{\dimA}\}$ of $A$ and a basis $\basis{Ae}=\{p_1,\ldots,p_r\}$ of a projective module $Ae$. Then the integral of this projective module is an $r\times r$ matrix
    \[\int_{\II_A\simeq[0,1]^{\times\dimA}} Ae =
    \frac{1}{2}
    \left(
    \begin{matrix}
      \sum\limits_{i=1}^{\dimA} \lambda_{i,j,k}
    \end{matrix}
    \right)_{k,j} \]
    where $\lambda_{i,j,k} \in \CC$ is obtained by $a_ip_j=\sum\limits_{k=1}^r\lambda_{i,j,k} p_k$.
    \label{thm:main 2608160003 3}
  \item {\rm(Proposition \ref{prop:int-inj mod})}
    With the notations of \ref{thm:main 2608160003 3}, the dual statement holds, i.e.,
    with respect to the dual basis $\{p_1^*,\ldots,p_r^*\}$ of $D(eA)$, we have
    \[ \int_{\II_A}D(eA) =
    \frac{1}{2}
    \left(
    \begin{matrix}
    \sum\limits_{i=1}^{\dimA}\mu_{i,j,k}
    \end{matrix}
    \right)_{j,k},\]
    where $\mu_{i,j,k}\in\CC$ is obtained from $p_j a_i=\sum\limits_{k=1}^r\mu_{i,j,k}p_k$.
    \label{thm:main 2608160003 4}
\end{enumerate}
\end{theoremAlph}

Third, we introduce integral profiles of modules, which are parameterised integrals.
Formally, the integral profile $\IntProf_{\II_A} M$ of a module $M$ is defined as a map in Part II, and it is not used in Part I.
Thus, its rigorous definition is therefore given in Part II, see Definition~\ref{def:int prof}.
For a module $M$, the integral profile of it used in this paper is denoted by $\int_{\II_A(\pmb{t})}M$,
where the parameter is a $\dimA$-dimensional vector $\pmb{t}$ that controls the integration domain $\II_A(\pmb{t})$.
One of the results of this paper is a classification of modules over finite-dimensional algebras by means of their integrals. The classification theorem is as follows.

\begin{theoremAlph}[{Theorem \ref{thm:param-int-module iso}}]
Two left $A$-modules $M$ and $N$ are isomorphic if and only if
\[\int_{\II_A(\pmb{t})}N  \sim \int_{\II_A(\pmb{t})}M\]
holds for all $\pmb{t}\in(\RR^{>0})^{\times\dimA}$. Here, the transition matrix is independent of $\pmb{t}$.
\end{theoremAlph}

\noindent
Therefore, The integral profile is a complete invariant of finite-dimensional modules. More precisely, a finite collection of its distinguished evaluations recovers not only the isomorphism class of every module but also all morphism spaces between modules.

Theorems \ref{thm:main 2608160003}, \ref{thm:main 2608160003 2}, \ref{thm:main 2608160003 3}, and \ref{thm:main 2608160003 4} can all be generalised to the setting of integral profiles; the proof only needs to use the fact that, for any fixed parameter $\pmb{t}$, the corresponding statements always hold.
See Proposition \ref{prop:param-int-dir sum}, Corollary \ref{coro:param-int-dir sum}, Proposition \ref{prop:param-int-proj mod}, and Proposition \ref{prop:param-int-inj mod}.
In particular, Propositions \ref{prop:param-int-proj mod} and \ref{prop:param-int-inj mod} provide integral characterisations of projective and injective modules, respectively. These characterisations are \textbf{necessary and sufficient}, see the following theorems.

\begin{theoremAlph}
Assume that $A=\CC\Q/\I$ is a finite-dimensional complex algebra.
\begin{enumerate}[label={\rm(\arabic*)}]
  \item {\rm(Theorem \ref{thm:mat char-proj mod})}
    Fix a basis of a given left $A$-module $M\in{_A\modcat}$. Then $M$ is projective if and only if
    \[ \IntProf_{\II_A} M
    \sim \frac{\mu_{\II_A(\pmb{t})}(\II_A(\pmb{t}))}{2} \cdot
    \bigoplus_{v\in\Q_0}
    \bigg(
      \sum_{i=1}^{\dimA}t_iL_i^{(v)}
    \bigg)^{\oplus m_v}, \]
    where all $m_v$ are non-negative integers such that $\sum\limits_{v\in\Q_0}p_vm_v=\dim_{\CC}(M)$ {\rm(}$p_v=\dim_{\CC}P(v)${\rm)}.
  \item {\rm(Theorem \ref{thm:mat char-inj mod})}
    Fix a basis of a given left $A$-module $M\in{_A\modcat}$. Then $M$ is injective if and only if
   \[
    \IntProf_{\II_A} M
    \sim \frac{\mu_{\II_A(\pmb{t})}(\II_A(\pmb{t}))}{2} \cdot
    \bigoplus_{v\in\Q_0}
    \bigg(
      \sum_{i=1}^{\dimA}t_iR_i^{(v)}
    \bigg)^{\oplus n_v},
    \]
    where all $n_v$ are non-negative integers such that $\sum\limits_{v\in\Q_0} c_vn_v = \dim_{\CC}(M)$ {\rm(}$c_v=\dim_{\CC}I(v)${\rm)}.
\end{enumerate}
\end{theoremAlph}

Forth, we introduce the integral coefficient ring $\pandiring(A)$ for each finite-dimensional algebra $A$,
whose elements are obtained by the integral profile of projective and injective modules.
The following theorem shows the Morita invariance of integral coefficient ring.

\begin{theoremAlph}[{Theorem \ref{thm:Morita inv icring}}]
As a complex algebra without an identity, the isomorphism class of the integral coefficient ring of a finite-dimensional complex algebra equipped with a compatible orthogonal fixed embedding system $($see Definition \ref{def:compatible orth emb sys}$)$ is a Morita invariant.
\end{theoremAlph}

We introduce the integral coefficient ring in order to study the (embedded) total integral profiles of projective resolutions and injective coresolutions of modules in the formal power series ring $\pandiring(A)[[x]]$, and thereby obtain a characterisation of homological dimensions.
Here, for each module $M$ and its minimal projective resolution $\pmb{P}^{\bullet}_M=(P^n_M, d^n_M)$,
we use $\IntProf_{\II_A} (\pmb{P}^{\bullet}_M, x)$ to represent the total projective integral profile of $M$
and use $\IntProfEmb_{\II_A} (\pmb{P}^{\bullet}_M, x)$ to represent the embedded total projective integral profile of $M$.
The following result provide a description of global dimension of algebra by using $\pandiring[[x]]$.

\begin{theoremAlph}[{Theorem \ref{thm:gldim-int}}] \label{thm:main 2608160903}
Let $A=\CC\Q/\I$ be a finite-dimensional complex algebra, and $\pandiring(A)$ be its integral coefficient ring.
The following statements are equivalent:
\begin{enumerate}[label={\rm(\arabic*)}]
  \item $\gldim A < \infty$;
  \item $\displaystyle \sup_{v\in\Q_0} \deg_x\IntProf_{\II_A}(\pmb{P}_{S_v}^\bullet;x) < \infty$;
  \item $\displaystyle \IntProf_{\II_A}(\pmb{P}_{\top(A)}^\bullet;x) \,\widetilde{\in}\, \pandiring(A)[x]$;
  \item $\displaystyle \sup_{v\in\Q_0} \deg_x\IntProf_{\II_A}(\pmb{I}_{\bullet}^{S_v};x) < \infty$;
  \item $\displaystyle \IntProf_{\II_A}(\pmb{I}_{\bullet}^{\top(A)};x) \,\widetilde{\in}\, \pandiring(A)[x]$.
\end{enumerate}
Here, $\,\widetilde{\in}\,$ presents an element lies in a set up to a canonical embedding.
\end{theoremAlph}

\noindent
Obviously, Theorem \ref{thm:main 2608160903} shows that $\gldim A = \infty$ if and only if one of following statements hold:
\begin{itemize}
  \item $\displaystyle \IntProf_{\II_A}(\pmb{P}_{\top(A)}^\bullet;x) \,\widetilde{\in}\, \pandiring(A)[[x]]\backslash\pandiring(A)[x]$;
  \item $\displaystyle \IntProf_{\II_A}(\pmb{I}_{\bullet}^{\top(A)};x) \,\widetilde{\in}\, \pandiring(A)[[x]]\backslash\pandiring(A)[x]$.
\end{itemize}

\noindent
The following result provide a description of finitistic dimension of algebra by using $\pandiring[[x]]$.

\begin{theoremAlph}[{Theorem \ref{thm:findim-int}}]
The following statements are equivalent:
\begin{enumerate}[label={\rm(\arabic*)}]
  \item $\findim A<\infty$;
  \item there is $d\in\NN$ such that, for every $0\ne M\in{_A\modcat}$,
  \[ \IntProf_{\II_A} (\pmb{P}_M^\bullet;x) \,\widetilde{\in}\, \pandiring(A)[x]
   \text{ implies } \deg_x \IntProf_{\II_A} (\pmb{P}_M^\bullet;x) \=< d; \]
  \item there is $d\in\NN$ such that, for every $0\ne M\in{_A\modcat}$ whose embedded total projective integral profile
  is a polynomial, one has
  \[\IntProfEmb_{\II_A}P_M^n = 0 \quad \text{for every } n>d.\]
\end{enumerate}
\end{theoremAlph}

In Subsection \ref{subsect:Suffi condi for fdim}, we provide a theorem which gives two sufficient conditions,
such that if a finite-dimensional complex algebra $A$ satisfies one of them, then its finitistic dimension $\findim A$ is finite.
See the following result.

\begin{theoremAlph}[{Theorem \ref{thm:diff-findim}}]
Let $A$ be a nonzero finite-dimensional complex algebra.
Suppose that there are integers $k,s \>= 0$ and polynomials $p_0(z),\ldots, p_k(z)\in\CC[z]$ $(p_k\ne0)$,
such that one of the following conditions holds for every $0\ne M\in{_A\modcat}$:
\begin{enumerate}[label={\rm(\arabic*)}]
  \item $\displaystyle R_M(x) := L_{p_0,p_1,\ldots, p_k} \IntProfEmb_{\II_A}(\pmb P_M^\bullet;x) \in\pandiring(A)[x]$, and $\deg_xR_M(x) \=< s$,
  where $L_{p_0,p_1,\ldots, p_k}$ is the differential operator $\sum\limits_{j=0}^{k}x^jp_j(x\frac{\dd}{\dd x})$;
\label{eq:diff findim 1}

  \item $\displaystyle R_M^{\mathrm{lin}}(x)
 := L_{p_0,p_1,\ldots, p_k} \bigg(
 (\mathcal L) \IntProfEmb_{\II_A} (\pmb P_M^\bullet;x) \bigg)
 \in\pandiring(A)[x] \text{ and }
 \deg_xR_M^{\mathrm{lin}}(x)\=<s.$
\label{eq:diff findim 2}
\end{enumerate}
Here, $k$, $s$ and $p_0,\ldots,p_k$ are independent of $M$.
Then \[\findim A\=< \max (\{0,s-k\}\cup \{n\in\ZZ_{\>= 0}:p_k(n)=0\} ) <\infty. \]
\end{theoremAlph}

In particular, we provide an example for finite-dimensional complex algebra $A$ with infinite global dimension such that:
$\rad^4(A)=0\ne\rad^3(A)$; $A$ is not self-injective; $A$ is syzygy-infinite; and $A$ is non-monomial.

The paper is organized as follows.

\paragraph{Part I.} We introduce the integrals of modules, and compute the integrals of special modules in this part.
Since the computation of these integrals depends on the integration domain, we first consider integrals over a fixed domain.
This gives concrete integral values for modules, which are then used to define the parameterized integral over varying domains.
There are three subsections.
\begin{itemize}
  \item In Section~\ref{sect:cent int}, we study the (average) central integral of the regular module.
  \item In Section~\ref{sect:int-mod-prob sp}, we fix the integration domain to be a probability space, so that the total measure is normalized to $1$.
  \item In Section~\ref{sect:int-mod-meas sp}, we introduce the parameterized integral of a module, which is a generalisation of the notion of indefinite integral.
\end{itemize}

\paragraph{Part II.}
In this part, we apply the parameterised integral developed in Part~I to homological dimensions of finite-dimensional algebras. We set the integration domain to be $\II_A(\pmb{t})$, where $\pmb{t}$ is a parameter, and establish intrinsic connections between homological dimensions and integrals of modules.

\begin{itemize}
  \item In Section~\ref{sect:int coef ring}, we construct the integral coefficient ring $\pandiring(A)$ for a finite-dimension complex algebra, and we introduce the embedded total integral profile of a non-negative chain complex, which lies in the formal power series ring $\pandiring(A)[[x]]$.
  \item In Section~\ref{sect:resol}, we give equivalent characterisations of projective and injective modules in terms of their integral profiles (see Theorems~\ref{thm:mat char-proj mod} and~\ref{thm:mat char-inj mod}).
      For an arbitrary module $M$, its minimal projective and injective resolutions yield the total projective and total injective integral profiles.
      We prove that the polynomiality of these series detects the finiteness of the corresponding homological dimensions, and that the degree equals the dimension (Lemmas~\ref{lemm:pdim-int 0812} and~\ref{lemm:idim-int 0812}).
  \item In Section~\ref{sect:homodim intdescrib}, we use the integral coefficient ring $\pandiring(A)$ to characterise the global dimension and the little finitistic dimension of $A$, see Theorems \ref{thm:gldim-int} and \ref{thm:findim-int}).
\end{itemize}

\paragraph{Acknowledgements}
The author would like to thank Hanpeng Gao, Dajun Liu, Xin Ma, and Houjun Zhang for helpful discussions and to express our gratitude to the Guizhou Provincial Key Laboratory of Applied Mathematics and Computing Power \& Algorithms for providing us with an academic discussion laboratory.

\section*{\centering Part I: Integrals of algebras and modules}
\addcontentsline{toc}{section}{Part I: Integrals of algebras and modules}
\label{part:I}

\textsl{In this part, we develop the integral theory for finite-dimensional modules over a complex algebra by interpreting each module through its structure homomorphism. Starting from the special Banach module categories, we define the integrals of modules, compute it explicitly for projective and injective modules, and then introduce the parameterized integral that will serve as the key tool for the homological applications in Part~II.}

\section{Preliminaries: special Banach module categories}

\textsl{The special Banach module category $\scrA^p$ $(p \>= 1)$ was first introduced by Leinster in \cite{Lei2023FA}, with the aim of providing a categorical interpretation of the Lebesgue integral by algebraic methods. Subsequent work in \cite{LLHZ2025, Liu2025pre} extended Leinster's works and characterised more general integrals. In this section, we review some fundamental concepts in algebra and representation theory \cite{ASS2006}, as well as the main results of \cite{Lei2023FA, LLHZ2025, Liu2025pre}.}

\subsection{Algebras and quiver representations}

Let $\FF$ be a field.

\begin{definition}\rm
An \defines{$\FF$-linear algebra $\itLamb$} (or an \defines{$\FF$-algebra} for short) over a field $\FF$ is an $\FF$-vector space with a multiplication
\[\cdot: \itLamb \times \itLamb \to \itLamb, (a, b) \mapsto ab\]
such that $k(ab)=(ka)b=a(kb)$ holds for all $a, b \in \itLamb$ and $k\in \FF$.
For simplicity, all $\FF$-linear algebras in this paper are called algebras.
\end{definition}
In \cite{Gab1972,Gab1973,Gab2006Repr}, Gabriel has shown that each $\FF$-algebra $\itLamb$ is Morita equivalent to a \defines{bound quiver algebra} $\FF\Q/\I$, where $\Q$ is a quiver (i.e., a digraph), $\FF\Q$ is an algebra whose basis is the set of all paths in $\Q$,
the multiplication defined on $\FF\Q$ is the composition of paths, and $\FF\Q/\I$ is the quotient space obtained by $\FF\Q$ modulo an ideal $\I$ of $\FF\Q$. We call $\Q$ is the \defines{quiver} of $\itLamb$.
In representation theory, $\Q$ can be viewed as a quadruple $(\Q_0,\Q_1,\fcts,\fctt)$
where $\Q_0$ is the set of all vertices, $\Q_1$ is the set of all arrows,
and $\fcts$ and $\fctt$ are functions $\Q_1\to \Q_0$
sending each arrow to its starting point and ending point, respectively.
Therefore, $\Q$ can be written as $(v\To{a} w)_{v,w\in\Q_0, a\in\Q_1}$.

\begin{definition}\rm
Let $\itLamb$ be a finite-dimensional $\FF$-algebra.
\begin{itemize}
  \item[(1)] A \defines{left $\itLamb$-module} $V$ is a linear space with a left $\itLamb$-action
\[ \itLamb \times V \to V, (\lambda,v) \mapsto \lambda.v \]
such that $h: \itLamb \to \End_{\FF}(V)$, $\lambda \mapsto (h_{\lambda}: v\mapsto \lambda.v)$ is a homomorphism of algebras.
  \item[(2)] Dually, a \defines{right $\itLamb$-module} $V$ is a linear space with a left $\itLamb$-action
\[ V \times \itLamb \to V, (v,\lambda) \mapsto v.\lambda \]
such that $h: \itLamb \to (\End_{\FF}(V))^{\op}$, $\lambda \mapsto (h_{\lambda}: v\mapsto v.\lambda)$ is a homomorphism of algebras.
\end{itemize}
\end{definition}

We write the left (resp., right) multiplication of $a \in \itLamb$ on $m \in M$ as $a.m$ (resp., $m.a$) to distinguish it from other products that may appear in the same formula. When no confusion is possible, it may also be written as $a\cdot m$ (resp., $m\cdot a$) or simply $am$ (resp., $ma$). The choice of notation depends on personal preference and varies with the context.
In representation theory, left $\itLamb$-modules and right $\itLamb$-modules are also called left-representations and right-representations of $\itLamb$, respectively.

Notice that $\itLamb=\FF\Q/\I$ has an identity $1 = \displaystyle\sum_{i\in\Q_0}\e_i$
($\e_i$ is the path of length zero corresponding to the vertex $i\in\Q_0$),
then, for the case of $V$ to be a left $\itLamb$-module, we have
$V = 1.V = \sum\limits_{i\in\Q_0}\e_iV = \bigoplus\limits_{i\in\Q_0}\e_iV$
which provides a quiver left-representation
\[ ( \e_iV \oT{a} \e_jV )_{i,j\in\Q_0, a\in \Q_1} \]
since each arrow $a\in\Q_1$ with $\fcts(a)=i$ and $\fctt(a)=j$
corresponds to a linear map $a: \e_jV\to \e_iV$,
$\e_jx\mapsto a\cdot \e_jx = ax = \e_i\cdot ax$ ($ax\in V$).
Conversely, for each quiver left-representation
\[ ( V_i \oT{\varphi_a} V_j )_{i,j\in\Q_0, a\in \Q_1} \]
($V_i$, $V_j$ are linear spaces and $\varphi_a$ is a linear map),
we get a left $\itLamb$-module $M:=\bigoplus\limits_{i\in\Q_0} V_i$ whose left $\itLamb$-action is
the map $\itLamb \times M \to M$ naturally induced by
$a.v_i := \varphi_a(v_i)$ for all $a\in\Q_1$ ($\fcts(a)=i$) and $v_i\in V_i$.
Therefore, we often treat quiver left-representations and left modules as the same in representation theory.
Dually, in the case for $V$ to be a right $\itLamb$-module, we obtain a quiver right-representation
\[ ( \e_iV \To{a} \e_jV )_{i,j\in\Q_0, a\in \Q_1} \]
in a similar way.

\begin{definition}\rm
Let $V$ and $W$ be two left (resp., right) $\itLamb$-modules.
Then a \defines{left {\rm(}resp., right{\rm)} $\itLamb$-homomorphism} $h:V\to W$ from $V$ to $W$ is a linear map such that
\[ h(\lambda.v)=\lambda.h(v)  \quad (\text{resp.}, ~ h(v.\lambda)=h(v).\lambda ) \]
holds for all $\lambda\in\itLamb$ and $v\in V$.
\end{definition}

Assume $A=\FF\Q/\I$ is a bound quiver algebra. Considering the quiver left-representation
\[(\e_iV \oT{a} \e_jV )_{i,j\in\Q_0, a\in \Q_1}\]
given by a left $\itLamb$-module $V$
and the quiver left-representation
\[(\e_iW \oT{a} \e_jW )_{i,j\in\Q_0, a\in \Q_1}\]
given by a left $\itLamb$-module $W$,
we get that each element $v\in V$ has a decomposition $v=\sum\limits_{i\in\Q_0}v_i=(v_i)_{i\in\Q_0}$,
each element $w\in W$ has a decomposition $w=\sum\limits_{i\in\Q_0}w_i=(w_i)_{i\in\Q_0}$,
and $h$ can be restricted on each $\e_iV$ which induces a family of linear maps $(h_i:=h|_{\e_iV}: \e_iV\to \e_iW)_{i\in\Q_0}$.
Then $h(v) = \sum\limits_{i\in\Q_0}h_i(v_i)$, and for each arrow $a\in \Q_1$ with $\fcts(a)=i$ and $\fctt(a)=j$,
we have $0\ne av_j \in \e_iV$ holds for all $v_j\in \e_jV$ and have
\[ h_i(a.v_j) = h(a.v_j) = a.h(v_j) = a.h_j(v_j), \]
i.e., the following diagram
\[\xymatrix{
  \e_iV \ar@{<-}[r]^{a} \ar[d]_{h_i} & \e_jV \ar[d]^{h_j} \\
  \e_iW \ar@{<-}[r]^{a} & \e_jW
}\]
commutes.
Dually, in the case of $V$ and $W$ to be two right $\itLamb$-module,
each right $\itLamb$-homomorphism $h:V\to W$ provides the following diagram
\[\xymatrix{
  \e_iV \ar@{->}[r]^{a} \ar[d]_{h_i} & \e_jV \ar[d]^{h_j} \\
  \e_iW \ar@{->}[r]^{a} & \e_jW
}\]
commutes.

\begin{definition}\rm
An \defines{$(A,B)$-bimodule} $M$ is both a left $A$-module and a right $B$-module such that $amb=(am)b=a(mb)$,
and an \defines{$(A,B)$-homomorphism} $h: M \to N$ of $(A,B)$-bimodules is both a left $A$-homomorphism and a right $B$-homomorphism.
Moreover, we call a left $A$- (resp., right $B$-, $(A,B)$-) homomorphism $h$ is a \defines{monomorphism} if it is an injection,
and we call $h$ is an \defines{epimorphism} if it is a surjection.
An left $A$- (resp., right $B$-, $(A,B)$-) isomorphism of modules is both a monomorphism and an epimorphism.
\end{definition}

\subsection{Normed algebras and normed modules}

Let $\FF$ be a field equipped with an absolute value $|\cdot|:=|\cdot|_{\FF}: \FF\to\RR^{\>=0}$.
Here, for any $x,y\in\FF$, $\|x\|=0$ if and only if $x=0$; $|xy|=|x||y|$; and $|x+y|\=<|x|+|y|$.
A \defines{seminorm} defined on an $\FF$-linear space $A$ is a map $\|\cdot\|_A:A\to \RR^{\>=0}$ such that
\begin{enumerate}[label=(SN\arabic*)]
  \item $\| a\|_A \>=0$ \label{seminrom1};
  \item $\| \lambda a\|_A = |\lambda|\; \| a\|_A$ ($\forall \lambda \in\FF, a\in A$) \label{seminrom2};
  \item $\| a_1+a_2\|_A \=< \| a_1\|_A + \| a_2 \|_A$ ($\forall a_1,a_2\in A$) \label{seminrom3}.
\end{enumerate}
Furthermore, we call $\|\cdot\|_A:A\to \RR^{\>=0}$ is a \defines{norm} if:
\begin{enumerate}[label=(N\arabic*)]
  \item it satisfies \ref{seminrom1}--\ref{seminrom3};
  \item and, for an element $a\in A$, $\|a\|_A=0$ and $a=0$ coincide.
\end{enumerate}
In this case, we call $A$ a \defines{normed space}.
An important situation is that $A$ is complete, i.e., each Cauchy sequence in $A$ is convergent.
In this case, we call  $A$ a \defines{Banach space}.

\begin{definition} \rm
If a normed (resp., Banach) space $A$ as above is also an $\FF$-algebra
such that $\|a_1a_2\|_A \=< \|a_1\|_A \|a_2\|_A$,
then we call it a \defines{normed $($resp., Banach$)$ $\FF$-algebra}.
\end{definition}

Now, take two finite-dimensional normed $\FF$-algebras $A$ and $B$.

\begin{definition} \rm
If an $\FF$-linear space $V$ is both a left $A$-module and a right $B$-module,
and for each $v\in V$, $a\in A$ and $b\in B$, the equation $(a.v).b = a.(v.b)$ holds,
then we call that $V$ is an \defines{$(A,B)$-bimodule}.
\end{definition}

For simplicity, \textbf{the fields considered in this article are all complex fields},
i.e., we take $\FF=\CC$ in this paper,
so we can omit the step of defining topological structures on general algebraic closed fields.
Then all $\FF$-algebras are $\CC$-algebras (or say complex algebras) in this assumption.

\begin{lemma} \label{lemm:C algs are Banach}
Let $A$ be a finite-dimensional normed complex algebra.
Then the finite dimensional property of it implies its completeness,
i.e., $A$ is a Banach algebra.
\end{lemma}

\begin{proof}
Let $\dimA = \dim_{\CC} A$ and choose a linear basis $\{e_1,\dots,e_n\}$ of $A$.
We first prove that the norm $\|\cdot\|_A$ defined on $A$
is equivalent to the standard $\ell^1$-norm $\|\cdot\|_1: A\to\RR^{\>=0}$,
$x=\sum\limits_{i=1}^n x_i e_i \mapsto \sum\limits_{i=1}^n|x_i|$ on the coordinate space
\footnote{Fix the norms on the underlying finite-dimensional vector spaces.
Since all norms in finite dimension are equivalent,
the particular choice affects numerical constants but not completeness or density.}.
Clearly $(A,\|\cdot\|_1)$ is isometrically isomorphic to $(\CC^{\dimA},\|\cdot\|_1)$, hence complete.
For any $x=\sum\limits_{i=1}^n x_i e_i \in A$, we have
\[ \|x\|_A
  \=< \sum_{i=1}^n |x_i|\,\|e_i\|_A
  \=< \|x\|_1\cdot \max_{1\=< i\=< n}\|e_i\|_A =: \|x\|_1 C_1 . \]
So $\|\cdot\|_A\=< C_1\|\cdot\|_1$ holds for some $C_1\in\RR^{\>=0}$.
For the reverse inequality, consider the unit sphere $\mathbb{S}^1:=\{x\in A:\|x\|_1=1\}$.
Since $(A,\|\cdot\|_1)\cong \mathbb C^n$, $\mathbb{S}^1$ is compact.
The function $f(x):=\|x\|_A$ is continuous on $\mathbb{S}^1$
because $\|x-y\|_A \=< C_1\|x-y\|_1$. Hence $f$ attains a minimum,
say $m:=\min\limits_{x\in \mathbb{S}^1}\|x\|_A$.
If $m=0$, then some $x_0\in \mathbb{S}^1$ satisfies $\|x_0\|_A=0$, which forces $x_0=0$ by positive definiteness,
contradicting $\|x_0\|_1=1$. Thus $m>0$. Now for any nonzero $x\in A$,
the vector $\frac{x}{\|x\|_1}$ lies in $\mathbb{S}^1$,
so $\left\|\frac{x}{\|x\|_1}\right\|_A\>= m$ implies $\|x\|_A\>= m\|x\|_1$.
Let $C_2=\frac{1}{m}$, then $\|x\|_1 \>= C_2\|x\|_A$. Hence $\|\cdot\|_A$ and $\|\cdot\|_1$ are equivalent.

Next, we show that $A$, as a normed algebra $(A,\|\cdot\|_A)$, is a complete algebra.
Let $\{x_k\}_{k\in\NN^+}$ be a Cauchy sequence in $A$ up to the norm $\|\cdot\|_A$.
From $\|x-y\|_1 \=< C_2\|x-y\|_A$, it follows that $\{x_k\}_{k\in\NN^+}$ is also a Cauchy sequence up to the norm $\|\cdot\|_1$.
Since $\|\cdot\|_1$ is complete, there exists $x\in A$ such that
$\lim\limits_{k\to+\infty}\|x_k-x\|_1 = 0$.
Then using $\|x_k-x\| \=< C_1\|x_k-x\|_1$,
we get $\lim\limits_{k\to+\infty}\|x_k-x\|= 0$.
Thus every Cauchy sequence converges, so $A$, as a normed algebra $(A,\|\cdot\|_A)$, is complete as required.

Finally, we show that the submultiplicativity of the norm already implies the continuity of the multiplication.
Indeed, if $x_k\to x$ and $y_k\to y$ in $\|\cdot\|_A$, then
\[ \|x_k y_k - xy\|_A
 = \|(x_k-x)y_k + x(y_k-y)\|_A
\=< \|x_k-x\|_A\|y_k\|_A + \|x\|_A\|y_k-y\|_A. \]
Since $\{y_k\}$ is bounded (convergent sequences are bounded),
the right-hand side tends to $0$. Hence multiplication is continuous.
Combining completeness and continuity (with the given submultiplicative norm), $A$ is a Banach algebra. This completes the proof.
\end{proof}

Put $A$ and $B$ are two finite-dimensional complex algebras.
Then their finite dimensional properties imply their completeness,
i.e., $A$ and $B$ are Banach algebras by Lemma \ref{lemm:C algs are Banach}.
Assume that there is homomorphism $\homo:A\to B$ of algebras.
Then each norm $\|\cdot\|_B$ defined on $B$ implies a seminorm $\Verb\cdot\Verb$ defined on $A$
which sends each $a\in A$ to $\Verb a\Verb_A := \|\homo(a)\|_B\in\RR^{\>=0}$.

\begin{remark} \rm \
\begin{enumerate}[label=(\arabic*)]
  \item Sometimes we need that $\homo:A\to B$ satisfies \textbf{Central-Image Condition} ($=$ CIC for short), i.e., $\Ima(\homo)\subseteq\cent{B}$. When CIC is needed, we will provide a specific instruction.
    CIC ensures that every right $B$-module admits a natural $(A,B)$-bimodule structure via $a\cdot v:=v\cdot\homo(a)$.
    If an $(A,B)$-bimodule structure is already available by other means, then CIC is not necessary.
  \item Let $A$ be a finite-dimensional $\FF$-algebra. For any left $A$-module $M$, it corresponds to a homomorphism $\rho_M: A\to B:=\End_{\FF}(M), a\mapsto (m\mapsto am)$.
    In this case, $B$ is a $(A,B)$-bimodule whose left $A$-action and right $B$-action is:
    \[ A\times B\times B \to B, ~ (a,x,b) \mapsto \rho_M(a)xb. \]
\end{enumerate}
\end{remark}

\begin{definition}\rm
Assume that $A$ is a finite-dimensional complex algebra with a seminorm $\Verb\cdot\Verb_A$
and $B$ is a finite-dimensional Banach complex algebra.
An $(A,B)$-bimodule $V$ equipped with a norm $\|\cdot\|_V$ is called a \defines{normed module}
if the two module actions are continuous and, after an equivalent renorming if necessary,
satisfy $\|avb\|\=< \Verb a\Verb_A\|v\|_V\|b\|_B$. Furthermore,
\begin{itemize}
  \item[(1)] we call that it is a \defines{$\homo$-normed module} if
    $\Verb\cdot\Verb_A$ is induced by the norm $\|\cdot\|_B$ defined on $B$;
  \item[(2)] and we call that it is a \defines{Banach module} if it is complete.
\end{itemize}
\end{definition}

Let $(\Q_B,\I_B)$ be the bound quiver of $B$, then $B$ has a $\CC$-linear representation
\[ B = \FF\Q_B/\I_B = \bigoplus_{l\in\NN}\bigoplus_{\wp\in(\Q_B)_l} (\CC\wp+\I_B) \]
where $(\Q_B)_l$ is the set of all paths in $(\Q_B,\I_B)$ of length $l$.
It follows that $B$ has a basis $\basis{B}=\{\wp+\I: \wp\in (\Q_B)_l\}$.
Since $B$ is finite-dimensional, we assume $\dim_{\CC} B=\dimB$ and $\basis{B}=\{b_i: 1\=< i\=< \dimB\}$.
Then any $b_i$ can be equipped with a weight $\norm_{B,i}$ given by
\begin{center}
$\displaystyle \norm_B: \bigcup\limits_{l\in\NN}(\Q_B)_{l} \to \RR^{\>=0}, b_i\mapsto\norm_{B,i}:=\norm_B(b_i)$,
\end{center}
and in this case, the following map
\begin{align}\label{eq:B norm}
\|\cdot\|_B: B \to \RR^{\>=0}, ~ b=\sum_{i=1}^{\dimB} c_ib_i \mapsto
 \left( \sum_{i=1}^{\dimB}
   |c_i|^p\norm_{B,i}^p
 \right)^{\frac{1}{p}}
\end{align}
is a norm which implies that $B$ is a Banach complex algebra, see \cite[Section 3.1]{Liu2025pre}.

\begin{remark}\rm
Note that a left $A$-module $M$ can be seen as a homomorphism $\rho_M: A\to \End_{\FF}(M)$ of algebras,
but CIC does not hold in general, except the case of $M$ is a direct sum of some one-dimensional left $A$-modules.
But $M$ as a left $\End_{\FF}(M)$-module may be not a direct sum of some one-dimensional left $\End_{\FF}(M)$-modules.
\end{remark}

\begin{assumption}\rm
In this paper, the norm defined on $B$ we considered is always of the form given in \eqref{eq:B norm},
and, in the existence of $\homo:A\to B$, the seminorm $\Verb\cdot\Verb_A$ defined on $A$ we considered is always induced by the norm $\|\cdot\|_B$ as above.
\end{assumption}

\begin{proposition} \label{prop:normed bimod 260809}
If CIC holds, then, under the above assumptions, each right $B$-module $V$ is also an $(A,B)$-bimodule,
and it equipped with a norm $\|\cdot\|_V$ is a $\homo$-normed $(A,B)$-bimodule.
\end{proposition}

\begin{proof}
First of all, we show that $V$ is an $(A,B)$-bimodule.
The left $A$-action $A\times V \to V$, $(a,v)\mapsto a.v := v\homo(a)$ is well-defined via $\homo: A \to B$.
We write the right $B$-action as $V\times B\to V$, $(v,b)\mapsto v\cdot b$.
To verify the bimodule associativity, we compute
\[ (a.v)\cdot b = (v\homo(a))\cdot b = v\cdot (\homo(a)b) = v\cdot (b\homo(a)) = (v\cdot b)\cdot \homo(a) = a.(v\cdot b) \]
by $\homo(A) \subseteq \cent{B}$, Thus, $V$ is an $(A,B)$-bimodule.

Next, since $B$ and $V$ are finite-dimensional, the right $B$-action $V\times B\to V$, $(v,b)\mapsto v\cdot b$ is continuous and hence bounded.
Thus, there exists a constant $C\in\RR^{>0}$ such that $\|v\cdot b\|_V \=< C\|v\|_V\|b\|_B$ for all $v\in V$ and $b\in B$.
Define a new norm $\|\cdot\|_V'$ on $V$ by
\[ \|v\|_V' := \sup_{\|b\|_B\=< 1}\|v\cdot b\|_V.\]
The boundedness of the right $B$-action implies $\|v\|_V'\=< C\|v\|_V$.
On the other hand, since $V$ is a unital right $B$-module, $v=v\cdot1_B$.
Hence
\[ \|v\|_V' \>= \left\|v\cdot\frac{1_B}{\|1_B\|_B}\right\|_V = \frac{1}{\|1_B\|_B}\|v\|_V.\]
Therefore, $\|\cdot\|_V'$ is a norm equivalent to
$\|\cdot\|_V$.

We next show that the right $B$-action is contractive with respect
to $\|\cdot\|_V'$. Let $b\in B$. If $b=0$, the assertion is clear.
If $b\neq0$, then
\begin{align*}
   \|v\cdot b\|_V'
& = \sup_{\|c\|_B\=< 1} \|(v\cdot b)\cdot c\|_V\\
& = \sup_{\|c\|_B\=< 1} \|v\cdot(bc)\|_V \\
& \=< \|b\|_B \sup_{\|c\|_B\=< 1} \left\| v\cdot\frac{bc}{\|b\|_B} \right\|_V\\
& \mathop{\=<}\limits^{\spadesuit} \|b\|_B\|v\|_V',
\end{align*}
where $\spadesuit$ holds since $\left\|\frac{bc}{\|b\|_B}\right\|_B \=< \|c\|_B \=< 1.$

Finally, for any $a\in A$, $v\in V$, and $b\in B$, we have
\begin{align*}
     \|a. v\cdot b\|_V'
& =  \|(v\cdot\homo(a))\cdot b\|_V'\\
&\=< \|b\|_B\|v\cdot\homo(a)\|_V'\\
&\=< \|b\|_B\|\homo(a)\|_B\|v\|_V'\\
& =  \Verb a\Verb_A\|v\|_V'\|b\|_B.
\end{align*}
This proves the required inequality, and continuity follows immediately.
Since $\|\cdot\|_V'$ is equivalent to $\|\cdot\|_V$, the module $V$ is a $\homo$-normed $(A,B)$-bimodule after an equivalent renorming.
\end{proof}

\subsection{Measure space, juxtapositions} \label{subsect:measure space}

Assume $A=\FF\Q_A/\I_A = \bigoplus\limits_{\wp}(\CC\wp+\I_A)$. Then $A$ has a natural basis $\basis{A}$ whose elements are path in the bound quiver $(\Q_A,\I_A)$ of $A$. Since the dimension $\dimA$ of $A$ is finite, then we can write $\basis{A}$ as $\{ a_i : 1\=< i\=< \dimA \}$.
We define
\[\II_A := \prod_{i=1}^{\dimA} \{\lambda_i a_i : \lambda_i\in \II:=[r,s]\}
  \mathop{\longleftrightarrow}\limits^{\text{1-1}} [r,s]^{\times \dimA}
  \ \  (r \=< s) \]
then each function $f: \II_A\to B$ defined on $\II_A$ can be seen as a function $\tilde{f}: A\to B$,
where $\tilde{f}(x)=f(x)$ if $x\in\II_A$ ($\subseteq A$) and $\tilde{f}(x)=0$ if $x\notin\II_A$.

Now we equip $\II_A$ with a measure. Let $(\CC, \Sigma_{\CC}, \mu_{\CC})$ be an arbitrary measure space,
where $\Sigma_{\CC}$ is a $\sigma$-algebra of subsets of $\CC$ and
$\mu_{\CC}: \Sigma_{\CC} \to \RR^{\>= 0}$ is a countably additive map
(For instance, one may take $\mu_{\CC}$ to be the Lebesgue measure on $\CC$,
but our construction does not depend on this choice).
We can define a \defines{relative measure} $\mu_{\II_A}$ on $\II_A$ by pulling back the product of $\mu_{\CC}$ such that
\[ \mu_{\II_A}\bigg(\prod_{i=1}^{\dimA}[r_i,s_i]\bigg) = \prod_{i=1}^{\dimA} \mu_{\CC}([r_i,s_i]) \]
holds for all choices of intervals $[r_i, s_i] \subseteq [r,s]$.
This definition extends uniquely to a countably additive measure on the $\sigma$-algebra
generated by all cubes $\prod\limits_{i=1}^{\dimA} [r_i,s_i]$,
making $(\II_A, \Sigma_{\II_A}, \mu_{\II_A})$ a measure space.
If the ambient measure $\mu_{\CC}$ is chosen to be the Lebesgue measure,
then $\mu_{\II_A}$ is the corresponding product Hausdorff measure.
Here, for a comprehensive treatment of Hausdorff measures, we refer to \cite{Rogers1999}.
However, we emphasise that the subsequent arguments do not rely on this particular choice.
Any measure $\mu_{\CC}$ with countable additivity suffices.

Fix an integer $2^{\dimA}\>= 2$ and a nested sequence of finite measurable partitions
\[ \mathcal{P}_0 := \{\II_A\} \subset \mathcal{P}_1 \subset \mathcal{P}_2 \subset \cdots \]
such that every member of $\mathcal{P}_u$ is the disjoint union of $2^{\dimA}$ members of $\mathcal{P}_{u+1}$,
and for simplicity, we assume that the $2^{\dimA}$ cells have equal measure.
For any $1\=< p \in \RR$, the Bochner integral
\[ \| f \| := \bigg( (\Bochner)\int_{\II_A} \|f(x)\|_{B}^p \dd\mu_{\II_A} \bigg)^{\frac{1}{p}} \]
defines the norm for each function $f:\II_A\to B$ if $(\Bochner)\displaystyle \int_{\II_A} \|f(x)\|_{B}^p \dd\mu_{\II_A}$ is convergent. We call the following space
\[ L^p(\II_A;B) := \bigg\{ f : \|f\|^p = (\Bochner)\int_{\II_A} \|f(x)\|_{B}^p \dd\mu_{\II_A} < +\infty \bigg\} \]
is an \defines{$L^p$-space} in analysis.
Here, \textbf{we identify functions that are equal almost everywhere}.

Assume $\mathcal{P}_u = \{C_{u,1},\ldots, C_{u,N_u}\}$, and define
\[ E_u = \{ f: \II_A \to B : f \text{ is constant on every } C_{u,t} \in \mathcal{P}_u \}. \]
Then every $f\in E_u$, call it a \defines{step function}, has a unique expression
\[ f = \sum_{i=1}^{N_u} b_i \id_{C_{u,i}}, \qquad b_1, b_2,\ldots, b_{N_u}\in B, \]
after null cells are removed. Here, for each subset $S\subseteq \II_A$,
$\id_S$ is the function $\II_A$ sends each $x\in S$ to the identity $1_B$ in $B$
and sends each $x\in S\backslash \II_A$ to the zero element $0_B$ in $B$.
Then we have:
\begin{itemize}
  \item[(1)] each $E_u$ is a $\homo$-normed $(A,B)$-bimodule whose norm can be induced by using the norm defined on $E_{u-1}$,
    and in this situation, unlike Proposition \ref{prop:normed bimod 260809}, CIC is not necessary;
  \item[(2)] the formulas $N_u = 2^{u\dimA}$, $E_0 \cong B$, and $E_{u+1} \cong E_u^{\oplus 2^{\dimA}} \cong B^{\oplus N_{u+1}} = B^{\oplus 2^{(u+1)\dimA}}$,
    where the isomorphism $\gamma_u: E_u^{\oplus 2^{\dimA}} \mathop{\to}\limits^{\cong} E_{u+1}$
    is called a \defines{juxtaposition} in \cite{Lei2023FA, LLHZ2025, Liu2025pre};
  \item[(3)] and $B\cong E_0 \To{\subseteq} E_1 \To{\subseteq} E_2 \To{\subseteq} \cdots \To{\subseteq} E_u \To{\subseteq} \cdots$.
\end{itemize}
See \cite[Lemmas 4.11 and 4.13]{Liu2025pre}.
For $u\leq v$, the inclusion maps $E_u\=< E_v$ form a direct system of normed $(A,B)$-bimodules.
Its direct limit taken in the category ${_A\mathsf{Ban}_B}$ of Banach $(A,B)$-bimodules satisfies
\begin{align}\label{eq:Lp}
  \varinjlim_{{}_A\mathsf{Ban}_B} E_u \cong \widehat{\mathbf S_{\homo}(\II_A)},
\end{align}

where $\bfS_{\homo}(\II_A)$ is the set of all step functions on $\II_A$, and $\w{\bfS_{\homo}(\II_A)}$ denotes its completion.
To some extent, $\bfS_{\homo}(\II_A)$ may be regarded as a categorical analogue of an $L^p$-space in the sense of $\scrA^p_{\homo}$.
We write $\w{\bfS_{\homo}(\II_A)}$ as $L^p(\II_A;B)$ in the paper
\footnote{The space $L^p(\II_A;B)$ is not, in general, a classical $L^p$-space; rather, it is an $L^p$-type construction for finite-dimensional algebras.
This is the reason why we retain the notation $L^p(\II_A;B)$ throughout this paper.
In the special case $A=B=\RR^{\times n}$, which are basic and semisimple, $\homo=\ident_{\RR^{\times n}}$,
and $\II_A=[0,1]^{\times n}$ equipped with the Lebesgue measure, $L^p(\II_A;B)$ reduces to the classical $L^p$-space.
In particular, when $p=1$, $A=B=\RR$ and $\homo=\ident_A$,
$L^1(\II_A;A)$ $\big(= L^1([0,1];\RR) := \w{\bfS_{\ident_{\RR}}([0,1])}\big)$ coincides with $L^1([0,1])$.}.
Define the composition
\[ \tilde{\gamma}_u: E_u^{\oplus 2^{\dimA}}
  \To{\cong} E_{u+1}
  \To{\subseteq} \bigcup_{u\in\NN} E_u
  \To{\hookrightarrow} L^p(\II_A;B). \]
Then we obtain a limit
\[ \gamma = \dirlim \tilde{\gamma}_u: L^p(\II_A; B)^{\oplus 2^{\dimA}} \to L^p(\II_A; B), \]
which is an extension of juxtaposition $\gamma_u$, and we call it a \defines{complete juxtaposition}.

\subsection{Special Banach module categories}

\begin{definition}\rm
A \defines{special Banach module category} $\scrA^p_{\homo}$ is a category whose object is a triple $(N,\eta_N,\delta_N)$
given by a Banach $(A,B)$-bimodule $N$, $(A,B)$-homomorphism $\eta_N: B\to N$,
and a linear map $\delta_N: N^{\oplus 2^{\dimA}} \to N$
such that the following conditions \ref{Aobj1} and \ref{Aobj2} hold;
and whose morphism $\theta: (N,\eta_N,\delta_N) \to (M,\eta_M,\delta_M)$ is an $(A,B)$-homomorphism
such that the following conditions \ref{Amorp1} and \ref{Amorp2} hold.
\begin{enumerate}[label=($\mathscr{A}$\arabic*)]
  \item There is an index set $I_N$ which implies an $(A,B)$-homomorphism
    $\mathds{P}_N:B^{\times I_N}\to N$ sends $(1_B)_{I_N}$ to an element
    $v_N:=\mathds{P}_N((1_B)_{I_N})\in N$ with $\|v\|_V < \mu_{\II_A}(\II_A)$.
    Here, $(1_B)_{I_N} := (1_B)_{1\times I_N}$ is an element in the Cartesian product
    $B^{\times I_N} = \{ (b_i)_{1\times I_N} := (b_i)_{i\in I_N} \mid b_i \in B\}$.
    \label{Aobj1}
  \item The linear map $\delta_N$ is both a bounded $\CC$-linear map
    and an $(A,B)$-homomorphism sending $(v)_{1\times 2^{\dimA}}=(v,v,\ldots,v)$ to $v$.
    By the boundedness, it is clear that for any Cauchy sequence
    $\{x_i\}_{i\in \NN}$ in $N^{\oplus_p 2^{\dimA}}$, we have $\delta(\invlim x_i) = \invlim \delta(x_i)$.
    \label{Aobj2}
  \item $\theta(v_N)=v_M$.
    \label{Amorp1}
  \item the following diagram
\[
\xymatrix@R=1.5cm@C=1.5cm{
 N^{\oplus 2^{\dimA}}
\ar[r]^{\delta_N}
\ar[d]_{\theta^{\oplus 2^{\dimA}}
}
 &
 N
\ar[d]^{\theta}
\\
M^{\oplus 2^{\dimA}}
\ar[r]_{\delta_M}
 &
M
}
\]
holds.
    \label{Amorp2}
\end{enumerate}
\end{definition}

Special Banach module category $\scrA^p_{\homo}$ can be seen as a subcategory of the $(A,B)$-bimodule category ${_A\Modcat_B}$ defined on $\homo:A \to B$, whose objects are Banach $(A,B)$-bimodules satisfying the conditions \ref{Aobj1} and \ref{Aobj2}; and whose morphisms are $(A,B)$-homomorphisms satisfying the conditions \ref{Amorp1} and \ref{Amorp2}.
The following theorem shows that any function in $L^1(\II_A;B)$ has an abstract integration, and this abstract integration can be described by using a unique morphism starting with the initial object in $\scrA^1_{\homo}$.

\begin{theorem}[{\!\!\cite[Theorem 5.1]{Liu2025pre}}] \label{thm:int}
Take $\II=[0,1]$ {\rm(}then we have
$\II_A \mathop{\longleftrightarrow}\limits^{\text{1-1}} [0,1]^{\times \dimA}$
in this case{\rm)}.
The special Banach module category $\scrA^1_{\homo}$ \textbf{always} has an initial object $(L^1(\II_A;B),\id_{\II_A},\gamma)$
and an object $(B,\mu_{\II_A}(\II_A), \ave)$ such that the unique morphism
\[ \Int_{\homo}\in \Hom_{\scrA^1_{\homo}}
 (
   (L^1(\II_A;B),\id_{\II_A},\gamma),
   (B,\mu_{\II_A}(\II_A), \ave)
 ) \]
sends each $f\in L^1(\II_A;B)$ to its abstract integral, i.e.,
\[ f \mapsto (\scrA^1_{\homo})\int_{\II_A} f\dd\mu_{\II_A} := \Int_{\homo}(f). \]
Here, $\ave: B^{\oplus 2^{\dimA}} \to B$ sends each $(y_1,y_2,\ldots, y_{2^{\dimA}})$ to the average
$\sum\limits_{i=1}^{2^{\dimA}} \frac{1}{2^{\dimA}} y_i$.
\end{theorem}

\section{Central integrals of regular modules} \label{sect:cent int}

\textsl{With the notion of integration in place, it is natural to study integrals of certain algebra homomorphisms. In this section, we investigate the first type of integral considered in this paper, namely the \textbf{central integral}, which is obtained by integrating over the canonical inclusion from the center $\cent{A}$ of a finite-dimensional algebra $A$ into $A$. We shall see that the central integral is exactly $\cent{A}$.}

\subsection{Central integrals}

Consider the canonical inclusion $\iota_A: \cent{A} \to A$ from the center $\cent{A}$ of $A$ to a finite-dimensional complex algebra $A$. Obviously, $\iota_A$ satisfies CIC, and $\cent{A}$ is a $\CC$-linear subspace of $A$.
We choose a cube $\II_{\cent{A}}$ in $\cent{A}$ and equip it with a
natural relative Hausdorff measure $\mu_{\II_{\cent{A}}}$ of $\mu_{\II_A}$.
For $f\in L^1(\II_{\cent{A}};A)$, its integral $\Int_{\iota_A}(f) = (\scrA^1_{\iota_A}) \int_{\II_{\cent{A}}} f\dd\mu_{\II_{\cent{A}}}$ exists by Theorem \ref{thm:int}.
To be simplicity, we use $\CentIntA{f}$ to represent this integral.

\begin{definition} \rm
We call $\CentIntA{f}$ a \defines{central categorical integral} of $A$ if $f$ satisfies \defines{central condition},
i.e., $f(x)\in \cent{A}$ for almost every $x\in\II_{\cent{A}}$.
\end{definition}

The adjective ``central'' refers to the structure homomorphism and not to an arbitrary $A$-valued integrand. The integral is guaranteed to be central when the function itself is center-valued.

\begin{lemma} \label{lemm:center-closure}
If a function $f\in L^1(\II_{\cent{A}};A)$ satisfies central condition, then $\CentIntA{f} \in \cent{A}$.
\end{lemma}

\begin{proof}
By \eqref{eq:Lp}, the set of all step functions are dense in $L^1(\II_{\cent{A}};A)$,
so there is a Cauchy sequence of functions $\big\{f_i=\sum\limits_{j=1}^r a_{ij}\id_{J_{ij}}\big\}_{i\in\NN}$,
where every $f_i$ is a step function on pairwise disjoint measurable cells $J_{ij}$,
such that $f = \invlim f_i$ and $\bigcup_j J_{ij} = \II_{\cent{A}}$.
We have
\[\CentIntA{f_i} = \sum_{j=1}^r a_{ij}\mu_{\II_{\cent{A}}}(J_{ij}).\]
Without loss of generality, we assume that all $a_{ij}$ lie in $\cent{A}$,
then the sum $\sum\limits_{j=1}^r a_{ij}\mu_{\II_{\cent{A}}}(J_{ij})$ lies in $\cent{A}$,
i.e., each $f_i(x)\in\cent{A}$ for almost every $x$.
Moreover, $\cent{A}$ is a finite-dimensional, hence closed, subspace of $A$,
then passage to limits $f=\invlim f_i$ proves this proposition
since the morphism $\Int_{\iota_A}$ is continuous \footnote{Indeed, $\Int_{\iota_A}$ is a Bochner integration.}.
\end{proof}

\subsection{Average central integrals}

For each $f\in L^1(\II_{\cent{A}}; A)$ and a positive measurable subset $J\subseteq \II_{\cent{A}}$,
we have $\frac{1}{\mu_{\II_{\cent{A}}}(J)}\id_{J}f \in L^1(\II_{\cent{A}}; A)$ because
$\II_{\cent{A}}\backslash J$ is measurable, and its integral is well-defined and equals to
\[  \frac{1}{\mu_{\II_{\cent{A}}}(J)}
  \bigg(
  (\scrA^1_{\iota_A}) \int_{\II_{\cent{A}}} f  \dd\mu_{\II_{\cent{A}}} -
  (\scrA^1_{\iota_A}) \int_{\II_{\cent{A}}\backslash J} f  \dd\mu_{\II_{\cent{A}}}
  \bigg). \]

\begin{definition} \rm
The following central integral of $A$
\[ \olCentIntA{J}{f}
:= \CentIntA{\frac{1}{\mu_{\II_{\cent{A}}}(J)}\id_Jf}
 = \frac{1}{\mu_{\II_{\cent{A}}}(J)}\CentIntA{\id_Jf}\]
is an \defines{average central integral} of $A$ when $f\in L^1(\II_{\cent{A}};A)$ satisfies central condition.
We use $\ACI(A)$ to represent the set of all average central integrals of $A$.
\end{definition}

\begin{lemma} \label{lemm:AveCentInt1}
Let $M$ be a left $A$-module whose left $A$-action $A\times M\to M$ is induced by the homomorphism $\rho_M: A\to \End_{\CC}(M)$,
and $\jmath_A: \II_{\cent{A}} \to A$ be the restriction of the canonical inclusion $\iota_A$.
Then for every positive measurable subset $J\subseteq \II_{\cent{A}}$, we have:
\begin{enumerate}[label={\rm(\arabic*)}]
  \item $\olCentIntA{J}{f} \in \cent{A}$ for each $f\in L^1(\II_{\cent{A}};A)$ satisfies central condition;
  \item and every finite-dimensional left $A$-module $M$ receives an operator $\rho_M(\olCentIntA{J}{f})$
    which is a left $A$-homomorphism in $\End_A(M)$.
\end{enumerate}
\end{lemma}

\begin{proof}
The average central integral $\olCentIntA{J}{f}$ is a central integral of $A$
whose integrand function is $\frac{1}{\mu_{\II_{\cent{A}}}(J)}\id_{J}f$.
Thus, (1) holds by Lemma \ref{lemm:center-closure}. For (2), we write $c_J = \olCentIntA{J}{f}$ in this proof.
By (1), we have $c_J \in\cent{A}$, i.e., for all $a \in A$, we have $ac_J=c_Ja$,
and so, $\rho_M(a) \rho_M(c_J) = \rho_M(c_J) \rho_M(a)$.
For any $a\in A$ and $m \in M$, we have
\[
  \rho_M(c_J)(a.m)
= \rho_M(c_J) \rho_M(a)(m)
= \rho_M(a) \rho_M(c_J)(m)
= a.\rho_M(c_J)(m).
\]
Therefore, $\rho_M(c_J) : M \to M$ commutes with the left $A$-action on $M$,
and then we obtain $\rho_M(c_J) \in \End_A(M)$. This proves (2).
\end{proof}

\begin{lemma} \label{lemm:AveCentInt2}
Each $c\in \cent{A}$ is an average central integral of $A$.
\end{lemma}

\begin{proof}
For any $c\in\cent{A}$, take $f$ is the function $\II_{\cent{A}} \to A$, $x\mapsto c$.
The computation above gives $\olCentIntA{\II_{\cent{A}}}{f}=c$.
\end{proof}

By the measure $\mu_{\II_{\cent{A}}}$ is a relative measure of $\mu_{\II_A}$,
$\II_{\cent{A}}$ can be written as a measure space $(\II_{\cent{A}}, \Sigma_{\cent{A}}, \mu_{\II_{\cent{A}}})$.
According to the above notations, the set of all average central integrals of $A$ is
\begin{align*}
  \ACI(A) =
\{ \olCentIntA{J}{f}: &
  f\in L^1(\II_{\cent{A}}; A) \text{ satisfies central condition}, \\
  & J \in \Sigma_{\cent{A}},~\text{~and~} \mu_{\II_{\cent{A}}}(J)>0
\}.
\end{align*}

\begin{proposition} \label{prop:AveCentInt}
We have a bijection
\begin{align*}
  \cent{A} & \mathop{\longrightarrow}\limits^{\text{1-1}} \ACI(A), \\
c & \mapsto
  \frac{ \displaystyle
    (\scrA^1_{\iota_A})\int_{\II_{\cent{A}}} c\id_{\II_{\cent{A}}} \dd\mu_{\II_{\cent{A}}}
  }{\mu_{\II_{\cent{A}}}(\II_{\cent{A}})}.
\end{align*}
\end{proposition}

\begin{proof}
Lemmas \ref{lemm:AveCentInt1} and \ref{lemm:AveCentInt2} admits this statement.
\end{proof}

\subsection{The naturality of average central integrals}

The following result shows the naturality of central integral actions.

\begin{theorem} \label{thm:AveCentInt} \
Each $\olCentIntA{J}{f}\in \ACI(A)$ provides a family
\[(\rho_M(\olCentIntA{J}{f}))_{M \in {_A\modcat}}=
\bigg(
  \rho_M
    \bigg(
      \frac{1}{\mu_{\II_{\cent{A}}}(J)}
        \cdot (\scrA^1_{\iota_A}) \int_{\II_{\cent{A}}} \id_Jf\dd\mu_{\II_{\cent{A}}}
    \bigg)
\bigg)_{M \in {_A\modcat}} \]
of endomorphisms which induces a natural transformation
\[ \rho_{*}(\olCentIntA{J}{f}) : \ident_{{_A\modcat}} \longrightarrow \ident_{{_A\modcat}}. \]
\end{theorem}

\begin{proof}
For every $c\in\cent{A}$ and $m\in M$, we have $\varphi(c.m) = c.\varphi(m)$,
it follows that
\begin{align*}
  (\varphi \compos \rho_M(c))(m)
& = \varphi(\rho_M(c)(m)) \\
& = \varphi(c.m) \\
& = c.\varphi(m) \\
& = \rho_N(c) \compos \varphi (m),
\end{align*}
and so $\varphi \compos \rho_M(c)
= \rho_N(c) \compos \varphi$ holds.
Proposition \ref{prop:AveCentInt} shows $\cent{A} = \ACI(A)$,
then this theorem holds.
\end{proof}

Given an average central integral $\olCentIntA{J}{f}$ of $A$.
For every bounded complex $X^{\bullet} = (X^m, d_X^m)_{m\in\ZZ}$ in $\Dcat^b(A)$,
define $\rho_{X^{\bullet}}(\olCentIntA{J}{f})$ by
\[ \rho_{X^{\bullet}}(\olCentIntA{J}{f}) : X^\bullet \longrightarrow X^\bullet \]
where $\rho_{X^{\bullet}}(\olCentIntA{J}{f})(x^m) = \olCentIntA{J}{f}. x^m  \in X^m$ ($\forall~ x^m\in X^m$).
Since $\olCentIntA{J}{f}$ is central, each $\rho_{X^m}(\olCentIntA{J}{f})$ commutes with every homomorphism of $A$-modules,
in particular with the differentials $d_X^m$.
Hence $\rho_{X^\bullet}(\olCentIntA{J}{f})$ is indeed a chain map.
Moreover, it is compatible with chain homotopies and quasi-isomorphisms,
so it descends to a well-defined endomorphism in $\Dcat^b(A)$.

\begin{theorem}\label{thm:AveCentIntDer}
For each $\olCentIntA{J}{f}\in \ACI(A)$,
the family $(\rho_{X^{\bullet}}(\olCentIntA{J}{f}))_{X^\bullet \in \Dcat^b(A)}$ induces a natural transformation
\[\rho_{*}(\olCentIntA{J}{f}) : \ident_{\Dcat^b(A)} \longrightarrow \ident_{\Dcat^b(A)}.\]
\end{theorem}


\begin{proof}
Lte $c=\olCentIntA{J}{f}\in\cent{A}$, and let $\varphi:X^\bullet\to Y^\bullet$ be a morphism in the category
$\mathsf{C}^b(A)$ of bounded complexes of left $A$-modules.
For each degree $m$, the map
$\varphi^m:X^m\to Y^m$ is a homomorphism of left $A$-modules,
so it commutes with left multiplication by $c$, i.e.,
$ \varphi^m(c.x^m)=c.\varphi^m(x^m)$, for all $x^m\in X^m$.
Thus, $\varphi^m\compos\rho_{X^m}(c) = \rho_{Y^m}(c)\compos\varphi^m$ for every $m$. It follows that
\[ \varphi\compos\rho_{X^\bullet}(c) = \rho_{Y^\bullet}(c)\compos\varphi.\]
Therefore, the family $(\rho_{X^\bullet}(c))_{X^\bullet\in\mathsf{C}^b(A)}$ is natural on the category of bounded complexes.
This family is compatible with chain homotopies, and hence induces a natural transformation on the bounded homotopy category.
Since it also commutes with quasi-isomorphisms, it descends through the localization, i.e., Verdier quotient, to a natural transformation
\[ \rho_*(c): \ident_{\Dcat^b(A)} \to  \ident_{\Dcat^b(A)}.\]
Finally, Proposition \ref{prop:AveCentInt} gives $\ACI(A)=\cent{A}$, then this theorem holds.
\end{proof}

\subsection{The algebraic structures of average central integrals}

Define
\[\mathfrak{Z}(A) := \{\rho_{*}(\olCentIntA{J}{f}): \olCentIntA{J}{f} \in \ACI(A)\}\]
and
\[\Nat(\ident_{\Dcat^b(A)}, \ident_{\Dcat^b(A)}) := \{\text{natural transformation from~} \ident_{\Dcat^b(A)} \text{~to~} \ident_{\Dcat^b(A)}\}.\]
The following result is a direct corollary of Theorem \ref{thm:AveCentIntDer}.

\begin{corollary} \label{coro:AveCentIntDer}
There is a canonical embedding $\mathfrak{Z}(A) \mathop{\longrightarrow}\limits^{\subseteq} \Nat(\ident_{\Dcat^b(A)}, \ident_{\Dcat^b(A)})$.
\end{corollary}

We can check that $\Nat(\ident_{\Dcat^b(A)}, \ident_{\Dcat^b(A)})$ is a $\CC$-algebra
by the following three simple calculations:
\begin{itemize}
\item
First of all, for \(\eta,\theta \in \Nat(\ident_{\Dcat^b(A)}, \ident_{\Dcat^b(A)})\), we have a natural addition
\[ (\eta + \theta)_{X^{\bullet}} = \eta_{X^{\bullet}} + \theta_{X^{\bullet}}
\quad\text{for every } X^{\bullet} \in \Dcat^b(A). \]
Since $\eta_{X^{\bullet}}, \theta_{X^{\bullet}} \in \End_{\Dcat^b(A)}$,
the sum $\eta_{X^{\bullet}} + \theta_{X^{\bullet}}$ lines in $\End_{\Dcat^b(A)}$.
To verify naturality, let $\varphi: X^{\bullet} \to Y^{\bullet}$ be a morphism in $\Dcat^b(A)$.
Then
$(\eta + \theta)_{Y^{\bullet}} \compos \varphi
= \eta_{Y^{\bullet}}\compos\varphi + \theta_{Y^{\bullet}}\compos\varphi
= \varphi\compos\eta_{X^{\bullet}} + \varphi\compos\theta_{X^{\bullet}}
= \varphi\compos(\eta + \theta)_{X^{\bullet}}$.
Thus $\eta+\theta$ is natural. The zero natural transformation $0$ is given by $0_{X^{\bullet}}$ for every $X^{\bullet}$, and serves as the additive identity. The additive inverse of $\eta$ is $-\eta$ defined pointwise by $(-\eta)_{X^{\bullet}} := -\eta_{X^{\bullet}}$. Hence $\Nat(\ident_{\Dcat^b(A)},\ident_{\Dcat^b(A)})$ is an Abelian group.

\item
Second, for any $k\in\CC$ and $\eta\in\Nat(\ident_{\Dcat^b(A)}, \ident_{\Dcat^b(A)})$, we have a natural scalar multiplication
\[ (k \eta)_{X^{\bullet}} = k\eta_{X^{\bullet}}
\quad\text{for every } X^{\bullet} \in \Dcat^b(A).\]
The distributivity, associativity, and unitality of scalar multiplication are inherited pointwise from the $\CC$-linear space structure of each $\End_{\Dcat^b(A)}(X^{\bullet})$.
Then $\Nat(\ident_{\Dcat^b(A)}, \ident_{\Dcat^b(A)})$ is a $\CC$-linear space.

\item
Finally, the multiplication in $\Nat(\ident_{\Dcat^b(A)}, \ident_{\Dcat^b(A)})$ is defined by composition, i.e.,
for each $\eta,\theta \in \Nat(\ident_{\Dcat^b(A)}, \ident_{\Dcat^b(A)})$, its product is given by
\[ (\eta \compos \theta)_{X^{\bullet}} := \eta_{X^{\bullet}} \compos \theta_{X^{\bullet}}
\quad\text{for every } X^{\bullet} \in \Dcat^b(A). \]
Then for any morphism $\varphi: X^{\bullet} \to Y^{\bullet}$ in $\Dcat^b(A)$, we have
$(\eta\compos\theta)_{Y^{\bullet}} \compos \varphi
= \eta_{Y^{\bullet}} \compos \theta_{Y^{\bullet}} \compos \varphi
= \eta_{Y^{\bullet}} \compos \varphi \compos \theta_{X^{\bullet}}
= \varphi \compos \eta_{X^{\bullet}} \compos \theta_{X^{\bullet}}
= \varphi \compos (\eta\compos\theta)_{X^{\bullet}}$.
Thus $\eta\compos\theta$ is also a natural transformation.
The identity natural transformation $\mathds{1} := (\ident_{X^{\bullet}})_{X^{\bullet} \in \Dcat^b(A)}$
serves as the multiplicative unit, since $\mathds{1} \compos \eta = \eta \compos \mathds{1} = \eta$ pointwise.
Moreover, for $\eta,\theta,\zeta\in\Nat(\ident_{\Dcat^b(A)}, \ident_{\Dcat^b(A)})$, the distributivity laws
$(\eta+\theta)\compos\zeta = \eta\compos\zeta + \theta\compos\zeta$
and $\eta\compos(\theta+\zeta) = \eta\compos\theta + \eta\compos\zeta$
follow pointwise from the distributivity in each endomorphism space.
\end{itemize}

Then Corollary \ref{coro:AveCentIntDer} implies that $\mathfrak{Z}(A)$ is a $\CC$-algebra.
The following results further prove that $\mathfrak{Z}(A)$ is a commutative algebra.

\begin{corollary}\label{coro:frakZ}
The set $\mathfrak{Z}(A)$ is a commutative $\CC$-subalgebra of $\Nat(\ident_{\Dcat^b(A)}, \ident_{\Dcat^b(A)})$.
Furthermore, the map
\begin{center}
$\cent{A} \longrightarrow \mathfrak{Z}(A),\quad
c \longmapsto \rho_*(c)$
\end{center}
is an isomorphism of $\CC$-algebras, which admits
\[\cent{A} \cong \mathfrak{Z}(A) \cong \ACI(A).\]
\end{corollary}

\begin{proof}
First of all, we have known that $\Nat(\ident_{\Dcat^b(A)}, \ident_{\Dcat^b(A)})$ is a $\CC$-algebra.
It therefore suffices to check that $\mathfrak{Z}(A)$ is closed under addition, scalar multiplication, and composition.

Let $c_1, c_2 \in \cent{A}$ and $\lambda \in \CC$.
For every bounded complex $X^\bullet \in \Dcat^b(A)$, we have
$\rho_{X^\bullet}(c_1 + c_2)
= \rho_{X^\bullet}(c_1) + \rho_{X^\bullet}(c_2)$
since $\rho_{X^\bullet}$ is $\CC$-linear.
Thus, as natural transformations, we have $\rho_*(c_1 + c_2) = \rho_*(c_1) + \rho_*(c_2)$.
Similarly, $\rho_{X^\bullet}(\lambda c_1) = \lambda \rho_{X^\bullet}(c_1)$,
hence $\rho_*(\lambda c_1) = \lambda \rho_*(c_1)$.
For multiplication, since $c_1, c_2 \in \cent{A}$ commute with every element of $A$,
for each degree $m$ and each $x \in X^m$ we have
$\rho_{X^m}(c_1) \compos \rho_{X^m}(c_2)(x)
= c_1.(c_2.x) = (c_1 c_2).x = \rho_{X^m}(c_1 c_2)(x)$.
Thus, we get $\rho_{X^\bullet}(c_1) \compos \rho_{X^\bullet}(c_2)
= \rho_{X^\bullet}(c_1 c_2)$, and consequently,
$\rho_*(c_1)\compos\rho_*(c_2) = \rho_*(c_1 c_2)$.
Therefore $\mathfrak{Z}(A)$ is closed under the addition, scalar multiplication and multiplication,
so it is a $\CC$-subalgebra of $\Nat(\ident_{\Dcat^b(A)}, \ident_{\Dcat^b(A)})$.

Moreover, for every bounded complex $X^\bullet \in \Dcat^b(A)$ and every degree $m$,
let $x \in X^m$ be arbitrary. Then
$( \rho_{X^\bullet}(c_1) \compos \rho_{X^\bullet}(c_2) )(x)
= \rho_{X^m}(c_1)( \rho_{X^m}(c_2)(x) )
= c_1. (c_2. x)$.
Notice that $c_1, c_2 \in \cent{A}$, then it equals to
$c_2.(c_1.x) = (\rho_{X^\bullet}(c_2) \compos \rho_{X^\bullet}(c_1) )(x)$.
Therefore, $\rho_{X^\bullet}(c_1) \compos \rho_{X^\bullet}(c_2)
= \rho_{X^\bullet}(c_2) \compos \rho_{X^\bullet}(c_1)$
for every $X^\bullet \in \Dcat^b(A)$.
Taking the family over all \(X^\bullet\), we obtain
$\rho_*(c_1) \compos \rho_*(c_2)
= \rho_*(c_2) \compos \rho_*(c_1)$, and then $\mathfrak{Z}(A)$ is a commutative $\CC$-subalgebra of $\Nat(\ident_{\Dcat^b(A)}, \ident_{\Dcat^b(A)})$.

Next, we write the correspondence given in this statement as
$\Phi: \cent{A} \to \mathfrak{Z}(A)$, $c \mapsto \rho_*(c)$.
The surjectivity follows immediately from the definition of $\mathfrak{Z}(A)$ as the image of $\Phi$.
For injectivity, suppose $\Phi(c_1) = \Phi(c_2)$ in $\mathfrak{Z}(A)$.
Then, for every $X^\bullet \in \Dcat^b(A)$, we have $\rho_{X^\bullet}(c_1) = \rho_{X^\bullet}(c_2)$.
In particular, take $X^\bullet$ to be the stalk complex concentrated in degree $0$,
then $\rho_A(c_1) = \rho_A(c_2)$ as endomorphisms of the left $A$-module $A$.
Applying both sides to the identity $1_A$ in $A$, we obtain $c_11_A = c_2 1_A$,
hence $c_1 = c_2$. Thus $\Phi$ is injective, and therefore an isomorphism of $\CC$-algebras.
Furthermore, the isomorphism $\cent{A} \cong \mathfrak{Z}(A) \cong \ACI(A)$ holds by using $\Phi$ and Proposition \ref{prop:AveCentInt}.
\end{proof}

\section{Integrals of modules (in probability space)} \label{sect:int-mod-prob sp}

\textsl{Let $A$ be a finite-dimensional complex algebra.
In this subsection we introduce the integrals of $A$-modules and, fix the integration domain to be a probability space, compute the integrals of modules.}

\subsection{Modules as functions in $L^1$}
\label{subsec:module-as-L1}

For any left $A$-module $M\in {_A\modcat}$, we have a homomorphism of algebras
\[ \rho_M: A \to B_M:=\End_{\CC}(M), \quad a \mapsto (\rho_M(a): m\mapsto a.m). \]
In this subsection, we regard each module $M$ via its structure homomorphism $\rho_M$ as a function.
Moreover, in this subsection, we take the norm $\|\cdot\|_{B_M}$ defined on $B_M$
is given by the $1$-norms of matrixes, and take $\II = [0,1]$. Then we have
\[ \|\ident_B\|_{B_M} = 1, \quad
\II_A = \prod\limits_{a\in\basis{A}} \II a
  ~\mathop{\longleftrightarrow}\limits^{\text{1-1}}~ [0,1]^{\times \dimA}, \quad
\text{and} \quad \mu_{\II_A}(\II_A)=1\]
in this assumption.
Thus, $(\II_A, \Sigma_{\II_A}, \mu_{\II_A})$ forms a probability space
\footnote{In probability theory, a \defines{probability space} is a measure space $(\Omega,\Sigma,\mu)$ with $\mu(\Omega)=1$.}.

\begin{lemma} \label{lemm:module-as-L1}
We have $\rho_M\in L^1(\II_A;B_M)$ for any $M\in{_A\modcat}$.
\end{lemma}

\begin{proof}
The integral in \eqref{eq:int-mod} is defined over the probability space $(\II_A, \Sigma_{\II_A}, \mu_{\II_A})$,
where the integrand is the function $\tilde{\rho}_M:=\rho_{M}|_{\II_A}: \II_A \to B_M$.
Since $B_M$ is finite-dimensional, by Lemma \ref{lemm:C algs are Banach} it is a Banach space,
and $\rho_M$ is continuous, and hence its restriction to $\II_A$ is bounded.
Thus, there exists a constant $C>0$ such that
\[ \|\rho_M(a)\|_{B_M} \=< C
\quad\text{for all } a\in \II_A. \]

The function $\tilde{\rho}_M$ is strongly measurable
(indeed, it is continuous) and
\[
  (\scrA^1_{\rho_M})\int_{\II_A} \|\rho_M\|_{B_M} \, d\mu_{\II_A}
\=< C \cdot (\scrA^1_{\rho_M}) \int_{\II_A} \id_{\II_A} d\mu_{\II_A}
 = C < +\infty,
\]
Hence $\rho_M \in L^1(\II_A; B_M)$ as required.
\end{proof}

\begin{lemma}\label{lemm:rho-functor}
For each left $A$-module $M$, $\rho_M$ induces a functor
\[ \rho_M^{\otimes}: \scrA^p_{\ident_A} \to \scrA^p_{\rho_M}. \]
\end{lemma}

\begin{proof}
On the one hand, $\rho_M$ induces a correspondence
\[ \rho_M^{\otimes} := -\otimes_A B_M:
  (X,\eta_X, \delta_X)
    \mapsto
      (X\otimes_A B_M, \eta_{X\otimes_A B_M}, \delta_{X\otimes_A B_M})\]
such that the following conditions hold (note: $X\in{_A\modcat_A}$).
\begin{enumerate}[label={\rm(\arabic*)}]
  \item The norm defined on $X\otimes_AB_M$ is a projective tensor norm, i.e., for each $z\in X\otimes_AB_M$
    \[ \|z\|_{X\otimes_AB_M}
    = \inf \bigg\{ \sum_{i} \|x_i\|_X \|b_i\|_{B_M} : z = \sum_{i} x_i\otimes b_i \bigg\}. \]
  \item The homomorphism $\eta_{X\otimes_A B_M}: B_M \to X\otimes_A B_M$ sends each $b\in B_M$ to $\eta_X(1_A)\otimes b$,
    where $\|\eta(1_A)\|_X \=< \mu_{\II_A}(\II_A)=1$.
    Then
    \begin{align*}
      \|\eta_{X\otimes_A B_M}(1_B)\|_{X\otimes_AB_M} & = \| \eta_M(1_A) \otimes \ident_M \|_{X\otimes_AB_M} \\
    & \=< \|\eta_M(1_A)\|_X \|\ident_M\|_{B_M} \\
    & \=< \|\ident_M\|_{B_M}=1
    \end{align*}
    (here, $\ident_M$ is the identity $1_{B_M}$ in $B_M$).
  \item $\delta_{X\otimes_A B_M}: (X\otimes_A B_M)^{\oplus 2^{\dimA}} \cong X^{\oplus 2^{\dimA}} \otimes_A B_M \to X\otimes_A B_M$ sends each
    $(x_1\otimes b, x_2\otimes b, \ldots, x_{2^{\dimA}}\otimes b)$
    to $\delta_X(x_1,x_2,\ldots,x_{2^{\dimA}}) \otimes b$.
    Thus, we have
    \begin{align*}
        &\delta_{X\otimes_A B_M}((\eta_M(1_A) \otimes \ident_M)_{1\times 2^{\dimA}}) \\
      =~&\delta_X(\eta_M(1_A), \eta_M(1_A), \ldots, \eta_M(1_A))\otimes \ident_M \\
      =~&\eta_M(1_A) \otimes \ident_M.
    \end{align*}
\end{enumerate}
Then $(X\otimes_A B_M, \eta_{X\otimes_A B_M}, \delta_{X\otimes_A B_M})$ satisfies \ref{Aobj1} and \ref{Aobj2},
i.e., it is an object in $\scrA^1_{\rho_M}$.

On the other hand, $\rho_M^{\otimes}$ acts on morphism $\theta$ satisfies
\[ (\theta\otimes B_M) (\eta_X(1_A)\otimes \ident_M)
 = \theta(\eta_X(1_A))\otimes \ident_M
 = \eta_Y(1_A)\otimes \ident_M\]
and such that the diagram
\[\xymatrix@C=2cm{
  X \ar[d]_{\theta} \ar[r]^{\rho_M^{\otimes}}
& X\otimes_A B_M \ar[d]^{\theta\otimes B_M}\\
  Y \ar[r]_{\rho_M^{\otimes}}
& Y\otimes_A B_M
}\]
commutes. Then $\theta\otimes B_M$ satisfies \ref{Amorp1}.

It remains to verify that \(\theta \otimes B_M\) satisfies condition \ref{Amorp2}, i.e., the diagram
\[
\xymatrix@C=2.5cm{
  X^{\oplus 2^{\dimA}} \ar[r]^{\delta_X} \ar[d]_{\theta^{\oplus 2^{\dimA}}}
  & X \ar[d]^{\theta} \\
  Y^{\oplus 2^{\dimA}} \ar[r]^{\delta_Y}
  & Y
}
\]
commutes after applying the tensor functor. Since the diagram commutes in $\scrA^1_{\ident_A}$,
we have $\theta \compos \delta_X = \delta_Y \compos \theta^{\oplus 2^{\dimA}}$.
Tensoring both sides with \(B_M\) gives
\[
  (\theta \otimes B_M) \compos (\delta_X \otimes B_M)
= (\delta_Y \otimes B_M) \compos (\theta^{\oplus 2^{\dimA}} \otimes B_M).
\]
Using the natural identifications $(X^{\oplus 2^{\dimA}}) \otimes_A B_M \cong (X \otimes_A B_M)^{\oplus 2^{\dimA}}$
and $(Y^{\oplus 2^{\dimA}}) \otimes_A B_M \cong (Y \otimes_A B_M)^{\oplus 2^{\dimA}}$,
we obtain
\[ (\theta \otimes B_M) \compos \delta_{X \otimes_A B_M}
= \delta_{Y \otimes_A B_M} \compos (\theta \otimes B_M)^{\oplus 2^{\dimA}}.\]
Thus \ref{Amorp2} holds. Together with the verification of \ref{Amorp1} already given,
$\theta \otimes B_M$ is a morphism in $\scrA^1_{\rho_M}$.

Finally, $\rho_M^{\otimes}$ preserves identities and composition. Indeed,
$\ident_X \otimes B_M = \ident_{X \otimes_A B_M}$,
and for composable morphisms $\theta_1: X \to Y$, $\theta_2: Y \to Z$, we have
$(\theta_2 \compos \theta_1) \otimes B_M
= (\theta_2 \otimes B_M) \compos (\theta_1 \otimes B_M)$.
Therefore $\rho_M^{\otimes}$ is a functor.
\end{proof}

\begin{lemma}\label{lemm:Lp-tensor}
There is an isomorphism
\[ L^p(\II_A;A) \otimes_A B_M \cong L^p(\II_A; A \otimes_A B_M).\]
In particular, if $p=1$, then the above isomorphism is an isometric isomorphism.
\end{lemma}

\begin{proof}
We only proof the case of $p=1$ and show that this isomorphism is isometric.
The case of $p>1$ is similar (and in this case this isomorphism is not isometric in general).
Let $E_u$ be the space of $A$-valued step functions constant on the cells of $\mathcal{P}_u$, so that
\[ L^1(\II_A;A) \cong \w{\bigcup_{u\>= 0} E_u}.\]
Since $B_M$ is finite-dimensional, the projective tensor product commutes with direct sums and completions
(The projective tensor product is isometrically compatible with the $L^1$-direct sum,
but not with the general $L^p$-direct sum, cf. \cite[Example 2.6, pp. 19--21 and Example 2.19, pp.~29--30]{Ryan2002}).
Therefore, \[ L^1(\II_A;A) \otimes_A B_M \cong \w{U}, \]
where $\w{U}$ is the completion of $U=\bigcup\limits_{u\>= 0} (E_u \otimes_A B_M )$.
For each $u\in\NN$, $E_u$ is isomorphic to $A^{\oplus N_u}$ (where \(N_u = 2^{u\dim A}\)), then
\[ E_u \otimes_A B_M
\cong (A^{\oplus N_u}) \otimes_A B_M
\cong (A \otimes_A B_M)^{\oplus N_u}
\cong B_M^{\oplus N_u}.\]
Thus, $E_u \otimes_A B_M$ is precisely the space of $(A \otimes_A B_M)$-valued step functions constant on the same cells.
Hence
\[ \w{U} \cong L^1(\II_A; A \otimes_A B_M). \]
Combining these identifications gives the desired isometry.
\end{proof}

\subsection{Integrals of arbitrary modules (in probability space)} \label{subsect:int-arbitr module}

By Lemma \ref{lemm:module-as-L1}, we obtain a integral (in probability space $(\II_A, \Sigma_{\II_A}, \mu_{\II_A})$)
\begin{align}\label{eq:int-mod}
   (\scrA^1_{\rho_M}) \int_{\II_A} \rho_M \dd\mu_{\II_A} \in B_M
\end{align}
for each $M\in{_A\modcat}$ by using the special Banach module category $\scrA^1_{\rho_M}$.

\begin{definition} \rm
We define its integral is \eqref{eq:int-mod} and write it as
$\displaystyle (\scrA^1_{\rho_M}) \int_{\II_A} M \dd\mu_{\II_A}$ or, for simplicity, as $\displaystyle \int_{\II_A} M$ in this paper.
\end{definition}

Let \[c_{\II_A} :=
(\scrA^1_{\ident_A}) \int_{\II_A} (\ident_A|_{\II_A}: \II_A\to A) \dd\mu_{\II_A} ~ (\in A). \]
Then we have the following formula.

\begin{theorem}\label{thm:int-module}
$\displaystyle \int_{\II_A} M = \rho_M(c_{\II_A})$.
\end{theorem}

\begin{proof}
Write $\rho_M$ as $\rho_M: A \to B_M$, $x\mapsto \rho_M(x)$. Next we show
\begin{align} \label{eq:int-module linearity}
  (\scrA^1_{\rho_M}) \int_{\II_A} \rho_M(x)\dd\mu_{\II_A} = \rho_M\bigg((\scrA^1_{\ident_A}) \int_{\II_A} x\dd\mu_{\II_A} \bigg).
\end{align}
Let $\id_{\II_A}^{(A)}: \II_A \to A$,  be the constant function $x \mapsto 1_A$,
and let $\id_{\II_A}^{(B_M)}: \II_A \to B_M$ be the constant function $x \mapsto \ident_M$.
Consider the morphism
\[\Int_{\ident_A} \in \Hom_{\scrA^1_{\ident_A}}
  (
    (L^1(\II_A; A), \id_{\II_A}, \gamma),
    (A, 1_A, \ave)
  )\]
describing the abstract integral in the category $\scrA^1_{\ident_A}$.
Using the uniqueness of the initial object and the probability space $(\II_A, \Sigma_{\II_A}, \mu_{\II_A})$,
we get that $\Int_{\ident_A}$ is the unique morphism in $\scrA^1_{\ident_A}$ satisfying
\[ \Int_{\ident_A}(\mathbf{1}_{\II_A}^{(A)}) = \mu_{\II_A}(\II_A) 1_A = 1_A. \]
By Lemma \ref{lemm:rho-functor}, $\rho_M$ induces a functor
$\rho_M^{\otimes}: \scrA^p_{\ident_A} \to \scrA^p_{\rho_M}$.
Furthermore, we have that the following diagram
\[
\xymatrix@C=1.75cm{
  L^1(\II_A;A) \ar[r]^{\Int_{\ident_A}} \ar[d]_{\rho_M^{\otimes}}
& A \ar[d]^{\rho_M^{\otimes}} \\
  L^1(\II_A;B_M) \ar[r]_{\rho_M^{\otimes}(\Int_{\ident_A})}
& B_M
}
\]
commutes since $\rho_M^{\otimes}(L^1(\II_A;A)) = L^1(\II_A;A)\otimes_A B_M \cong L^1(\II_A;B_M)$
(see Lemma \ref{lemm:Lp-tensor}) and $\rho_M^{\otimes}(A) = A\otimes_A B_M \cong B_M$.
Notice that for every $f \in L^1(\II_A; A)$,
$\rho_M^{\otimes}: L^1(\II_A;A) \to L^1(\II_A;B_M)$ sends $f$ to the function $\rho_M\compos f$ in $L^1(\II_A;B_M)$,
hence
\begin{align}\label{eq:int-module 1}
   \rho_M^{\otimes}(\Int_{\ident_A})(\rho_M \compos f)
 = \rho_M^{\otimes}(\Int_{\ident_A})(\rho_M^{\otimes}(f))
 = \rho_M(\Int_{\ident_A}(f)).
\end{align}
In particular, for the function $\id_{\II_A}^{(B_M)} = \rho_M \compos \id_{\II_A}^{(A)}$, we obtain
\[ \rho_M^{\otimes}(\Int_{\ident_A})(\id_{\II_A}^{(B_M)})
= \rho_M(\Int_{\ident_A}(\id_{\II_A}^{(A)}))
= \rho_M(\mu_{\II_A}(\II_A)\cdot 1_A)
= \rho_M(1_A) = \ident_M.
\]

On the other hand, for the morphism $\Int_{\rho_M}: L^1(\II_A; B_M) \to B_M$
in the special Banach module category $\scrA^1_{\rho_M}$, we have
\[ \Int_{\rho_M}(\id_{\II_A}^{(B_M)}) = \mu_{\II_A}(\II_A) \ident_M = \ident_M. \]
Thus $\rho_M^{\otimes}(\Int_{\ident_A})$ and $\Int_{\rho_M}$ agree on the same constant function.
By the uniqueness of the initial object, we get
\begin{align}\label{eq:int-module 2}
  \rho_M^{\otimes}(\Int_{\ident_A}) = \Int_{\rho_M}
\end{align}

Now let $f = \ident_{\II_A}: \II_A \to A$, $x \mapsto x$.
It lies in $L^1(\II_A; A)$. Then $\rho_M \compos f = \rho_M(\cdot)$, that is, the function $x \mapsto \rho_M(x)$.
On the right of \eqref{eq:int-module 2}, we have
\[ \Int_{\rho_M}(\rho_M \compos f)
= (\scrA^1_{\rho_M}) \int_{\II_A} \rho_M\dd\mu_{\II_A}
= \int_{\II_A} M. \]
On the left of \eqref{eq:int-module 2}, we have
\[ \rho_M^{\otimes}(\Int_{\ident_A})(\rho_M \compos f)
\mathop{=}\limits^{\eqref{eq:int-module 1}}
  \rho_M(\Int_{\ident_A}(\ident_{\II_A}))
= \rho_M
  \left(
    (\scrA^1_{\ident_A})
    \int_{\II_A} \ident_{\II_A}\dd\mu_{\II_A}
  \right) = \rho_M(c_{\II_A}).
\]
\eqref{eq:int-module 2} implies $\displaystyle \int_{\II_A} M = \rho_M(c_{\II_A})$, we are done.
\end{proof}

If $\mu_{\II_A}(\II_A)>0$ and $\ne 1$, the average should be written as
\[\bar{c}_{\II_A} := \frac{1}{\mu_{\II_A}(\II_A)} (\scrA^1_{\ident_A}) \int_{\II_A} \ident_A \dd \mu_{\II_A}.\]
Correspondingly, we have
\[\frac{1}{\mu_{\II_A}(\II_A)} \int_{\II_A} M
= \frac{1}{\mu_{\II_A}(\II_A)} (\scrA^1_{\rho_M}) \int_{\II_A} \rho_M \dd\mu_{\II_A} = \rho_M (\bar{c}_{\II_A}).\]
This indicates that choosing probability space to study integrals of modules does not lose generality

\begin{theorem}\label{thm:int-module iso}
For two left $A$-modules $M$ and $N$ and a homomorphism $h:M\to N$,
if $h$ is an isomorphism, i.e., $M\mathop{\to}\limits^{h}_{\cong} N$, then
\[ \bigg( \int_{\II_A} N \bigg) \compos h = h \compos \left(\int_{\II_A} M\right).\]
\end{theorem}

\begin{proof}
By Theorem \ref{thm:int-module} we have
\[\int_{\II_A} M = \rho_M(c_{\II_A}) \quad \text{and} \quad \int_{\II_A} N = \rho_N(c_{\II_A}). \]
Notice that for any $a\in A$ and $m\in M$, the homomorphism $h:M\to N$ admits $h(am)=ah(m)$.
Then we have $h\compos \rho_M(a) = \rho_N(a)\compos h$, i.e., the following diagram
\[
\xymatrix{
 M \ar[r]^{\rho_M(a)} \ar[d]_{h} & M \ar[d]^{h} \\
 N \ar[r]_{\rho_N(a)} & N
}
\]
commutes. Here, if $h$ is an isomorphism, then we have
\begin{align*}
  h \compos \left(\int_{\II_A} M\right) \compos h^{-1} =\rho_N(c_{\II_A}) = \int_{\II_A}N
\end{align*}
as required.
\end{proof}

\subsection{Integrals of direct sums (in probability space)}
\label{subsect:int-dir sum}

Put two left $A$-modules $M$ and $N$. Now we compute $\displaystyle\int_{\II_A}M\oplus N$.

\begin{proposition}\label{prop:int-dir sum}
The integral of a  sum of two left $A$-modules is the direct sum of integrals of those two left $A$-modules.
\end{proposition}

\begin{proof}
By Theorem \ref{thm:int-module}, we have
\[ \int_{\II_A} (M \oplus N)
= \rho_{M \oplus N}(c_{\II_A}). \]
For every $a \in A$, the representation of the direct sum is the direct sum of the representations, i.e.,
$\rho_{M \oplus N}(a) = \rho_M(a) \oplus \rho_N(a)$. Therefore,
\[
\rho_{M \oplus N}(c_{\II_A})
= \rho_M(c_{\II_A}) \oplus \rho_N(c_{\II_A})
= \left(\int_{\II_A} M\right) \oplus \left(\int_{\II_A} N\right).
\]
This proves the proposition.
\end{proof}

\begin{remark}\rm
In the proof of Proposition \ref{prop:int-dir sum}, $\rho_M \oplus \rho_N$ denotes the block diagonal endomorphism of $M \oplus N$ given by $\operatorname{diag}(\rho_M, \rho_N)$.
To be more precise, let $(m,n)\in M\oplus N$ and $a\in A$.
By the componentwise definition of the left $A$-action on a direct sum, $\rho_{M\oplus N}(a)(m,n)=(am,an)$.
On the other hand, $(\rho_M\oplus\rho_N)(a)(m,n) =(\rho_M(a)m,\rho_N(a)n) =(am,an)$.
Thus, $\rho_{M\oplus N}(a)=(\rho_M\oplus\rho_N)(a)$ for every $a\in A$.
It follows that $\rho_{M\oplus N}=\rho_M\oplus\rho_N$.
\end{remark}

We immediately obtain the following corollary.

\begin{corollary}\label{coro:int-dir sum}
For a family left $A$-module $M_1$, $\ldots$, $M_n$, we have
\[\int_{\II_A} \bigoplus_{i=1}^n M_i = \bigoplus_{i=1}^n \int_{\II_A} M_i.\]
\end{corollary}

\subsection{Integrals of projective modules (in probability space)} \label{subsect:int-proj module}

An indecomposable projective left module $P$ over a finite-dimensional complex algebra $A$ is isomorphic to $Ae$ where $e$ is an \defines{idempotent}, i.e., $P \cong Ae$ for some $e^2=e$.
Theorem \ref{thm:int-module iso} shows that
\[ \int_{\II_A} Ae = h \compos \left(\int_{\II_A} P\right) \compos h^{-1} \]
where $h$ is the isomorphism between $Ae$ and $P$.
Now, we calculate the integral of $Ae$ in the probability space $(\II_A, \Sigma_{\II_A},\mu_{\II_A})$.

\begin{proposition}\label{prop:int-proj mod}
Assume $\basis{A}=\{a_1,\ldots,a_{\dimA}\}$ and $\II=[0,1]$.
Let $e \in A$ be a nonzero idempotent. Then for any basis $\{p_1,\ldots,p_r\}$ of $Ae$, we have
\[\int_{\II_A} Ae =
\frac{1}{2}
\left(
\begin{matrix}
  \sum\limits_{i=1}^{\dimA} \lambda_{i,1,1} & \sum\limits_{i=1}^{\dimA} \lambda_{i,2,1} & \cdots & \sum\limits_{i=1}^{\dimA} \lambda_{i,r,1} \\
  \sum\limits_{i=1}^{\dimA} \lambda_{i,1,2} & \sum\limits_{i=1}^{\dimA} \lambda_{i,2,2} & \cdots & \sum\limits_{i=1}^{\dimA} \lambda_{i,r,2} \\
  \vdots & \vdots & & \vdots \\
  \sum\limits_{i=1}^{\dimA} \lambda_{i,1,r} & \sum\limits_{i=1}^{\dimA} \lambda_{i,2,r} & \cdots & \sum\limits_{i=1}^{\dimA} \lambda_{i,r,r}
\end{matrix}
\right)_{r\times r} \]
where $\lambda_{i,j,k} \in \CC$ is obtained by $a_ip_j=\sum\limits_{k=1}^r\lambda_{i,j,k} p_k$.
\end{proposition}

\begin{proof}
First of all, we have
\[c_{\II_A}=(\scrA^1_{\ident_A})\int_{\II_A}\ident_A \dd\mu_{\II_A} = \frac{1}{2}\sum_{i=1}^{\dimA}a_i\]
and have
\[\rho_{Ae}:A \to B_{Ae}=\End_{\CC}(Ae), \quad \rho_{Ae}(a)(xe)=axe ~ (a,\, x\in A).\]
Then $\rho_{Ae}(c_{\II_A})(xe)
= c_{\II_A}xe
= \frac{1}{2}\sum\limits_{i=1}^{\dimA}a_i xe,$
i.e., $\rho_{Ae}(c_{\II_A}) \in B_{Ae}$ is a restriction $L_{c_{\II_A}}|_{Ae}$ of the left multiplication
$L_{c_{\II_A}}: A\to A$, $a \mapsto c_{\II_A}a$.

Assume that $\{p_1,\dots,p_r\}$ is a basis of $Ae$ where $r=\dim Ae$.
Then for each $a_i$ and $r_j$, write
\[ a_ip_j=\sum\limits_{k=1}^r\lambda_{i,j,k} p_k, \]
we have
\begin{align*}
  \rho_{Ae}(c_{\II_A})(p_j)
= L_{c_{\II_A}}(p_j)
= \bigg(\frac{1}{2}\sum_{i=1}^{\dimA}a_i\bigg)\cdot p_j\
= \frac{1}{2}\sum_{i=1}^{\dimA}\sum_{k=1}^r\lambda_{i,j,k}p_k\
= \sum_{k=1}^r
  \bigg(
    \frac{1}{2}\sum_{i=1}^{\dimA}\lambda_{i,j,k}
  \bigg)p_k.
\end{align*}
It follows the integral $\displaystyle \int_{\II_A} Ae$ by Theorem \ref{thm:int-module}.
\end{proof}

\subsection{Integrals of injective modules (in probability space)} \label{subsect:int-inj-mod}

Recall that for a finite-dimensional complex algebra $A$, the standard duality $D := \Hom_{\CC}(-, \CC)$ sends left modules to right modules and vice versa.
A left $A$-module $I$ is injective if and only if $D(I)$ is projective as a right $A$-module. Equivalently, every finite-dimensional indecomposable injective left $A$-module is isomorphic to $D(eA)$ for some primitive idempotent $e \in A$, where $D(eA)$ is viewed as a left $A$-module via
\[ (a \cdot \varphi)(x) := \varphi(xa) \qquad (a\in A,\ x\in eA,\ \varphi\in D(eA)). \]
More generally, every finite-dimensional injective left $A$-module is
isomorphic to a finite direct sum of modules of the form $D(eA)$.
Theorem \ref{thm:int-module iso} shows that
\[
\int_{\II_A}D(eA)
=
h\compos\left(\int_{\II_A}I\right)\compos h^{-1},
\]
where $h:I\to D(eA)$ is an $A$-module isomorphism.
Now, we calculate the integral of $D(eA)$ in the probability space
$(\II_A,\Sigma_{\II_A},\mu_{\II_A})$.

\begin{proposition}\label{prop:int-inj mod}
Assume $\basis{A}=\{a_1,\ldots,a_{\dimA}\}$ and $\II=[0,1]$. Let $e\in A$ be a nonzero idempotent, and let $\{p_1,\ldots,p_r\}$ be a basis of the right $A$-module $eA$. Then, with respect to the dual basis
$\{p_1^*,\ldots,p_r^*\}$ of $D(eA)$, we have
\[
\int_{\II_A}D(eA)
=
\frac{1}{2}
\left(
\begin{matrix}
  \sum\limits_{i=1}^{\dimA}\mu_{i,1,1}
  &
  \sum\limits_{i=1}^{\dimA}\mu_{i,1,2}
  &
  \cdots
  &
  \sum\limits_{i=1}^{\dimA}\mu_{i,1,r}
  \\
  \sum\limits_{i=1}^{\dimA}\mu_{i,2,1}
  &
  \sum\limits_{i=1}^{\dimA}\mu_{i,2,2}
  &
  \cdots
  &
  \sum\limits_{i=1}^{\dimA}\mu_{i,2,r}
  \\
  \vdots & \vdots & & \vdots
  \\
  \sum\limits_{i=1}^{\dimA}\mu_{i,r,1}
  &
  \sum\limits_{i=1}^{\dimA}\mu_{i,r,2}
  &
  \cdots
  &
  \sum\limits_{i=1}^{\dimA}\mu_{i,r,r}
\end{matrix}
\right)_{r\times r},
\]
where $\mu_{i,j,k}\in\CC$ is obtained from $p_j a_i=\sum\limits_{k=1}^r\mu_{i,j,k}p_k$.
\end{proposition}

\begin{proof}
Let $\{p_1^*,\ldots,p_r^*\}$ be the dual basis of $D(eA) =\End_{\CC}(D(eA))$ defined by $p_j^*(p_k)=\delta_{jk}$,
where $\delta_{jk}$ is the Kronecker symbol. First of all, we have
\[ c_{\II_A} = \frac{1}{2}\sum_{i=1}^{\dimA}a_i,\]
and, by Theorem \ref{thm:int-module}, have
\[\rho_{D(eA)}: A\to B_{D(eA)},
(\rho_{D(eA)}(a)(\varphi))(ex)=\varphi(exa) ~ (a,\, x\in A, \varphi\in D(eA)).\]
For any $\varphi\in D(eA)$ and $x\in eA$, we obtain
\begin{align*}
  (\rho_{D(eA)}(c_{\II_A})(\varphi))(x)
= \varphi(xc_{\II_A})
= \varphi\bigg( \frac{1}{2}\sum_{i=1}^{\dimA}xa_i \bigg)
= \frac{1}{2}\sum_{i=1}^{\dimA}\varphi(xa_i).
\end{align*}
Now let $\{p_1,\ldots,p_r\}$ be a basis of $eA$, and write
\[p_j a_i=\sum\limits_{k=1}^r\mu_{i,j,k}p_k,\]
then, for each $j$ and $\ell$, we have
\begin{align*}
   (\rho_{D(eA)}(c_{\II_A})(p_j^*))(p_\ell)
&= \frac{1}{2} \sum_{i=1}^{\dimA}p_j^*(p_\ell a_i)
 = \frac{1}{2} \sum_{i=1}^{\dimA}p_j^*
   \left(
     \sum_{k=1}^r\mu_{i,\ell,k}p_k
   \right) \\
&= \frac{1}{2}
   \sum_{i=1}^{\dimA}
   \sum_{k=1}^r
   \mu_{i,\ell,k}p_j^*(p_k)
 = \frac{1}{2}
   \sum_{i=1}^{\dimA}\mu_{i,\ell,j}.
\end{align*}
Therefore,
\[
\rho_{D(eA)}(c_{\II_A})(p_j^*)
=
\frac{1}{2}
\sum_{\ell=1}^r
\bigg(
  \sum_{i=1}^{\dimA}\mu_{i,\ell,j}
\bigg)p_{\ell}^*,
\]
which gives the required matrix by Theorem \ref{thm:int-module}.
\end{proof}

\begin{remark} \rm
Let $L_{c_{\II_A}}: A\to A$, $x \mapsto c_{\II_A}x$ be the left multiplication by $c_{\II_A}$,
and let $R_{c_{\II_A}}:A\to A$, $x\longmapsto xc_{\II_A}$ be the right multiplication by $c_{\II_A}$.
Comparing Proposition~\ref{prop:int-proj mod} and Proposition~\ref{prop:int-inj mod}, we have
\[ \int_{\II_A}Ae = L_{c_{\II_A}}|_{Ae} \quad \text{and} \quad
\int_{\II_A}D(eA) = D(R_{c_{\II_A}}|_{eA}). \]
\end{remark}

\section{Integrals of modules (in parameterized measure space)} \label{sect:int-mod-meas sp}

\textsl{In this subsection, we consider parameterized integrals, i.e., integrals over varying cubes
\[ \II_A(\pmb{t}) \mathop{ \longleftrightarrow }\limits^{\text{1-1}} [0,t_i]\times \cdots\times[0,t_{\dimA}], \]
where $t_1, \ldots, t_{\dimA} \>= 0$. The reason we consider parameterized integrals is that the inverse of Theorem \ref{thm:int-module iso} usually does not hold. This means that the integral of the module will lose some information about this module.}

\subsection{Integrals of arbitrary modules (in parameterized measure space)}

Assume that $\basis{A}=\{a_1,\ldots,a_{\dimA}\}$. For each
\[ \pmb{t}=(t_1,\ldots,t_{\dimA})\in(\RR^{\>=0})^{\times \dimA}, \]
we define the parameterized cube
\begin{equation}\label{eq:param-cube}
  \II_A(\pmb{t})
  :=\prod_{i=1}^{\dimA}\{\lambda_i a_i:0\=< \lambda_i\=< t_i\}
  \mathop{\longleftrightarrow}\limits^{\text{1-1}}
  \prod_{i=1}^{\dimA}[0,t_i].
\end{equation}
Notice that $\prod\limits_{i=1}^{\dimA}[0,t_i]$ is contained in
$[0,t]^{\times \dimA}$ ($t$ is a constant $> \max\{t_i:1\=< i\=< \dimA\}$),
but it is not equal to $\II_A \mathop{\longleftrightarrow}\limits^{\text{1-1}} [0,t]^{\times \dimA}$ in general.
We equip $\II_A(\pmb{t})$ with the $\sigma$-algebra $\Sigma_{\II_A(\pmb{t})}$ and the measure $\mu_{\II_A(\pmb{t})}$ transported through the bijection in \eqref{eq:param-cube}. Thus,
\begin{equation}\label{eq:param-measure}
  \mu_{\II_A(\pmb{t})}(\II_A(\pmb{t})) = \prod_{i=1}^{\dimA}t_i.
\end{equation}
If $\II_A(\pmb{t})\subseteq\II_A$, then this measure agrees with the restriction of the ambient measure
$\mu_{\II_A(\pmb{t})} =\mu_{\II_A}|_{\Sigma_{\II_A(\pmb{t})}}$.
In particular, if $\II_A$ is induced by $[0,1]^{\times \dimA}$,
the restriction description requires $0\=< t_i\=< 1$ for every $i$.

Suppose that $\II_A(\pmb{t})\subseteq\II_A$.  For every
$f\in L^1(\II_A(\pmb{t});B)$, let
\[
  \widetilde f(x)=
  \begin{cases}
    f(x),&x\in\II_A(\pmb{t});\\
    0,&x\in\II_A\backslash\II_A(\pmb{t}).
  \end{cases}
\]
Then
\begin{equation}\label{eq:param-zero-extension}
  (\scrA^1_{\homo})
  \int_{\II_A(\pmb{t})}f\dd\mu_{\II_A(\pmb{t})}
  =
  (\scrA^1_{\homo})
  \int_{\II_A}\widetilde f\dd\mu_{\II_A}.
\end{equation}

\begin{definition}\label{def:param-int-module} \rm
For every finite-dimensional left $A$-module $M$, its
\defines{parameterized integral} over $\II_A(\pmb{t})$ is
\[
  \int_{\II_A(\pmb{t})}M
  :=(\scrA^1_{\rho_M})
    \int_{\II_A(\pmb{t})}\rho_M
    \dd\mu_{\II_A(\pmb{t})}
  \in B_M=\End_{\CC}(M).
\]
\end{definition}

Put
\begin{equation*}
  c_{\II_A(\pmb{t})}
  :=(\scrA^1_{\ident_A})
    \int_{\II_A(\pmb{t})}\ident_A
    \dd\mu_{\II_A(\pmb{t})}
  \in A.
\end{equation*}
Now we provide the parameterized integral of any left $A$-module.

\begin{lemma}\label{lemm:param-barycenter}
For every $\pmb{t}=(t_1,\ldots,t_{\dimA})\in(\RR^{\>=0})^{\times\dimA}$, we have
\begin{align*}
  c_{\II_A(\pmb{t})}
= \frac{1}{2}\sum_{i=1}^{\dimA}
    \bigg(t_i^2\prod_{\substack{1\=< l\=< n\\l\neq i}}t_l\bigg)a_i
= \frac{\mu_{\II_A(\pmb{t})}(\II_A(\pmb{t}))}{2}\bigg(\sum_{i=1}^{\dimA}t_i a_i\bigg).
\end{align*}
\end{lemma}

\begin{proof}
Write every $x\in\II_A(\pmb{t})$ as $x=\sum\limits_{i=1}^{\dimA}\lambda_i a_i$, $0\=<\lambda_i\=< t_i$.
Then Fubini's theorem (see \cite[Theorem~2.37]{Folland1999}) \footnote{The vector-valued version follows from the Fubini theorem for Bochner integrals; see \cite[Theorem~9, p.~190]{DunfordSchwartz1958}.} gives
\begin{align*}
c_{\II_A(\pmb{t})}
  &=\int_{0}^{t_1}\cdots\int_{0}^{t_{\dimA}}
      \bigg(\sum_{i=1}^{\dimA}\lambda_i a_i\bigg)
      \dd\lambda_{\dimA}\cdots\dd\lambda_1\\
  &=\sum_{i=1}^{\dimA}
      \bigg(
        \int_{0}^{t_i}\lambda_i\dd\lambda_i
      \bigg)
      \bigg(
        \prod_{\substack{1\=< l\=< \dimA\\l \ne i}}
        \int_{0}^{t_l}\dd\lambda_l
      \bigg)a_i\\
  &=\frac12\sum_{i=1}^{\dimA}
      \bigg(t_i^2\prod_{\substack{1\=< l\=< \dimA\\l\neq i}}t_l\bigg)a_i \\
  &= \frac{1}{2}\bigg(\prod_{l=1}^{\dimA}t_l\bigg)
    \bigg(\sum_{i=1}^{\dimA}t_i a_i\bigg)
\end{align*}
as required. Here, $\mu_{\II_A(\pmb{t})}(\II_A(\pmb{t})) = \prod\limits_{l=1}^{\dimA}t_l$.
\end{proof}

\begin{theorem}\label{thm:param-int-module}
For every $\pmb{t}=(t_1,\ldots,t_{\dimA})\in(\RR^{\>=0})^{\times\dimA}$
and every finite-dimensional left $A$-module $M$, we have
\[\int_{\II_A(\pmb{t})}M = \rho_M(c_{\II_A(\pmb{t})})\]
In particular, if $t_i=0$ for some $i$, then $\mu_{\II_A(\pmb{t})}(\II_A(\pmb{t}))=0$ and $\displaystyle \int_{\II_A(\pmb{t})}M=0$.
\end{theorem}

\begin{proof}
Since $\rho_M$ is linear and continuous,
\begin{align*}
  \int_{\II_A(\pmb{t})}M
  &=(\scrA^1_{\rho_M})
    \int_{\II_A(\pmb{t})}\rho_M
    \dd\mu_{\II_A(\pmb{t})}\\
  &\mathop{=}\limits^{\spadesuit}
    \rho_M\left(
      (\scrA^1_{\ident_A})
      \int_{\II_A(\pmb{t})}\ident_A
      \dd\mu_{\II_A(\pmb{t})}
    \right)\\
  &=\rho_M(c_{\II_A(\pmb{t})}),
\end{align*}
and the remaining formulas follow immediately. Here, $\spadesuit$ is true since \eqref{eq:int-module linearity} holds for each fixing $\II_A(\pmb{t})$.
\end{proof}


\begin{corollary}\label{cor:param-int-module}
Keep the notations from Theorem \ref{thm:param-int-module}. For every $1\=< i\=< \dimA$, we have
\begin{equation}\label{eq:param-coefficient-recovery}
  \rho_M(a_i)
= \bigg(
    \frac{\partial^2}{\partial t_i^2}
    \prod_{\substack{1\=< l\=< \dimA\\l\neq i}}
      \frac{\partial}{\partial t_l}
  \bigg)
  \bigg(\int_{\II_A(\pmb{t})}M\bigg)
  \bigg|_{\pmb{t}\to \pmb{0}^+}.
\end{equation}
\end{corollary}

\begin{proof}
By Theorem \ref{thm:param-int-module},
\[
\int_{\II_A(\pmb{t})}M = \rho_M(c_{\II_A(\pmb{t})})
= \frac{1}{2} \bigg(\prod_{l=1}^{\dimA} t_l\bigg)
   \rho_M\bigg(\sum_{i=1}^{\dimA} t_i a_i\bigg)
= \frac12 \sum_{i=1}^{\dimA}
  \bigg(
    t_i \prod_{l=1}^{\dimA} t_l
  \bigg) \rho_M(a_i).
\]
Now apply the differential operator
$\frac{\partial^2}{\partial t_i^2}
  \prod\limits_{\substack{1\=< l\=< \dimA\\ l\neq i}}
    \frac{\partial}{\partial t_l}$
to the above formula, we have, for a fixed $i$, the monomial corresponding to $a_i$ is
$t_i \prod\limits_{l=1}^{n} t_l = t_i^2 \prod\limits_{l\neq i} t_l$.
Note that:
\begin{itemize}
  \item differentiating $\prod\limits_{l\neq i} t_l$ once with respect to each $t_l$ ($l\neq i$) gives $1$;
  \item differentiating $t_i^2$ twice with respect to $t_i$ gives $2$;
  \item the factor $\frac{1}{2}$ outside cancels with the $2$ from the second derivative.
\end{itemize}
Thus the contribution from the $i$-th term is exactly $\rho_M(a_i)$.

For any $j \neq i$, we have
\[ t_j \prod_{l=1}^{\dimA} t_l = t_i t_j \prod_{l\ne i,j} t_l \]
contains $t_i$ only to the first power, so its second derivative with respect to $t_i$ vanishes.
Hence all $j\neq i$ terms vanish. Therefore, we have
\[
\bigg(
  \frac{\partial^2}{\partial t_i^2}
  \prod_{\substack{1\=< l\=< \dimA\\ l\neq i}}
  \frac{\partial}{\partial t_l}
\bigg)
\bigg(\int_{\II_A(\pmb{t})}M\bigg)
=
\rho_M(a_i).
\]
Evaluating at $\pmb{t}\to\pmb{0}^+$ gives the same result, since the expression is independent of $\pmb{t}$.
\end{proof}

The following result show that it is essential that one and the same map simultaneously conjugates the complete parameterized integral. It also completes the inverse of Theorem \ref{thm:int-module iso}.

\begin{theorem} \label{thm:param-int-module iso}
For two left $A$-modules $M$ and $N$ and an invertible $\CC$-linear map $h:M\to N$,
$h$ is an isomorphism if and only if
\begin{align}\label{eq:param-int-simultaneous-conjugacy}
  \bigg( \int_{\II_A(\pmb{t})}N \bigg) \compos h= h \compos \bigg(\int_{\II_A(\pmb{t})}M\bigg)
\end{align}
holds for all $\pmb{t}\in(\RR^{>0})^{\times\dimA}$.
\end{theorem}

\begin{proof}
If $h$ is an isomorphism, then $\rho_N(a)=h\compos\rho_M(a)\compos h^{-1}$ holds for every $a\in A$ by Theorem \ref{thm:int-module iso} (Fixing $\pmb{t}$).
Conversely, assume that \eqref{eq:param-int-simultaneous-conjugacy} holds for every $\pmb{t}$,
then, by Corollary \ref{cor:param-int-module}, applying the differential operator $\frac{\partial^2}{\partial t_i^2}
  \prod\limits_{\substack{1\=< l\=< \dimA\\ l\neq i}}
    \frac{\partial}{\partial t_l}$ to both sides, we obtain
$\rho_N(a_i)\compos h = h\compos\rho_M(a_i)$ for every $1\=< i\=< \dimA$
(where $\basis{A} = \{a_1,\ldots,a_{\dimA}\}$). Then the same equality holds for every $a\in A$.
Hence $h$ is an $A$-module isomorphism.
\end{proof}

\subsection{Integrals of direct sums (in parameterized measure spaces)}
\label{subsect:param-int-dir sums}

We provide a proposition about parameterized integrals of the direct sums of modules here.
This proposition is parallel to Proposition \ref{prop:int-dir sum}.

\begin{proposition}\label{prop:param-int-dir sum}
For any two finite-dimensional left $A$-modules $M$ and $N$, we have
\[
  \int_{\II_A(\pmb{t})}(M\oplus N)
  =\bigg(\int_{\II_A(\pmb{t})}M\bigg)
   \oplus
   \bigg(\int_{\II_A(\pmb{t})}N\bigg).
\]
\end{proposition}

\begin{proof}
For every $a\in A$, we have $\rho_{M\oplus N}(a)=\rho_M(a)\oplus\rho_N(a)$.
It follows from Theorem \ref{thm:param-int-module} that
\begin{align*}
   \int_{\II_A(\pmb{t})}(M\oplus N)
&= \rho_{M\oplus N}(c_{\II_A(\pmb{t})}) \\
&= \rho_M(c_{\II_A(\pmb{t})}) \oplus \rho_N(c_{\II_A(\pmb{t})}) \\
&= \bigg(\int_{\II_A(\pmb{t})}M\bigg) \oplus \bigg(\int_{\II_A(\pmb{t})}N\bigg).
\end{align*}
\end{proof}

More generally, we have the following direct corollary of Proposition \ref{prop:param-int-dir sum} by induction.
This corollary is parallel to Corollary \ref{coro:int-dir sum}.

\begin{corollary}\label{coro:param-int-dir sum}
for finite-dimensional left $A$-modules
$M_1,\ldots,M_s$,
\[
  \int_{\II_A(\pmb{t})}
    \bigg(\bigoplus_{j=1}^{s}M_j\bigg)
  =\bigoplus_{j=1}^{s}
    \bigg(\int_{\II_A(\pmb{t})}M_j\bigg).
\]
\end{corollary}

\subsection{Integrals of projective modules (in parameterized measure spaces)}
\label{subsect:param-int-proj mod}

We provide a proposition which is parallel to Proposition \ref{prop:int-proj mod}.

\begin{proposition}\label{prop:param-int-proj mod}
Assume $\basis{A}=\{a_1,\ldots,a_{\dimA}\}$. Let $e \in A$ be a nonzero idempotent.
Then for any basis $\{p_1,\ldots,p_r\}$ of $Ae$, we have
\begin{equation*}
  \int_{\II_A(\pmb{t})}Ae
=\frac{\mu_{\II_A(\pmb{t})}(\II_A(\pmb{t}))}{2}
\left(
\begin{matrix}
  \sum\limits_{i=1}^{\dimA} t_i\lambda_{i,1,1} & \sum\limits_{i=1}^{\dimA} t_i\lambda_{i,2,1} & \cdots & \sum\limits_{i=1}^{\dimA} t_i\lambda_{i,r,1} \\
  \sum\limits_{i=1}^{\dimA} t_i\lambda_{i,1,2} & \sum\limits_{i=1}^{\dimA} t_i\lambda_{i,2,2} & \cdots & \sum\limits_{i=1}^{\dimA} t_i\lambda_{i,r,2} \\
  \vdots & \vdots & & \vdots \\
  \sum\limits_{i=1}^{\dimA} t_i\lambda_{i,1,r} & \sum\limits_{i=1}^{\dimA} t_i\lambda_{i,2,r} & \cdots & \sum\limits_{i=1}^{\dimA} t_i\lambda_{i,r,r}
\end{matrix}
\right)_{r\times r}
\end{equation*}
where $a_ip_j=\sum\limits_{k=1}^{r}\lambda_{i,j,k}p_k$, $\lambda_{i,j,k}\in\CC$.
\end{proposition}


\begin{proof}
The structure homomorphism of $Ae$ is $\rho_{Ae}:A\to B_{Ae}=\End_{\CC}(Ae)$, $\rho_{Ae}(a)(xe)=axe$.
By Lemma \ref{lemm:param-barycenter} and Theorem \ref{thm:param-int-module}, the formular
\begin{align*}
  \rho_{Ae}(c_{\II_A(\pmb{t})})(p_j)
  &=\frac12\bigg(\prod_{l=1}^{\dimA}t_l\bigg)
    \sum_{i=1}^{\dimA}t_i a_ip_j\\
  &=\frac12\bigg(\prod_{l=1}^{\dimA}t_l\bigg)
    \sum_{i=1}^{\dimA}\sum_{k=1}^{r}
      t_i\lambda_{i,j,k}p_k\\
  &=\sum_{k=1}^{r}
    \bigg(
      \frac12\bigg(\prod_{l=1}^{\dimA}t_l\bigg)
      \bigg(\sum_{i=1}^{\dimA}t_i\lambda_{i,j,k}\bigg)
    \bigg)p_k.
\end{align*}
shows that this proposition holds.
\end{proof}

Let $P\cong\bigoplus\limits_{v=1}^{m}Ae_v$
be a finite-dimensional projective left $A$-module, where the idempotents $e_v$ may be repeated. Then for an isomorphism
$h:\bigoplus\limits_{v=1}^{m}Ae_v\to P$, the corresponding matrices satisfy
\[ \int_{\II_A(\pmb{t})}P
 = h\compos
    \bigg(
      \bigoplus_{v=1}^{m}
      \int_{\II_A(\pmb{t})}Ae_v
    \bigg)
    \compos h^{-1}.
\]
This follows from Proposition \ref{prop:param-int-dir sum} and the
invariance of parameterized integrals under module isomorphisms.

\begin{remark} \label{rmk:int-char-proj mod} \rm
A description of projective module be found now.
Let $M$ be a finite-dimensional left $A$-module. Then the following statements are equivalent.
\begin{enumerate}[label={\rm(\arabic*)}]
  \item $M$ is projective.

  \item There is a positive integer $r\in\NN^+$ and $\CC$-linear maps
  $\iota:M\to A^{\oplus r}$, and $\pi:A^{\oplus r}\to M$
  such that:
  \begin{itemize}
    \item $\pi\compos\iota=\ident_M$;
    \item for every $\pmb{t}\in(\RR^{>0})^{\times\dim A}$, \
  \[
    \left(
      \int_{\II_A(\pmb{t})}A
    \right)^{\oplus r} \compos\iota
  = \iota\compos
    \left(
      \int_{\II_A(\pmb{t})}M
    \right)
   \quad\text{and}\quad
    \left(
      \int_{\II_A(\pmb{t})}M
    \right)\compos\pi
  = \pi\compos
    \left(
      \int_{\II_A(\pmb{t})}A
    \right)^{\oplus r}
  \] hold.
  \end{itemize}
\end{enumerate}
Indeed, $M$ is projective implies that $M$ is a direct summand of a free module $A^{\oplus r}$ $(r\in\NN)$.
Then we have a section $\iota: M\to A^{\oplus r}$ and a retraction $\pi:A^{\oplus r}\to M$ such that $\pi\compos\iota=\ident_M$.
Here, $\iota$ and $\pi$ are homomorphisms of left $A$-modules. Parameterized integrals can translate homomorphism $\iota$ and $\pi$ into the formula
\[ \left(
    \int_{\II_A(\pmb{t})}A
  \right)^{\oplus r}\compos\iota
\mathop{=}\limits^{\spadesuit}
\left(
    \int_{\II_A(\pmb{t})}A^{\oplus r}
  \right)\compos\iota
= \iota\compos
  \left(
    \int_{\II_A(\pmb{t})}M
  \right) \]
and
\[ \left(
    \int_{\II_A(\pmb{t})}M
  \right)\compos\pi
= \pi\compos
  \left(
    \int_{\II_A(\pmb{t})}A^{\oplus r}
  \right)
\mathop{=}\limits^{\spadesuit} \pi\compos
  \left(
    \int_{\II_A(\pmb{t})}A
  \right)^{\oplus r}.\]
Here, two equations marked ``$\spadesuit$'' are given by Corollary \ref{coro:param-int-dir sum}.
Conversely, suppose that the linear maps $\iota$ and $\pi$ satisfy the conditions in (2).
Then Theorem \ref{thm:param-int-module iso} and its homomorphism version imply that both $\iota$ and $\pi$ are $A$-module homomorphisms. Therefore $M$ is isomorphic to a direct summand of the free module $A^{\oplus r}$. Hence $M$ is projective.
\end{remark}

Another description of projective module is shown in Subsection \ref{subsect:proj resol} in Part II, see Theorem \ref{thm:mat char-proj mod}.

\subsection{Integrals of injective modules (in parameterized measure spaces)}
\label{subsect:param-int-inj mod}

Recall that the left $A$-action on $D(eA)$ is $(a\cdot f)(x)=f(xa) ~ (a\in A)$, $f\in D(eA) ~ (x\in eA).$
The following proposition is parallel to Proposition \ref{prop:int-inj mod}.

\begin{proposition}\label{prop:param-int-inj mod}
Assume $\basis{A}=\{a_1,\ldots,a_{\dimA}\}$.
Let $e\in A$ be a nonzero idempotent, and let $\{q_1,\ldots,q_r\}$ be a basis of the right $A$-module $eA$.
Then, with respect to the dual basis $\{q_1^*,\ldots,q_r^*\}$ of $D(eA)$, we have
\begin{equation*}
  \int_{\II_A(\pmb{t})}D(eA)
= \frac{\mu_{\II_A(\pmb{t})}(\II_A(\pmb{t}))}{2}
\left(
\begin{matrix}
\sum\limits_{i=1}^{\dimA}t_i\eta_{i,1,1} & \sum\limits_{i=1}^{\dimA}t_i\eta_{i,1,2} & \cdots & \sum\limits_{i=1}^{\dimA}t_i\eta_{i,1,r} \\
\sum\limits_{i=1}^{\dimA}t_i\eta_{i,2,1} & \sum\limits_{i=1}^{\dimA}t_i\eta_{i,2,2} & \cdots & \sum\limits_{i=1}^{\dimA}t_i\eta_{i,2,r} \\
\vdots & \vdots & & \vdots \\
\sum\limits_{i=1}^{\dimA}t_i\eta_{i,r,1} & \sum\limits_{i=1}^{\dimA}t_i\eta_{i,r,2} & \cdots & \sum\limits_{i=1}^{\dimA}t_i\eta_{i,r,r}
\end{matrix}
\right)_{r\times r}
\end{equation*}
where $\eta_{i,j,k}\in\CC$ is obtained from
$q_ja_i=\sum\limits_{k=1}^{r}\eta_{i,j,k}q_k$.
\end{proposition}

\begin{proof}
For every $1\=< j,k\=< r$, the sum $q_ja_i=\sum\limits_{k=1}^{r}\eta_{i,j,k}q_k$ gives
\[
  (a_i\cdot q_k^*)(q_j)
  =q_k^*(q_ja_i)
  =q_k^*\bigg(\sum_{s=1}^{r}\eta_{i,j,s}q_s\bigg)
  =\eta_{i,j,k}.
\]
Hence
\[ a_i\cdot q_k^* =\sum_{j=1}^{r}\eta_{i,j,k}q_j^*.\]
It follows from Theorem \ref{thm:param-int-module} that
\begin{align*}
  \rho_{D(eA)}(c_{\II_A(\pmb{t})})(q_k^*)
& =\frac12\bigg(\prod_{l=1}^{\dimA}t_l\bigg)
    \sum_{i=1}^{\dimA}t_i(a_i\cdot q_k^*)
  =\sum_{j=1}^{r}
    \bigg(
      \frac12\bigg(\prod_{l=1}^{\dimA}t_l\bigg)
      \bigg(\sum_{i=1}^{\dimA}t_i\eta_{i,j,k}\bigg)
    \bigg)q_j^*
\end{align*}
as required, where $\prod\limits_{l=1}^{\dimA}t_l = \mu_{\II_A(\pmb{t})}(\II_A(\pmb{t}))$.
\end{proof}

Let $I\cong\bigoplus\limits_{v=1}^{m}D(e_vA)$ be a finite-dimensional injective left $A$-module, where the idempotents $e_v$ may be repeated.
Then, for a chosen isomorphism $h:\bigoplus\limits_{v=1}^{m}D(e_vA)\to I$, we have
\[ \int_{\II_A(\pmb{t})}I = h\compos
    \bigg(
      \bigoplus_{v=1}^{m}
      \int_{\II_A(\pmb{t})}D(e_vA)
    \bigg) \compos h^{-1}.\]

\section*{\centering Part II: Homological dimensions}
\addcontentsline{toc}{section}{Part II: Homological dimensions}
\label{part:II}

\textsl{In this part we consider homological dimensions of algebras. Moreover We still assume $\basis{A} = \{a_1,a_2,\ldots,a_{\dimA}\}$,
$\II_A = \prod\limits_{i=1}^{\dimA}[0,t]a_i \mathop{\longleftrightarrow}\limits^{\text{1-1}} [0,t]^{\times \dimA}$,
and $\II_A(\pmb{t}) = \prod\limits_{i=1}^{\dimA}[0,t_i]a_i \mathop{\longleftrightarrow}\limits^{\text{1-1}} \prod\limits_{i=1}^{\dimA}[0,t_i]$ ($\pmb{t}=(t_1,\ldots,t_{\dimA})$) in this part.}

\section{Preliminaries: homological dimensions and indefinite integrals} \label{sect:hdim and indef int}

\textsl{In this section we recall the definitions of the standard homological dimensions and introduce the notion of an integral profile as an indefinite integral.}

\subsection{Homological dimensions}

In this subsection we recall several standard homological dimensions of finite-dimensional algebras.
These dimensions will later be characterised in terms of the integral profiles developed in Part I.

\subsubsection{Global dimensions}

Let $M$ be a finite-dimensional left $A$-module. Its \defines{left projective resolution} is an exact sequence
\[ 0 \To{} P^d \To{} \cdots \To{} P^1 \To{} P^0 \To{} M \To{} 0,\]
such that all $P^i$ are left $A$-projective modules;
and its \defines{left injective coresolution} is an exact sequence is an exact sequence
\[ 0 \To{} M \To{} I_0 \To{} I_1 \To{} \cdots \To{} I_d \To{} 0, \]
such that all $I_j$ are left $A$-injective modules.
Dually, we can define \defines{right projective resolution} and \defines{right injective resolution} of any right $A$-module.

\begin{definition}\rm
Let $M$ be a finite-dimensional left $A$-module.
\begin{enumerate}[label={\rm(\arabic*)}]
  \item The \defines{left {\rm(resp.,} right{\rm)} projective dimension} $\pdim_AM$ (resp., $\pdim M_A$) of $M$ is the minimal length $d$ of a left (resp., right) projective resolution
    if such a finite resolution exists; otherwise we say $\pdim_A M$ (resp., $\pdim M_A$) $= \infty$.
  \item Dually, the \defines{left {\rm(resp.,} right{\rm)} injective dimension} $\idim_A M$ (resp., $\idim M_A$) of $M$ is the minimal length $d$ of an injective coresolution if such a finite coresolution exists; otherwise $\idim_A M$ (resp., $\idim M_A$) $= \infty$.
\end{enumerate}
\end{definition}

These dimensions measure how far a module is from being projective or injective, respectively.
A left (resp., right) hereditary algebra is such that the projective dimensions of all left (resp., right) modules over this algebra are less than or equal to $1$. Thus, the supremum of $\pdim_A M$ measure how far a ring is from being hereditary.
This supremum is called the global dimension of algebra.

\begin{definition} \rm \
\begin{enumerate}[label={\rm(\arabic*)}]
  \item The \defines{left} \defines{global dimension $\gldim_A A$} of $A$ is defined as
    $\gldim_A A = \sup\{ \pdim_A M : M \in {_A\modcat} \}$.
  \item The \defines{right} \defines{global dimension $\gldim A_A$} of $A$ is defined as
    $\gldim A_A = \sup\{ \idim M_A : M \in {\modcat_A} \}$.
\end{enumerate}
\end{definition}

\begin{remark}\rm
Notice that the left and right global dimensions may differ; see e.g. \cite{Jen1966} for a concrete example
$A= \left(
\begin{matrix}
  \ZZ & \QQ \\ 0 & \QQ
\end{matrix}
\right)$. Here, $A$ is a $\ZZ$-algebra, and it satisfies
\[\gldim_A A = 2 > 1 = \gldim A_A.\]
However, for the finite-dimensional algebras considered in this paper, this asymmetry disappears. Since every finite-dimensional algebra is \defines{semi-primary}, i.e., whose top is semi-simple, its left and right global dimensions are coincide \cite[Corollary 9]{Aus1955}.
Consequently, \defines{for a finite-dimensional algebra $A$, we may speak of its global dimension $\gldim A$}, defined equivalently via left or right $A$-modules.
\end{remark}

For finite-dimensional algebras $A=\FF\Q/\I$, global dimension can be computed by testing only simple modules, i.e.,
\begin{align}\label{eq:gldim-simp mod}
  \gldim A = \max_{v\in\Q_0} \pdim_A S_v,
\end{align}
where $S_v$ runs over a complete set of pairwise non-isomorphic simple left $A$-modules. This good description is provided by Auslander in \cite[Proposition 7]{Aus1955}.
In general, for a finite-dimensional algebra with a large linear space dimension, the global dimension does not have a very intuitive description. Indeed, the global dimensions of some complicated algebras are not easy to compute.
One of the current strategies is to compute the global dimensions of specific algebras,
see for example \cite[etc]{Bel2011, Z2017globalcohomo, ZZ2020, FOZ2025, XinZhang2026},
thereby seeking a good description of the global dimension for a large class of algebras.

\subsubsection{Finitistic dimensions}

Another important homological dimension is the finitistic dimension. Likewise, its definition may vary depending on which module category one considers.

\begin{definition} \rm Let $A$ be a ring.
\begin{enumerate}[label={\rm(\arabic*)}]
  \item The \defines{little left finitistic dimension} of $A$ is defined as \\
      $\findim(_AA) := \sup\{\pdim_A M : M \in {_A\modcat}, \pdim _AM<\infty\}$.
  \item The \defines{little right finitistic dimension} of $A$ is defined as \\
      $\findim(A_A) := \sup\{\pdim M_A : M \in {\modcat_A}, \pdim_A M_A<\infty\}$.
  \item The \defines{big left finitistic dimension} of $A$ is defined as \\
      $\Findim(_AA) := \sup\{\pdim_A M : M \in {_A\Modcat}, \pdim _AM<\infty\}$.
  \item The \defines{big right finitistic dimension} of $A$ is defined as \\
      $\findim(A_A) := \sup\{\pdim M_A : M \in {\Modcat_A}, \pdim_A M_A<\infty\}$.
\end{enumerate}
\end{definition}

In the study of homological dimensions of algebras, an interesting problem has attracted attention, namely the following question.

\begin{question}
Are the finitistic dimensions $\findim_AA$, $\findim A_A$, $\Findim _AA$ and $\Findim A_A$ always finite?
\end{question}

The conjecture that this number is always finite for every Artin algebra goes back to Bass \cite{Bass1960} and remains one of the central open problems in the homological theory of finite-dimensional algebras.
Its history and its links with other homological conjectures are surveyed in \cite{Huis1995}.
Important approaches to study this conjecture include the Igusa--Todorov method \cite{IT2005, Wei2009} and reduction techniques based on derived recollements \cite{ChX2017Finidim}.
Of course, another natural approach to this open problem is to verify the finitistic dimension conjecture in various special settings, thereby seeking a possible breakthrough,
see \cite[etc]{ZH2020findim, BP2020findim, Gel2022depth, MS2022findim, MN2024findim, GI2025findim}.

\subsection{Indefinite integrals of modules}

Let $\mathfrak{S}:=\{\mathfrak{S}_t\}_{t\in T}$, where $T=(T,\preceq)$ is an index set, be a sequence of in the $\sigma$-algebra $\Sigma_{\II_A}$ of a given measure space $(\II_A, \Sigma_{\II_A}, \mu_{\II_A})$ such that such that $\mathfrak{S}_t \subseteq \mathfrak{S}_{t'}$ exactly when $t\preceq t'$.
Then for each $f\in L^1(\II_A; B)$, we obtain a sequence of integrals
\begin{align}\label{eq:seq of int}
\bigg\{ \bigg(\mathfrak{S}_t, (\scrA^p_{\homo:A\to B})\int_{\mathfrak{S}_t} f \dd\mu_{\II_A}\bigg) :
     t \in T
   \bigg\} \subseteq \Sigma_{\II_A} \times B.
\end{align}
Here,
\[ (\scrA^p_{\homo})\int_{\mathfrak{S}_t} f \dd\mu_{\II_A}
:= (\scrA^p_{\homo})\int_{\II_A} \id_{\mathfrak{S}_t}f \dd\mu_{\II_A}, \]
and the partially ordered set ($=$ poset, for short) $\mathfrak{S} = (\mathfrak{S},\subseteq)$ is called an \defines{integral partially ordered sets} ($=$ i.poset for short), cf. \cite[Section 3]{LLL2025}.
By introducing the i.poset, we can characterise variable-limit integrals in a natural way.

\begin{example} \rm
In classical calculus, every variable-limit integral of the form
\[ (\Lebesgue) \int_{y}^{x} f \dd\mu = (\scrA^1_{\ident_{\RR}}) \int_{y}^{x} f\dd\mu\]
($c\=< y\=< x\=< d$, and $\mu$ is a Lebesgue measure or a Borel measure) can be expressed as a sequence of integrals
\[ \mathfrak{S}=\bigg\{ \bigg([y,x],\, \int_{y}^{x} f \, d\mu\bigg) \in \RR \times \RR : c\=< y\=< x\=< d \bigg\}. \]
The i.poset corresponding to this sequence of integrals is $\{[y,x] \subseteq \RR : c\=< y\=< x\=< d\}$,
where the partial order on this i.poset is given by the inclusion ``$\subseteq$'' of sets.
In particular, if the parameter $y$ in the above i.poset $\mathfrak{S}$ is fixed,
then the corresponding sequence of integrals describe an indefinite integral of $f$.
\end{example}

Next, we define that $\pmb{t} = (t_1,\ldots,t_{\dimA}) \preceq \pmb{t}' = (t_1',\ldots,t_{\dimA}')$ if and only if $t_i\=< t_i'$ holds for all $1\=< i\=< \dimA$.
Take $\mathfrak{S} = \{ \II_A(\pmb{t}) : \pmb{t}\in (\RR^{\>=0})^{\dimA} \}$.
Concerning sequences of integrals, there is an important perspective, namely that they can be naturally regarded as a map from a poset to the values of the integral. For example, \eqref{eq:seq of int} can be viewed as
\begin{align}\label{eq:int seq-map}
  T \to B, \quad t \mapsto (\scrA^p_{\homo}) \int_{\mathfrak{S}_t} f \dd\mu_{\II_A}.
\end{align}
There is a benefit to viewing it as a map. We no longer need to rely on the index set $T$ or the partial order structure, and this can provide some convenience for the subsequent exposition.
For example, let $T=\{\pmb{t}\in(\RR^{\>=0})^{\times\dimA}: 0\=< t_i \=< t\}$, which is a set without partially ordered,
and let $\mathfrak{S} = \{\II_A(\pmb{t}) := \prod\limits_{i=1}^{\dimA}[0,t_i]a_i : 0\=< t_i \=< t\}_{\pmb{t}\in T}$.
Then \eqref{eq:int seq-map} is
\begin{align}\label{eq:int seq-map:0811}
 T \to B, \quad t \mapsto (\scrA^p_{\homo}) \int_{\II_A(\pmb{t})} f \dd\mu_{\II_A(\pmb{t})}.
\end{align}

We know that the domain of integration $\II_A(\pmb{t})$ is in bijection with the product $\prod\limits_{i=1}^{\dimA} [0,t_i]$. Each factor in this product is an interval $[0,t_i]$ with a fixed lower endpoint and a varying upper endpoint. From this perspective, we may naturally refer to \eqref{eq:int seq-map:0811}, say parameterized integral in Part I, as an \defines{indefinite integral}.
Alternatively, we may regard each factor of $[0,t]^{\dimA}$ as a set indexed by $\basis{A}$, namely the $i$-th interval $[0,t]$ is viewed as a set attached to $a_i\in \basis{A}$. In this way, $[0,t]^{\dimA}$ can be seen as a bundle over $\basis{A}$. Consequently, for each $\pmb{t}=(t_1,t_2,\ldots,t_{\dimA})$, the product $\prod\limits_{i=1}^{\dimA} [0,t_i]$ is a section of this bundle. When the values of $\pmb{t}$ are fixed, the value of \eqref{eq:int seq-map:0811} is precisely the integral taken over the corresponding section. From this viewpoint, we also refer to \eqref{eq:int seq-map:0811} as an \defines{integral profile};
and the values of the integral under different parameters are called \defines{integral sections}, since they are a choice from the parameter bundle to the algebra.

\begin{definition}\rm \label{def:int prof}
An \defines{integral profile} (or an \defines{indefinite integrals}) of a left $A$-module $M$ is defined as the map
\[ \IntProf_{\II_A} M : (\RR^{\>=0})^{\times\dimA} \to B_M=\End_{\CC}(M), \quad
  \pmb{t} \mapsto \int_{\II_A(\pmb{t})} M
= (\scrA^1_{\rho_M}) \int_{\II_A(\pmb{t})}\rho_M\dd\mu_{\II_A(\pmb{t})}. \]
\end{definition}

Theorem \ref{thm:param-int-module iso} demonstrates an important result, namely that the integral profile of a module can recover all the information of that module, implying that it is feasible to study properties of modules by using the integral profile.
Theorem \ref{thm:param-int-module iso} can be translated in the following form:
\begin{itemize}
  \item \it Two left $A$-modules $M$ and $N$ are isomorphic if and only if there is a $\CC$-linear isomorphism $h:M\to N$ such that
    $\displaystyle \bigg( \IntProf_{\II_A}N \bigg) \compos h= h \compos \bigg(\IntProf_{\II_A}M\bigg)$.
\end{itemize}

Each left $A$-module $M$ is a $\CC$-linear space. Thus we can always find a large number of $\CC$-linear isomorphisms $h: M \To{\cong} \CC^{\oplus m}$ ($m=\dim_{\CC} M$).
If $\CC^{\oplus m}$ has a decomposition $\bigoplus\limits_{i=1}^{n}\CC^{\oplus m_i}$ such that each $\CC^{\oplus m_i}$ is a left $A$-module and
\[ \bigoplus_{i=1}^n
   \bigg(
     h^{-1} \compos \bigg(\IntProf_{\II_A} \CC^{\oplus m_i} \bigg) \compos h
   \bigg)
 = \bigg(\IntProf_{\II_A}M\bigg),\]
then Theorem \ref{thm:param-int-module iso} implies that $M$, as a left $A$-module, has a decomposition $\bigoplus\limits_{i=1}^n\CC^{\oplus m_i}$.

\section{Integral coefficient ring and its Morita invariance} \label{sect:int coef ring}

\textsl{For two modules $X$ and $Y$, we use $X \le Y$ to represent that $X$ is a submodule of $Y$,
and use $X \underline{\lesssim} Y$ to represent that $X$ is isomorphic to a submodule of $Y$.
Then, naturally, $X \ds Y$ represents that $X$ is a direct summand of $Y$,
and $X \isods Y$ represents that $X$ is isomorphic to direct summand of $Y$.}

\subsection{Integral coefficient rings}\label{subsect:int coef ring}

Let ${_A\proj}$ be the full subcategory of ${_A\modcat}$ containing all projective modules and let ${_A\inj}$ be the full subcategory of ${_A\modcat}$ containing all injective modules. Define
\[{_A}\pandi := \add({_A\proj} \cup {_A\inj}) \]
the full subcategory of ${_A\modcat}$ containing all projective and injective modules such that it is closed under finite direct sum, direct summand and isomorphism. Then
\[ \pandimod := \bigoplus\limits_{v\in\mathcal Q_0}(P(v)\oplus I(v))\in {_A}\pandi, \]
and any $M\in {_A}\pandi$ is isomorphic to a direct summand of $\pandimod^{\oplus\NN} \in {_A\Modcat}$,
i.e., $M\isods \pandimod^{\oplus\NN}$.
Each $f \in \operatorname{Map}( (\RR^{\>=0})^{\times\dimA}, \End_{\CC}(M))$ can be embedding in
\[\widetilde{\pandiring}(A) := \operatorname{Map}( (\RR^{\>=0})^{\times\dimA}, \End_{\CC}(\pandimod^{\oplus\NN}) )\]
by the correspondence
\[ f
   \mapsto (\mathfrak{Emb}_M(f)(\pmb{t}))_{\pmb{t}\in(\RR^{\>=0})^{\times\dimA}} :=
   \big(\sigma_M^{-1} \compos
      \left(
        \begin{smallmatrix}
          \ident_M \\ 0
        \end{smallmatrix}
      \right)
     \compos f(\pmb{t})\compos (\ident_M\ 0) \compos\sigma_M\big)_{\pmb{t}\in(\RR^{\>=0})^{\times\dimA}}, \]
where $\mathfrak{Emb}_M(f(\pmb{t}))$ is defined as
\[ \pandimod^{\oplus\NN}
  \To{\sigma_M,\; \cong}  M\oplus M'
  \To{(\ident_M\ 0)}  M
  \To{f(\pmb{t})} M
  \To{\left(
        \begin{smallmatrix}
          \ident_M \\ 0
        \end{smallmatrix}
      \right)} M\oplus M' \To{\sigma_M^{-1},\; \cong} \pandimod^{\oplus\NN} \]
for each $\pmb{t}\in (\RR^{\>=0})^{\times\dimA}$.
Note that $\widetilde{\pandiring}(A)$ is a ring, whose \textbf{multiplication is given by the pointwise multiplication of maps}.
Now we put an assumption for every $M\in{_A\pandi}$:

\begin{assumption} \label{assu:fix emb sys} \rm
Throughout this paper, for every $M\in{_A\pandi}$, we fix once and
for all a left $A$-module $M'$ and a left $A$-module isomorphism
\[ \sigma_M: \pandimod^{\oplus\NN} \To{\cong} M\oplus M'.\]
The family
\[\pmb{\sigma}_M := \{\sigma_M:M\in{_A\pandi}\}\]
is called the \defines{fixed embedding system} used in this paper.
All embeddings $\mathfrak{Emb}_M$ and the integral coefficient ring appearing below (see Definition \ref{def:icring}) are understood to be constructed with respect to this fixed embedding system. Consequently, integral coefficient ring is uniquely determined after the above embedding system has been fixed. We \textbf{do not claim that $\pandiring(A)$ is independent of this choice}.
\end{assumption}

The following lemma shows that this correspondence is an injection under Assumption \ref{assu:fix emb sys}.

\begin{lemma} \label{lemm:import emb}
For each left $A$-module $M\in{_A\pandi}$, the map
\[ \mathfrak{Emb}_M : \operatorname{Map}( (\RR^{\>=0})^{\times\dimA}, \End_{\CC}(M)) \to \widetilde{\pandiring}(A) \]
\checks{is} an injective, not necessarily unital, homomorphism of $\CC$-algebras
\end{lemma}

\begin{proof}
First, it is easy to see that $\mathfrak{Emb}_M$ is a homomorphism of $\CC$-algebras. Next we show that it is an injection.
If $\mathfrak{Emb}_M(f(\pmb{t}))=0$, then
$ \sigma_M^{-1} \compos
      \left(
        \begin{smallmatrix}
          \ident_M \\ 0
        \end{smallmatrix}
      \right)
    \compos f(\pmb{t})\compos (\ident_M\ 0) \compos\sigma_M =0$ holds for all $\pmb{t}\in(\RR^{\>=0})^{\times\dimA}$.
Since $\sigma_M$ is an isomorphism, then
$    \left(
        \begin{smallmatrix}
          \ident_M \\ 0
        \end{smallmatrix}
      \right)\compos f(\pmb{t}) \compos (\ident_M\ 0) = 0$ for all $\pmb{t}$.
Furthermore, $(\ident_M\ 0)$ is surjective, and
$\left(
        \begin{smallmatrix}
          \ident_M \\ 0
        \end{smallmatrix}
      \right)$ is injective. Thus $f(\pmb{t})=0$ for all $\pmb{t}$. It follows $f=0$ as required.
\end{proof}

All embeddings below are understood with respect to these fixed choices made in Assumption \ref{assu:fix emb sys}.
Now we can introduces integral coefficient rings for every finite-dimensional complex algebra By Lemma \ref{lemm:import emb}.

\begin{definition}\label{def:icring}\rm
The \defines{integral coefficient ring} of the pair $(A,\pmb{\sigma}_M)$ is the subalgebra
\[ \pandiring(A,\pmb{\sigma}_M)
:= \bigg\langle\bigg\{
     \mathfrak{Emb}_M\bigg( \IntProf_{\II_A}M \bigg): M\in{_A\pandi}
   \bigg\}\bigg\rangle
\subseteq \widetilde{\pandiring}(A) \]
of $\widetilde{\pandiring}(A)$ generated by all $\mathfrak{Emb}_M(\intProf_{\II_A}M )$ with $M\in{_A\pandi}$.
\end{definition}

Since we fixing $\pmb{\sigma}_M$ By Assumption \ref{assu:fix emb sys},
we can ignore $\pmb{\sigma}_M$ in $\pandiring(A,\pmb{\sigma}_M)$ if it won't cause confusion,
i.e., we use $\pandiring(A,\pmb{\sigma}_M)$ to represent $\pandiring(A)$ for simplicity
and call it is the integral coefficient ring of $A$ in this paper.

Recall that a \defines{complex} $\pmb{X}_{\bullet} = (X_n, d_n)$ defined on ${_A\modcat}$ is a sequence of modules and left $A$-homomorphisms
\[ \cdots \To{} X_{-2} \To{d_{-2}} X_{-1} \To{d_{-1}} X_0 \To{d_0} X_1 \To{d_1} X_2 \To{} \cdots \]
such that $d_{i+1}d_{i} = 0$ holds for all $i\in \ZZ$. Each module $X_i$ has a integral profile $\IntProf_{\II_A} X_i$,
then, naturally, $\pmb{X}_{\bullet}$ has a sequence of integral profiles
\[ \bigg(\IntProf_{\II_A} X_i\bigg)_{i\in\ZZ} \in
   \prod_{i\in\ZZ} \mathrm{Map}((\RR^{\>=0})^{\times\dimA}, \End_{\CC}(X_i)). \]

\begin{definition} \rm
Assume that $\pmb{X}_{\bullet} = (X_n, d_n)$ is a \defines{non-negative chain complex}, i.e., for each $n<0$, we have $X_n=0$ and $d_n=0$.
The \defines{total integral profile} and \defines{embedded total integral profile} of a complex $\pmb{X}_{\bullet}$ are defined as
\[\IntProf_{\II_A} (X_{\bullet};x) := \sum_{i\>= 0}\bigg(\IntProf_{\II_A} X_i\bigg)x^i \]
and
\[\IntProfEmb_{\II_A} (X_{\bullet};x) := \sum_{i\>= 0}\mathfrak{Emb}_{X_i}\bigg(\IntProf_{\II_A} X_i\bigg)x^i, \]
respectively.
\end{definition}

In particular, if $\pmb{X}_{\bullet} = (X_n, d_n)$ is a projective (resp. injective) complex,
i.e., all $X_i$ in $\pmb{X}_{\bullet}$ are projective (resp., injective),
then $\IntProf_{\II_A} X_n \in \pandiring(A)$.
Two facts are trivial:
\begin{itemize}
  \item if $\pmb{X}_{\bullet}=\pmb{P}_{\bullet}$ is projective, then
    \[ \IntProfEmb_{\II_A} (P_{\bullet};x) \in \pandiring(A)[[x]];  \]
  \item if $\pmb{X}_{\bullet}=\pmb{I}_{\bullet}$ is injective, then
    \[ \IntProfEmb_{\II_A} (I_{\bullet};x) \in \pandiring(A)[[x]].  \]
\end{itemize}

Since Lemma \ref{lemm:import emb} implies that all $\mathfrak{Emb}$ are embedding, for simplicity, we call
\begin{itemize}
  \item $\displaystyle \IntProf_{\II_A}(\pmb{X}_\bullet;x) \; \widetilde{\in} \; \pandiring(A)[[x]]$
    if $\displaystyle \IntProfEmb_{\II_A}(\pmb{X}_\bullet;x) \in \pandiring(A)[[x]]$ holds;
  \item and $\displaystyle \IntProf_{\II_A}(\pmb{X}_\bullet;x) \; \widetilde{\in} \; \pandiring(A)[x]$
    if $\displaystyle \IntProfEmb_{\II_A}(\pmb{X}_\bullet;x) \in \pandiring(A)[x]$ holds.
\end{itemize}

\subsection{Morita invariance of integral coefficient rings}
\label{subsect:Morita inv icring}

The integral coefficient ring in Definition~\ref{def:icring} depends, in general, on the relative positions of the copies of modules inside $\pandimod^{\oplus\NN}$.
In this subsection, we restrict the fixed embedding system in Assumption \ref{assu:fix emb sys} to a compatible orthogonal one.
Under this restriction, we prove that the isomorphism class of the integral coefficient ring is Morita invariant.
We first use the non-unital convention for the algebra generated by the embedded integral profiles.
The unital convention is treated at the end of the subsection.

\begin{definition}\label{def:compatible orth emb sys}\rm
For every isomorphism class of indecomposable modules occurring in ${_A\pandi}$, choose a representative $X$.
A fixed embedding system $\pmb{\sigma}_M$ is called a \defines{compatible orthogonal fixed embedding system}
if there are pairwise orthogonal fixed copies $X^{(1)},X^{(2)},\ldots$ in $\pandimod^{\oplus\NN}$ such that
\[  \pandimod^{\oplus\NN} \cong \bigoplus_{X} \bigoplus_{r\>=1} X^{(r)},\]
and, whenever $M\cong\bigoplus\limits_X X^{\oplus m_X}$, the isomorphism $\sigma_M$ identifies the embedded copy of $M$ with
$\bigoplus\limits_{X}\bigoplus\limits_{1\=< r\=< m_X} X^{(r)}$.
Here, the multiplicities $m_X$ are uniquely determined by the Krull--Schmidt theorem,
see, for example, \cite[Chapter II]{ARS1995}, \cite[Section 12]{AF1992}, and \cite{Krause2015}.
\end{definition}

For every $X$ as in Definition \ref{def:compatible orth emb sys} and every $r \>= 1$,
let \[F_{X,r} : (\RR^{\>=0})^{\times\dimA}\to \End_{\CC}(\pandimod^{\oplus\NN}) \in \widetilde{\pandiring}(A)\]
be the function whose restriction to $X^{(r)}$
is $\intProf_{\II_A}X$ and whose restriction to all the other fixed copies is zero, i.e.,
\begin{align}\label{def formula: F-int}
  F_{X,r}(\pmb t)|_{Y^{(l)}} =
  \begin{cases}
    \IntProf_{\II_A}X = (\scrA^1_{\homo}){\displaystyle \int_{\II_A(\pmb{t})}X} , & (Y,l)=(X,r);\\
    0, & (Y,l)\ne(X,r).
  \end{cases}
\end{align}
The orthogonality of the fixed copies gives
\begin{equation}\label{eq:orth int coeff product}
  F_{X,r}F_{Y,l}=0 \text{~whenever~} (X,r)\ne(Y,l).
\end{equation}

\begin{lemma}\label{lemm:cyclic int coeff component}
For $0\ne X\in{_A\pandi}$ and $r\>= 1$, the element $F_{X,r}$ satisfies no nonzero polynomial relation over $\CC$ without a
constant term. Consequently,
\begin{align*}
  \langle F_{X,r}\rangle = \bigoplus_{m\>= 1}\CC F_{X,r}^{m} \cong x\CC[x]
\end{align*}
as non-unital complex algebras.
\end{lemma}

\begin{proof}
Suppose that $g(F_{X,r})=0$ for some $g(x)=\sum\limits_{m=1}^{d}c_mx^m\in x\CC[x]$.
Restricting this equality to $X^{(r)}$, we obtain $g(\intProf_{\II_A}X(\pmb{t}))=0$
for every $\pmb{t}\in(\RR^{\>= 0})^{\times\dimA}$.
It follows that every eigenvalue of $\intProf_{\II_A}X(\pmb{t})$ is a root of $g$.
Then the trace $\operatorname{tr}(\intProf_{\II_A}X(\pmb{t}))$ takes values in a finite set.
Since it is continuous on $(\RR^{\>= 0})^{\times\dimA}$, it must be constant by \cite[Theorem 23.5]{Munkres2000}.

On the other hand, By Theorem \ref{thm:param-int-module} and Lemma \ref{lemm:param-barycenter}, we have
\begin{align*}
  \operatorname{tr}\bigg(\IntProf_{\II_A}X(\pmb{t})\bigg)
= \frac{1}{2}\sum_{i=1}^{\dimA}
  \bigg(t_i^2\prod_{\substack{1 \=< l \=< \dimA \\ l \ne i}} t_l\bigg)
  \operatorname{tr}(\rho_X(a_i)).
\end{align*}
By using the basis $\basis{A} = \{a_1, \ldots, a_{\dimA}\}$ of $A$, we assume $1_A=\sum\limits_{i=1}^{\dimA}\lambda_i a_i$.
Since
\begin{align*}
  0 < \dim_{\CC}X
    = \operatorname{tr}(\rho_X(1_A))
    = \sum_{i=1}^{\dimA}\lambda_i
      \operatorname{tr}(\rho_X(a_i)),
\end{align*}
we have $\operatorname{tr}(\rho_X(a_i))\ne0$ for at least one $i$.
The monomials $t_i^2\prod\limits_{l\ne i}t_l$ are pairwise distinct.
Therefore $\operatorname{tr}(\intProf_{\II_A}X(\pmb{t}))$ is a nonconstant polynomial,
which is a contradiction. Thus, $g=0$, and this lemma follows.
\end{proof}

\begin{proposition}\label{prop:orth decomposition icring}
Assume that $\pmb{\sigma}_M$ is a compatible orthogonal fixed embedding
system.  Then
\begin{equation}\label{eq:orth decomposition icring}
  \pandiring(A)=
  \bigoplus_{\substack{X \in \ind({_A\pandi}) \\ r \>= 1}} \langle F_{X,r}\rangle
\cong
  \bigoplus_{\substack{X \in \ind({_A\pandi}) \\ r \>= 1}} x\CC[x]
\end{equation}
as non-unital complex algebras.
\end{proposition}

\begin{proof}
For every $r\>= 1$, we have
\begin{align*}
  \mathfrak{Emb}_{X^{\oplus r}} \bigg(\IntProf_{\II_A}X^{\oplus r}\bigg) = \sum_{l=1}^{r}F_{X,l}.
\end{align*}
Then $F_{X,1}$ is the embedded integral profile of $X$ and for each $r\>= 2$, we have
\begin{align*}
  F_{X,r}
= \mathfrak{Emb}_{X^{\oplus r}} \bigg(\IntProf_{\II_A}X^{\oplus r}\bigg)
- \mathfrak{Emb}_{X^{\oplus(r-1)}} \bigg(\IntProf_{\II_A}X^{\oplus(r-1)}\bigg).
\end{align*}
Thus, every $F_{X,r}$ belongs to $\pandiring(A)$.

Conversely, by Definition \ref{def:compatible orth emb sys} and Corollary \ref{coro:param-int-dir sum},
if $M \cong \bigoplus\limits_X X^{\oplus m_X}$, then we have
\begin{align*}
   \mathfrak{Emb}_M\bigg(\IntProf_{\II_A}M\bigg)
&= \mathfrak{Emb}_M\bigg(\IntProf_{\II_A}\bigoplus_X X^{\oplus m_X}\bigg) \\
&= \mathfrak{Emb}_M\bigg(\bigoplus_X\IntProf_{\II_A} X^{\oplus m_X}\bigg) \\
&= \sum_X \mathfrak{Emb}_{X^{\oplus m_X}}\bigg(\IntProf_{\II_A} X^{\oplus m_X}\bigg)\\
&= \sum_X \sum_{r=1}^{m_X} \bigg(\mathfrak{Emb}_{X^{\oplus r}}\bigg(\IntProf_{\II_A} X^{\oplus r}\bigg) \\
& \hspace{14pt} - \mathfrak{Emb}_{X^{\oplus (r-1)}} \bigg(\IntProf_{\II_A} X^{\oplus (r-1)}\bigg)\bigg)\\
&= \sum_X\sum_{r=1}^{m_X}F_{X,r}.
\end{align*}
Thus, all generators of $\pandiring(A)$ belong to the right-hand side of \eqref{eq:orth decomposition icring}.
Equation \eqref{eq:orth int coeff product} shows that the sum is an algebraic direct sum,
and, by Lemma \ref{lemm:cyclic int coeff component}, this proposition holds.
\end{proof}

The next observation shows that the isomorphism type in
Proposition~\ref{prop:orth decomposition icring} does not depend on a change
of compatible orthogonal coordinates.

\begin{proposition}\label{prop:global conjugate icring}
Let $\pmb{\sigma}_M$ and $\pmb{\tau}_M$ be two compatible orthogonal fixed embedding systems for $A$.
Then there is an automorphism $u\in\Aut_A(\pandimod^{\oplus\NN})$ such that conjugation by $u$ induces an isomorphism
\begin{align*}
  \pandiring(A,\pmb{\sigma}_M) \To{\cong} \pandiring(A,\pmb{\tau}_M), \quad f\longmapsto u f u^{-1}.
\end{align*}
\end{proposition}

\begin{proof}
Both systems decompose $\pandimod^{\oplus\NN}$ into countably many fixed copies of every isomorphism class of indecomposable modules occurring in ${_A\pandi}$.
Match the corresponding copies and choose an $A$-module isomorphism on every matched pair.
Their direct sum is an $A$-module automorphism $u$ of $\pandimod^{\oplus\NN}$.
Conjugation by $u$ sends the function supported on a fixed copy in the first system to the corresponding
function supported on the matched copy in the second system.
It therefore sends every generator $F_{X,r}$ of the first ring to the correspondinggenerator of the second ring.
The assertion now follows from Proposition \ref{prop:orth decomposition icring}.
\end{proof}

The following theorem shows that the isomorphism class, as complex algebras without identities, of the integral coefficient ring of a finite-dimensional complex algebra equipped with a compatible orthogonal fixed embedding system is a Morita invariant.

\begin{theorem}\label{thm:Morita inv icring}
Let $A$ and $B$ be finite-dimensional complex algebras equipped with compatible orthogonal fixed embedding systems.
If $A$ and $B$ are Morita equivalent, then
\begin{align*}
  \pandiring(A)\cong\pandiring(B)
\end{align*}
as non-unital complex algebras.
\end{theorem}

\begin{proof}
Let $\mathcal E:{_A\modcat} \to {_B\modcat}$ be an equivalence.
Since $\mathcal E$ is an additive exact equivalence,
it preserves and reflects indecomposability, projectivity and injectivity.
It therefore gives a bijection between the isomorphism classes of indecomposable modules occurring in ${_A\pandi}$ and those occurring in ${_B\pandi}$.
For every indecomposable $X\in{_A\pandi}$ and every $r,m\>= 1$, define
\begin{align} \label{eq:Fint-corresp}
  F_{X,r}^{m}\mapsto F_{\mathcal E(X),r}^{m}.
\end{align}
By Lemma \ref{lemm:cyclic int coeff component}, \eqref{eq:Fint-corresp} defines the isomorphism from $\langle F_{X,r}\rangle$ to
$\langle F_{\mathcal E(X),r}\rangle$.
Products belonging to distinct pairs $(X,r)$ vanish on both sides by \eqref{eq:orth int coeff product}.
The above componentwise isomorphisms therefore provide, by Proposition \ref{prop:orth decomposition icring},
to an isomorphism $\pandiring(A)\To{\cong}\pandiring(B)$ as required.
\end{proof}

%

\section{Resolutions of modules} \label{sect:resol}

\textsl{This section is devoted to linking integral profiles with homological dimensions. We characterise projective and injective modules via integral profiles. We introduce total profiles attached to minimal (co)resolutions, and construct an integral coefficient ring that enables comparisons among profiles of different modules.}

\subsection{Projective resolutions of modules} \label{subsect:proj resol}

The same module will have different projective resolutions. Choosing the shortest projective resolution
\[ \cdots \To{} P_M^n \To{p_M^n} \cdots \To{p_M^2} P_M^1 \To{p_M^1} P_M^0 \To{p_M^0}  M \to 0, \]
say \defines{minimal projective resolution}, of $M$, we have that its projection dimension $\pdim_A M$ is $d$ if $P_M^n\ne 0$ holds for all $n\=< d$ and $P_M^n=0$ holds for all $n>d$, see \cite[Section 3.1]{R1979},
and in this case, we call $\Ker(p_M^i)$ is the \defines{$(i+1)$-th syzygy} of $M$ and write it is as $\Omega^{i+1}(M)$ (with the convention that $\Omega^0(M)=M$),
and each epimorphism $\tilde{p}_{i+1}: P_M^{i+1}\to \Omega^{i+1}(M)$ induced by $p_M^{i+1}$ is called a \defines{projective cover} of $\Omega^{i+1}(M)$.
Obviously, $\Omega^{d+1}(M)=0$ if and only if $\pdim_A M\=< d$.
Thus, it is particularly important to determine whether a syzygy is a projective module.

\begin{definition} \label{def:tot proj int prof} \rm
For a left $A$-module $M\in{_A\modcat}$, we define its \defines{total projective integral profile} by
\[\IntProf_{\II_A}(\pmb{P}_M^\bullet;x):= \sum_{n\>=0}\bigg(\IntProf_{\II_A}P_M^n\bigg)x^n\]
and define its \defines{embedded total projective integral profile} by
\[\IntProfEmb_{\II_A}(\pmb{P}_M^\bullet;x) :=
  \sum_{n\>= 0} \mathfrak{Emb}_{P_M^n} \left( \IntProf_{\II_A}P_M^n \right)x^n\]
\end{definition}

Sometimes, for simplicity, the embedded total projective integral profile of the minimal projective resolution $\pmb{P}_M^{\bullet}$ of $M$ is also referred to simply as the embedded total projective integral profile of $M$.

\begin{lemma}\label{lemm:import subalg 1}
Let $M\in{_A\modcat}$, then $\displaystyle \IntProf_{\II_A}(\pmb{P}_M^\bullet;x) \widetilde{\in} \pandiring(A)[[x]]$.
\end{lemma}

\begin{proof}
For every $n\>=0$, we have $P_M^n\in{_A\proj}\subseteq{_A\pandi}$.
By the definition of $\pandiring(A)$, we have $\mathfrak{Emb}_{P_M^n}(\IntProf_{\II_A}P_M^n) \in \pandiring(A)$,
Hence \[\sum_{n\>=0}\mathfrak{Emb}_{P_M^n}\left( \IntProf_{\II_A}P_M^n \right)x^n\in\pandiring(A)[[x]],\]
which completes the proof.
\end{proof}

For a finite-dimensional complex algebra $A=\CC\Q/\I$,
each vertex $v\in\Q_0$ corresponds to a simple module $S_v$
(see \cite[Section III.2, Lemma 2.1]{ASS2006}).
It has a unique projective cover (up to isomorphism)
\[ P(v) \to S(v) \to 0, \]
where $P(v)$ is the indecomposable projective module $Ae_v$
corresponding to the path $e_v$ of length zero,
and $S(v)\cong P(v)/\rad P(v)$ \cite[Section III.2, Lemma 2.4]{ASS2006}.
All indecomposable projective left $A$-modules are isomorphic to the form $Ae_v$, and
$\bigoplus\limits_{v\in\Q_0} Ae_v
= A\cdot\sum\limits_{v\in\Q_0} e_v
= A\cdot 1_A = A$. By the definition of $\FF\Q/\I$,
it is well-known that $Ae_v$ is generated by all paths in $(\Q,\I)$ ending with $v$,
i.e., $\basis{Ae_v} = \{\wp+\I \in (\Q,\I) : \mathfrak{t}(\wp)=v\}$.
Without loss of generality, assume $\basis{Ae_v} = \{p_{v,j} : 1\=< j\=< d_v=\dim Ae_v\}$,
then each $a_ip_{v,j}$ has a linear decomposition $a_ip_{v,j} = \sum\limits_{k=1}^{d_v} \lambda^{(v)}_{i,j,k}p_{v,k}$,
and obtain
\[ (a_ip_{v,j})_{1\=< j\=< d_v} =
\left(
\begin{matrix}
  p_{v,1} & p_{v,2} & \cdots & p_{v,d_v}
\end{matrix}
\right)
\left(
\begin{matrix}
  \lambda^{(v)}_{i,1,1}
& \lambda^{(v)}_{i,2,1}
& \cdots
& \lambda^{(v)}_{i,d_v,1} \\
  \lambda^{(v)}_{i,1,2}
& \lambda^{(v)}_{i,2,2}
& \cdots
& \lambda^{(v)}_{i,d_v,2} \\
  \vdots
& \vdots
&
& \vdots \\
  \lambda^{(v)}_{i,1,d_v}
& \lambda^{(v)}_{i,2,d_v}
& \cdots
& \lambda^{(v)}_{i,d_v,d_v}
\end{matrix}
\right).\]
Denote by $L_i^{(v)}$ the above matrix $(\lambda^{(v)}_{i,j,k})_{1\=<k,j\=<d_v}$.

We use $X\sim Y$ to represent two matrixes $X$ and $Y$ are similar.
In our paper, the integral profile of a left $A$-module $M\in{_A\modcat}$ is a matrix with a parameter vector $\pmb{t}$.
Therefore, the matrix similarity considered in this article mostly refers to the similarity of matrices with parameter vectors.

\begin{definition}\rm \label{def:similar}
We call two matrixes $X(\pmb{t})$ and $Y(\pmb{t})$ with parameter vector $\pmb{t}$ are similar if and only if there is an invertible matrix $H$ such that $X(\pmb{t})H = HY(\pmb{t})$.
Furthermore, if $X(\pmb{t}) \sim X(\pmb{t})_1\oplus X(\pmb{t})_2$, we use $X(\pmb{t})_1 \lesssim_{\oplus} X(\pmb{t})$ to represent $X(\pmb{t})_1$ is similar to a direct summand of $X(\pmb{t})$.
\end{definition}

Then matrix similarity in the classical sense is a special case of similarity given in Definition \ref{def:similar}.
Now we provide the following result which describes all projective modules.

\begin{theorem} \label{thm:mat char-proj mod}
Assume that $A=\CC\Q/\I$ is a finite-dimensional complex algebra.
Fix a basis of a given left $A$-module $M\in{_A\modcat}$. Then $M$ is projective if and only if
\begin{align}\label{eq:mat char-proj mod}
 \IntProf_{\II_A} M
\sim \frac{\mu_{\II_A(\pmb{t})}(\II_A(\pmb{t}))}{2} \cdot
    \bigoplus_{v\in\Q_0}
    \bigg(
      \sum_{i=1}^{\dimA}t_iL_i^{(v)}
    \bigg)^{\oplus m_v},
\end{align}
where all $m_v$ are non-negative integers such that $\sum\limits_{v\in\Q_0}p_vm_v=\dim_{\CC}(M)$ {\rm(}$p_v=\dim_{\CC}P(v)${\rm)}.
\end{theorem}

\begin{proof}
Suppose first that $M$ is projective. Since $A$ is a finite-dimensional basic algebra, the modules $P(v)=Ae_v$, $v\in\mathcal Q_0$, form a complete set of pairwise non-isomorphic indecomposable projective left $A$-modules.
Hence, by the Krull--Schmidt theorem (see for example \cite[Chapter II]{ARS1995}, \cite[Section 12]{AF1992}, and \cite{Krause2015}),
$M \cong \bigoplus\limits_{v\in\mathcal Q_0}P(v)^{\oplus m_v}$ for some $m_v\in\NN$.
Let $H$ be the matrix of such a left $A$-module isomorphism with respect to the chosen bases. Then, for every $a_i\in\basis A$, we have
\[ (\rho_M(a_i))_{m\times m}
= H\bigg(
    \bigoplus_{v\in\Q_0} (L_i^{(v)})^{\oplus m_v}
  \bigg) H^{-1}, \]
where $m=\dim_{\CC}M=\sum\limits_{v\in\Q_0}p_vm_v$.
By Proposition \ref{prop:param-int-proj mod}, we obtain
\begin{align*}
\IntProf_{\II_A}M(\pmb{t})
& = \frac{\mu_{\II_A(\pmb{t})}(\II_A(\pmb{t}))}{2}
    \sum_{i=1}^{\dimA}t_i\cdot (\rho_M(a_i))_{m\times m} \\
& = \frac{\mu_{\II_A(\pmb{t})}(\II_A(\pmb{t}))}{2}
    \sum_{i=1}^{\dimA}t_i \cdot
    H \bigg(
        \bigoplus_{v\in\Q_0} (L_i^{(v)})^{\oplus m_v}
      \bigg) H^{-1} \\
&\sim \frac{\mu_{\II_A(\pmb{t})}(\II_A(\pmb{t}))}{2}
    \sum_{i=1}^{\dimA}t_i
        \bigoplus_{v\in\Q_0} (L_i^{(v)})^{\oplus m_v} \\
&= \frac{\mu_{\II_A(\pmb{t})}(\II_A(\pmb{t}))}{2}
    \bigoplus_{v\in\Q_0}\bigg(\sum_{i=1}^{\dimA}t_iL_i^{(v)}\bigg)^{\oplus m_v}.
\end{align*}

Conversely, suppose that \eqref{eq:mat char-proj mod} holds for some fixed invertible matrix $H$ independent of $\pmb{t}$. Comparison of the recovered coefficients gives
\[ (\rho_M(a_i))_{m\times m}
= H \bigg(
      \bigoplus_{v\in\Q_0} (L_i^{(v)})^{\oplus m_v}
    \bigg) H^{-1} \]
for every $i$. Therefore $H$ defines an $A$-module isomorphism $\bigoplus\limits_{v\in\Q_0}P(v)^{\oplus m_v} \To{H} M$ by Theorem \ref{thm:param-int-module iso}.
\end{proof}

Given a nonzero left $A$-module $M\in{_A\modcat}$, let $\pmb{P}_M^\bullet = (P_M^n,p_M^n)_{n\>=0}$ be its minimal projective resolution. Put $\mathscr{P}_M := \prod\limits_{n\>=0} \operatorname{Map} ((\RR^{\>=0})^{\times\dimA},\End_{\CC}(P_M^n))$.
The image of
\[\mathfrak{Emb}_{P_M^n} : \operatorname{Map}( (\RR^{\>=0})^{\times\dimA}, \End_{\CC}(P_M^n)) \to \widetilde{\pandiring}(A),\]
can be canonically identified with the $n$-th component of $\mathscr{P}_M$.
It follows that $\mathfrak{Emb}_{P_M^n}$ admits an injection from
$\operatorname{Map} ( (\RR^{\>=0})^{\times\dimA},\End_{\CC}(P_M^n) )$ to $\mathscr{P}_M$
which is a canonical embedding into the $n$-th component.
To make the notation appear simple, we still write this embedding as
\[\mathfrak{Emb}_{P_M^n}: \operatorname{Map} ( (\RR^{\>= 0})^{\times\dimA},\End_{\CC}(P_M^n) ) \to \mathscr{P}_M. \]
This will not cause confusion for us.


The following result provides a description of projective dimensions of modules.

\begin{lemma}\label{lemm:pdim-int 0812}
For any $0\ne M\in{_A\modcat}$, $\pdim_A M<\infty$ if and only if $\IntProf_{\II_A}(\pmb{P}_M^\bullet;x)$ is a polynomial in the formal power series ring $\mathscr{P}_M[[x]]$.
In this case,
\[ \pdim_A M = \deg_x \IntProf_{\II_A}(\pmb{P}_M^\bullet;x).\]
\end{lemma}

\begin{proof}
Suppose first that $\pdim_A M=d<\infty$.
Since $\pmb{P}_M^\bullet$ is a minimal projective resolution,
we have $P_M^n=0$ for all $n>d$ and $P_M^d\ne0$. Therefore,
\[\IntProf_{\II_A}(\pmb{P}_M^\bullet;x) =
\sum_{n=0}^{d} \mathfrak{Emb}_{P_M^n}(\IntProf_{\II_A}P_M^n)x^n\]
is a polynomial of degree at most $d$.
We claim that $\IntProf_{\II_A}P_M^d\ne0$.
Indeed, if its integral profile were zero, then we obtain
$\rho_{P_M^d}(a)=0$ for every $a\in A$. In particular, $\rho_{P_M^d}(1_A)=0$, contradicting
$\rho_{P_M^d}(1_A) = \ident_{P_M^d} \ne 0$. Thus the coefficient of $x^d$ is nonzero, and hence
$\deg_x \IntProf_{\II_A}(\pmb{P}_M^\bullet;x) = d$.

Conversely, suppose that $\IntProf_{\II_A}(\pmb{P}_M^\bullet;x)$ is a polynomial, and put
$d := \deg_x \IntProf_{\II_A}(\pmb{P}_M^\bullet;x)$ Then $\IntProf_{\II_A}P_M^n=0$ for every $n>d$.
By the coefficient-recovery property, this implies $P_M^n=0$ for every $n>d$.
Moreover, since the coefficient of $x^d$ is nonzero, we have $P_M^d\ne0$.
Therefore the minimal projective resolution terminates precisely in
degree $d$, and consequently $\pdim_A M=d<\infty$.
\end{proof}

\subsection{Injective coresolutions of modules} \label{subsect:inj coresol}

The same module will have different injective coresolutions.
Choosing the shortest injective coresolution
\[ 0 \To{} M \To{i^M_0} I^M_0 \To{i^M_1} I^M_1 \To{i^M_2} I^M_2 \To{} \cdots, \]
say \defines{minimal injective coresolution}, of $M$, we have that its injective dimension $\idim_A M$ is $d$ if $I^M_n\ne 0$ holds for all $n\=< d$ and $I^M_n=0$ holds for all $n>d$, see \cite[Section 3.2]{R1979},
and in this case, we call $\Coker(i^M_n)$ is the \defines{$(n+1)$-th cosyzygy} of $M$ and write it as $\mho_{n+1}(M)$ (with the convention that $\mho_0(M)=M$),
and each monomorphism $\tilde{i}^M_{n+1}: \mho_{n+1}(M) \to I^M_{n+1}$ induced by $i^M_{n+1}$ is called an \defines{injective envelope} of $\mho_{n+1}(M)$.
Obviously, $\mho_{d+1}(M)=0$ if and only if $\idim_A M\=< d$.
Thus, it is particularly important to determine whether a cosyzygy is an injective module.

\begin{definition} \rm
For a left $A$-module $M\in{_A\modcat}$, we define its \defines{total injective integral profile} by
\[\IntProf_{\II_A}(\pmb{I}_\bullet^M;x) :=
  \sum_{n\>= 0} \left( \IntProf_{\II_A}I^M_n \right)x^n. \]
and its \defines{embedded total injective integral profile} by
\[\IntProfEmb_{\II_A}(\pmb{I}_\bullet^M;x) :=
  \sum_{n\>= 0} \mathfrak{Emb}_{I_n^M} \left( \IntProf_{\II_A}I^M_n \right)x^n. \]
\end{definition}

Sometimes, for simplicity, the embedded total projective integral profile of the minimal injective resolution $\pmb{I}_{\bullet}^M$ of $M$ is also referred to simply as the embedded total injective integral profile of $M$.

\begin{lemma}\label{lemm:import subalg 2}
Let $M\in{_A\modcat}$, then $\displaystyle \IntProf_{\II_A}(\pmb{I}_\bullet^M;x) \widetilde{\in} \pandiring(A)[[x]]$.
\end{lemma}

\begin{proof}
For every $n\>=0$, we have $I^M_n\in{_A\inj}\subseteq{_A\pandi}$,
and therefore $\mathfrak{Emb}_{I^M_n}(\IntProf_{\II_A}I^M_n) \in \pandiring(A)$, Thus,
\[ \sum_{n\>=0} \mathfrak{Emb}_{I^M_n} \left( \IntProf_{\II_A}I^M_n \right)x^n \in \pandiring(A)[[x]],\]
which completes the proof.
\end{proof}

Each simple module $S_v$ over $A=\FF\Q/\I$ has a unique injective envelope (up to isomorphism)
\[ 0 \to S(v) \to I(v), \]
where $I(v)$ is the indecomposable injective module $D(e_v A)$
corresponding to the path $e_v$ of length zero,
and $S(v)\cong \soc I(v)$ \cite[Section III.2, Lemma 2.6]{ASS2006}.
All indecomposable injective left $A$-modules are isomorphic to of the form $D(e_vA)$, and
$\bigoplus\limits_{v\in\Q_0} D(e_vA)$ is the injective cogenerator of the module category.
By the definition of $\FF\Q/\I$,
it is well-known that $D(e_vA)$ is the $\CC$-dual of the right $A$-module $e_vA$, which is generated by all paths in $(\Q,\I)$ starting with $v$,
i.e., $\basis{e_vA} = \{\wp+\I \in (\Q,\I) : \mathfrak{s}(\wp)=v\}$.
Let $\{q_{v,1},\ldots,q_{v,{_v\!d}}\}$ be a basis of $e_vA$ with ${_v\!d}=\dim e_vA=\dim D(e_vA)$, and let $\{q_{v,1}^*,\ldots,q_{v,{_v\!d}}^*\}$ be the corresponding dual basis of $D(e_vA)$.
Then each $q_{v,j} a_i$ has a linear decomposition $q_{v,j} a_i = \sum\limits_{k=1}^{{_v\!d}} \mu^{(v)}_{i,j,k} q_{v,k}$,
and the left $A$-action on $D(e_vA)$ is given by $(a_i\cdot q_{v,k}^*)(q_{v,j}) = q_{v,k}^*(q_{v,j}a_i) = \mu^{(v)}_{i,j,k}$.
Thus $a_i\cdot q_{v,k}^* = \sum\limits_{j=1}^{{_v\!d}} \mu^{(v)}_{i,j,k} q_{v,j}^*$, and we obtain
\[ (a_i\cdot q_{v,k}^*)_{1\=< k\=< {_v\!d}} =
\left(
\begin{matrix}
  q_{v,1}^* & q_{v,2}^* & \cdots & q_{v,{_v\!d}}^*
\end{matrix}
\right)
\left(
\begin{matrix}
  \mu^{(v)}_{i,1,1}
& \mu^{(v)}_{i,1,2}
& \cdots
& \mu^{(v)}_{i,1,{_v\!d}} \\
  \mu^{(v)}_{i,2,1}
& \mu^{(v)}_{i,2,2}
& \cdots
& \mu^{(v)}_{i,2,{_v\!d}} \\
  \vdots
& \vdots
&
& \vdots \\
  \mu^{(v)}_{i,{_v\!d},1}
& \mu^{(v)}_{i,{_v\!d},2}
& \cdots
& \mu^{(v)}_{i,{_v\!d},{_v\!d}}
\end{matrix}
\right).\]
Denote by $R_i^{(v)}$ the matrix $(\mu^{(v)}_{i,j,k})_{1\=< j,k \=< {_v\!d}}$.
Then $R_i^{(v)} = (\rho_{D(e_vA)}(a_i))_{{_v\!d}\times {_v\!d}}$ with respect to the dual basis $(q_{v,1}^*,\ldots,q_{v,{_v\!d}}^*)$.

Now we provide the following result which describes all injective modules.

\begin{theorem}\label{thm:mat char-inj mod}
Assume that $A=\CC\Q/\I$ is a finite-dimensional complex algebra.
Fix a basis of a given left $A$-module $M\in{_A\modcat}$. Then $M$ is injective if and only if
\[
 \IntProf_{\II_A} M
\sim \frac{\mu_{\II_A(\pmb{t})}(\II_A(\pmb{t}))}{2} \cdot
    \bigoplus_{v\in\Q_0}
    \bigg(
      \sum_{i=1}^{\dimA}t_iR_i^{(v)}
    \bigg)^{\oplus n_v},
\]
where all $n_v$ are non-negative integers such that $\sum\limits_{v\in\Q_0} c_vn_v = \dim_{\CC}(M)$ {\rm(}$c_v=\dim_{\CC}I(v)${\rm)}.
\end{theorem}

\begin{proof}
Suppose first that $M$ is injective. Since $A$ is a finite-dimensional basic algebra, the modules $I(v)=D(e_vA)$, $v\in\mathcal Q_0$, form a complete set of pairwise non-isomorphic indecomposable injective left $A$-modules.
Hence, by the Krull--Schmidt theorem,
$M \cong \bigoplus\limits_{v\in\mathcal Q_0} I(v)^{\oplus n_v}$ for some $n_v\in\NN$.
Let $H$ be the matrix of such a left $A$-module isomorphism with respect to the chosen bases. Then, for every $a_i\in\basis A$, we have
\[ (\rho_M(a_i))_{m\times m}
= H
  \bigg(
    \bigoplus_{v\in\Q_0} (R_i^{(v)})^{\oplus n_v}
  \bigg) H^{-1}, \]
where $m=\dim_{\CC}M=\sum\limits_{v\in\Q_0} c_vn_v$.
By Proposition \ref{prop:param-int-inj mod}, we obtain
\begin{align*}
\IntProf_{\II_A}M(\pmb{t})
& = \frac{\mu_{\II_A(\pmb{t})}(\II_A(\pmb{t}))}{2}
    \sum_{i=1}^{\dimA}t_i\cdot (\rho_M(a_i))_{m\times m} \\
& = \frac{\mu_{\II_A(\pmb{t})}(\II_A(\pmb{t}))}{2}
    \sum_{i=1}^{\dimA}t_i \cdot
    H \bigg(
        \bigoplus_{v\in\Q_0} (R_i^{(v)})^{\oplus n_v}
      \bigg) H^{-1} \\
&\sim \frac{\mu_{\II_A(\pmb{t})}(\II_A(\pmb{t}))}{2}
    \sum_{i=1}^{\dimA}t_i
        \bigoplus_{v\in\Q_0} (R_i^{(v)})^{\oplus n_v} \\
&= \frac{\mu_{\II_A(\pmb{t})}(\II_A(\pmb{t}))}{2}
    \bigoplus_{v\in\Q_0}\bigg(\sum_{i=1}^{\dimA}t_iR_i^{(v)}\bigg)^{\oplus n_v}.
\end{align*}

Conversely, suppose that the above matrix similarity holds for some fixed invertible matrix $H$ independent of $\pmb{t}$. Comparison of the recovered coefficients gives
\[ (\rho_M(a_i))_{m\times m}
= H \bigg(
      \bigoplus_{v\in\Q_0} (R_i^{(v)})^{\oplus n_v}
    \bigg) H^{-1} \]
for every $i$. Therefore $H$ defines an $A$-module isomorphism $\bigoplus\limits_{v\in\Q_0}I(v)^{\oplus n_v} \To{H} M$ by Theorem \ref{thm:param-int-module iso}.
\end{proof}

%


Given a nonzero finite-dimensional left $A$-module $M\in{_A\modcat}$,
let $\pmb{I}_{\bullet}^M = ( I^M_n,d_M^n )_{n\ge 0}$ be its minimal injective coresolution.
Put $\mathscr{I}_M := \prod\limits_{n\ge 0} \operatorname{Map} ((\RR^{\ge 0})^{\times d_A},\End_{\CC}(I^M_n))$.
The image of
\[\mathfrak{Emb}_{I^M_n} : \operatorname{Map}( (\RR^{\ge 0})^{\times d_A}, \End_{\CC}(I^M_n)) \to \widetilde{\pandiring}(A),\]
can be canonically identified with the $n$-th component of $\mathscr{I}_M$.
It follows that $\mathfrak{Emb}_{I^M_n}$ admits an injection from
$\operatorname{Map} ( (\RR^{\ge 0})^{\times d_A},\End_{\CC}(I^M_n) )$ to $\mathscr{I}_M$
which is a canonical embedding into the $n$-th component.
To make the notation appear simple, we still write this embedding as
\[\mathfrak{Emb}_{I^M_n}: \operatorname{Map} ( (\RR^{\ge 0})^{\times d_A},\End_{\CC}(I^M_n) ) \to \mathscr{I}_M. \]
This will not cause confusion for us.

A parallel result to Lemma \ref{lemm:pdim-int 0812} is as follows.

\begin{lemma}\label{lemm:idim-int 0812}
For any $0\ne M\in{_A\modcat}$, $\idim_A M<\infty$ if and only if $\displaystyle \IntProf_{\II_A} (\pmb{I}^M_{\bullet}; x)$ is a polynomial in the formal power series ring $\mathscr{I}_M[[x]]$.
In this case,\[ \idim_A M = \deg_x \IntProf_{\II_A} (\pmb{I}^M_{\bullet}; x).\]
\end{lemma}

\begin{proof}
Suppose first that $\idim_A M=d<\infty$.
Since $\pmb{I}_{\bullet}^M$ is a minimal injective coresolution, we have $I^M_n=0$ for all $n>d$ and $I^M_d\ne 0$.
Therefore
\[ \IntProf_{\II_A}(\pmb{I}_{\bullet}^M;x)
= \sum_{n=0}^{d} \mathfrak{Emb}_{I^M_n} ( \IntProf_{\II_A}I^M_n )x^n\]
is a polynomial of degree at most $d$.
We claim that $\IntProf_{\II_A}I^M_d\ne 0$. Indeed, if its integral profile were zero,
then we obtain $\rho_{I^M_d}(a)=0$ for every $a\in A$. In particular,
$\rho_{I^M_d}(1_A)=0$ contradicting $\rho_{I^M_d}(1_A)=\ident_{I^M_d}\ne 0$.
Thus the coefficient of $x^d$ is nonzero, and hence $\deg_x \IntProf_{\II_A} (\pmb{I}_{\bullet}^M;x)=d$.

Conversely, suppose that $\IntProf_{\II_A}(\pmb{I}_{\bullet}^M;x)$ is a polynomial, and put
$d:= \deg_x\IntProf_{\II_A}(\pmb{I}_{\bullet}^M;x)$. Then $\IntProf_{\II_A}I^M_n=0$ for every $n>d$.
By the coefficient-recovery property, this implies $I^M_n=0$ for every $n>d$.
Moreover, since the coefficient of $x^d$ is nonzero, we have $I^M_d\ne 0$.
Therefore the minimal injective coresolution terminates precisely in degree $d$, and consequently $\idim_A M=d<\infty$.
\end{proof}

\section{Homological dimensions vs. integral profiles} \label{sect:homodim intdescrib}

\textsl{This section uses the integral coefficient ring to characterise homological dimensions.
}

\subsection{Homological dimensions vs. total integral profiles} \label{subsect:gldim intdescrib}

\subsubsection{Global dimensions vs. total integral profiles} \label{subsubsect:gldim intdescrib}

For any finite-dimensional basic complex algebra $A=\CC\Q/\I$, the following result shows its global dimension.

\begin{proposition}\label{prop:gldim-int}
The following statements are equivalent:
\begin{enumerate}[label={\rm(\arabic*)}]
  \item $\gldim A < \infty$;
    \label{prop:gldim-int 1}
  \item $\displaystyle \sup_{v\in\Q_0} \deg_x\IntProf_{\II_A}(\pmb{P}_{S_v}^\bullet;x) < \infty$;
    \label{prop:gldim-int 2}
  \item $\IntProf_{\II_A}(\pmb{P}_{\top(A)}^\bullet;x)$ is a polynomial in the formal power series ring $\mathscr{P}_{\top(A)}[[x]]$, i.e.,
    \[\IntProf_{\II_A}(\pmb{P}_{\top(A)}^\bullet;x)\in \mathscr{P}_{\top(A)}[x];\]
    \label{prop:gldim-int 3}
  \item $\displaystyle \sup_{v\in\Q_0} \deg_x\IntProf_{\II_A}(\pmb{I}_{\bullet}^{S_v};x) < \infty$;
    \label{prop:gldim-int 4}
  \item $\IntProf_{\II_A}(\pmb{I}_{\bullet}^{\top(A)};x)$ is a polynomial in the formal power series ring $\mathscr{I}_{\top(A)}[[x]]$, i.e.,
    \[\IntProf_{\II_A}(\pmb{I}_{\bullet}^{\top(A)};x)\in \mathscr{I}_{\top(A)}[x];\]
    \label{prop:gldim-int 5}
\end{enumerate}
\end{proposition}

\begin{proof}
We prove the equivalence in a cycle.

\ref{prop:gldim-int 1} $\Rightarrow$ \ref{prop:gldim-int 2}:
Assume $\gldim A < \infty$. By the standard characterisation of global dimension,
$\gldim A = \max\limits_{v\in\Q_0} \pdim_A S_v.$
Thus $\pdim_A S_v < \infty$ for every $v\in\Q_0$. By Lemma \ref{lemm:pdim-int 0812}, each
$\IntProf_{\II_A}(\pmb{P}_{S_v}^\bullet;x)$ is a polynomial. Therefore the supremum of their degrees is finite, proving (2).

\ref{prop:gldim-int 2} $\Rightarrow$ \ref{prop:gldim-int 3}:
Assume (2). Since $\top(A) \cong \bigoplus\limits_{v\in\Q_0} S_v$, and projective resolutions commute with finite direct sums,
$\pmb{P}_{\top(A)}^\bullet \cong \bigoplus\limits_{v\in\Q_0} \pmb{P}_{S_v}^{\bullet}$.
Taking total projective integral profiles gives
\[ \IntProf_{\II_A}(\pmb{P}_{\top(A)}^\bullet;x) = \bigoplus_{v\in\Q_0} \IntProf_{\II_A}(\pmb{P}_{S_v}^\bullet;x).\]
By (2), each summand is a polynomial of finite degree, so the finite direct sum is also a polynomial. Hence (3) holds.

\ref{prop:gldim-int 3} $\Rightarrow$ \ref{prop:gldim-int 1}:
Suppose that $\IntProf_{\II_A}(\pmb{P}_{\top(A)}^\bullet;x)$ is a polynomial and let $d$ be its degree. Then for each $v\in\Q_0$, the summand $\IntProf_{\II_A}(\pmb{P}_{S_v}^\bullet;x)$ is a direct summand of a polynomial, hence itself a polynomial of degree at most $d$. By Lemma \ref{lemm:pdim-int 0812}, this implies $\pdim_A S_v \=< d$ for every $v$.
Therefore $\gldim A = \max\limits_{v\in\Q_0} \pdim_A S_v \=< d < \infty$. Thus (1) holds.

We can prove \ref{prop:gldim-int 1} $\Rightarrow$ \ref{prop:gldim-int 4} $\Rightarrow$ \ref{prop:gldim-int 5} $\Rightarrow$ \ref{prop:gldim-int 1} in a dual way.
\end{proof}

For a formal power series $F(x)\in\pandiring(A)[[x]]$, we put $\deg_xF:=\infty$ whenever $F(x)\notin\pandiring(A)[x]$.
Then by Lemma \ref{lemm:import subalg 1} and Proposition \ref{prop:gldim-int}, we obtain the following theorem, immediately.

\begin{theorem}\label{thm:gldim-int}
Let $A=\CC\Q/\I$ be a finite-dimensional complex algebra, and $\pandiring(A)$ be its integral coefficient ring.
The following statements are equivalent:
\begin{enumerate}[label={\rm(\arabic*)}]
  \item $\gldim A < \infty$;
    \label{thm:gldim-int 1}
  \item $\displaystyle \sup_{v\in\Q_0} \deg_x\IntProf_{\II_A}(\pmb{P}_{S_v}^\bullet;x) < \infty$;
    \label{thm:gldim-int 2}
  \item $\displaystyle \IntProf_{\II_A}(\pmb{P}_{\top(A)}^\bullet;x) \,\widetilde{\in}\, \pandiring(A)[x]$;
    \label{thm:gldim-int 3}
  \item $\displaystyle \sup_{v\in\Q_0} \deg_x\IntProf_{\II_A}(\pmb{I}_{\bullet}^{S_v};x) < \infty$;
    \label{thm:gldim-int 4}
  \item $\displaystyle \IntProf_{\II_A}(\pmb{I}_{\bullet}^{\top(A)};x) \,\widetilde{\in}\, \pandiring(A)[x]$.
    \label{thm:gldim-int 5}
\end{enumerate}
\end{theorem}

\subsubsection{Finitistic dimensions vs. total integral profiles} \label{subsubsect:findim intdescrib}

The finitistic dimension of a finite-dimension complex algebra $A$ we considered in this subsection is the little finitistic dimension
\[ \findim A = \sup\left\{ \pdim_A M: M\in{_A\modcat} \text{ and } \pdim_A M<\infty \right\}.\]

\begin{definition}\rm
Define the \defines{finitistic projective integral profile set} of $A$ by
\[ \fpipset_A :=
\left\{ \IntProfEmb_{\II_A} (\pmb{P}_M^\bullet;x):
    0 \ne M\in{_A\modcat} \text{ and }
    \IntProf_{\II_A} (\pmb{P}_M^\bullet;x) \, \widetilde{\in}\, \pandiring(A)[x]
\right\}.
\]
\end{definition}

Integration provides a simple characterization of finitistic dimension, see the following statement.

\begin{lemma}\label{lemm:findim-int}
$\displaystyle \findim A = \sup \{\deg_x F(x): F(x)\in\mathfrak F_A^{\mathrm{proj}}\}$.
\end{lemma}

\begin{proof}
For every nonzero $M\in{_A\modcat}$, we have
\[ \IntProf_{\II_A}(\pmb{P}_M^\bullet;x) \,\widetilde{\in}\, \pandiring(A)[x]
\text{ if and only if } \pdim_A M = \deg_x \IntProf_{\II_A}(\pmb{P}_M^\bullet;x) < \infty\]
By Lemma \ref{lemm:pdim-int 0812}, taking the supremum over all modules of finite projective dimension
gives the assertion.
\end{proof}

\begin{theorem} \label{thm:findim-int}
The following statements are equivalent:
\begin{enumerate}[label={\rm(\arabic*)}]
  \item\label{thm:findim-int 1}
  $\findim A<\infty$;

  \item\label{thm:findim-int 2}
  there is $d\in\NN$ such that, for every $0\ne M\in{_A\modcat}$,
  \[ \IntProf_{\II_A} (\pmb{P}_M^\bullet;x) \,\widetilde{\in}\, \pandiring(A)[x]
   \text{ implies } \deg_x \IntProf_{\II_A} (\pmb{P}_M^\bullet;x) \=< d; \]

  \item\label{thm:findim-int 3}
  $ \fpipset_A \subseteq \pandiring(A)[x]_{\=< d}$ for some $d\in\NN$;

  \item\label{thm:findim-int 4}
  there is $d\in\NN$ such that, for every $0\ne M\in{_A\modcat}$ whose embedded total projective integral profile
  is a polynomial, one has
  \[\mathfrak{Emb}_{P_M^n}(\IntProf_{\II_A}P_M^n) =0 \qquad \text{for every } n>d.\]
\end{enumerate}
\end{theorem}

\begin{proof}
For every $0\ne M\in{_A\modcat}$, Lemma \ref{lemm:pdim-int 0812}, together with the injectivity of the coefficient
embeddings, gives $\pdim_A M<\infty$ if and only if $\IntProf_{\II_A} (\pmb{P}_M^\bullet;x) \ \widetilde{\in}\ \pandiring(A)[x]$.
In this case,
\begin{align} \label{eq:pdim-embedded-degree}
\pdim_A M = \deg_x \IntProf_{\II_A} (\pmb{P}_M^\bullet;x) = \deg_x \IntProfEmb_{\II_A} (\pmb{P}_M^\bullet;x).
\end{align}

\ref{thm:findim-int 1} $\Rightarrow$ \ref{thm:findim-int 2}:
Assume $\findim A = d$. If $M$ satisfies $\IntProf_{\II_A} (\pmb{P}_M^\bullet;x) \ \widetilde{\in}\ \pandiring(A)[x]$,
then $\pdim_A M<\infty$, and, by Lemma \ref{lemm:pdim-int 0812}, we obtain
\[ \deg_x \IntProf_{\II_A} (\pmb{P}_M^\bullet;x) = \pdim_A M \=<  \findim A = d.\]
Thus \ref{thm:findim-int 2} holds.

\ref{thm:findim-int 2} $\Rightarrow$ \ref{thm:findim-int 1}:
Suppose that \ref{thm:findim-int 2} holds for some $d\in\NN$.
For an arbitrary module $0\ne M\in{_A\modcat}$ with a finite projective dimension,
we have known that its embedded total projective integral profile is a polynomial by Lemmas \ref{lemm:import subalg 1} and \ref{lemm:pdim-int 0812}.
Therefore, \eqref{eq:pdim-embedded-degree} implies
\[ \pdim_A M = \deg_x \IntProf_{\II_A} (\pmb{P}_M^\bullet;x) \=< d. \]
By the arbitrariness of $M$, we have $\findim A\=< d<\infty$.
Hence \ref{thm:findim-int 1} holds.

\ref{thm:findim-int 2} $\Leftrightarrow$ \ref{thm:findim-int 3}:
Condition \ref{thm:findim-int 2} says precisely that there exists $d\in\NN$ such that every polynomial in $\fpipset_A$ has
degree at most $d$. This is equivalent to $\fpipset_A \subseteq \pandiring(A)[x]_{\=<d}$.
Thus \ref{thm:findim-int 2} and
\ref{thm:findim-int 3} are equivalent.

\ref{thm:findim-int 2} $\Leftrightarrow$ \ref{thm:findim-int 4}:
Notice that
\[ \IntProfEmb_{\II_A} (\pmb{P}_M^\bullet;x) =
   \sum_{n\>=0} \mathfrak{Emb}_{P_M^n} \left( \IntProf_{\II_A}P_M^n \right)x^n. \]
If this formal power series is a polynomial, then
\begin{center}
$\displaystyle \deg_x \IntProfEmb_{\II_A} (\pmb{P}_M^\bullet;x) \=< d$
if and only if $\displaystyle \mathfrak{Emb}_{P_M^n} \left(\IntProf_{\II_A}P_M^n \right) = 0$
\end{center}
for every $n>d$.
Hence \ref{thm:findim-int 2} and \ref{thm:findim-int 4} are equivalent.
\end{proof}

\begin{remark} \rm
If these equivalent conditions given in Theorem \ref{thm:findim-int} hold,
then $\findim A$ is the least integer $d$ satisfying any of
Theorem \ref{thm:findim-int} \ref{thm:findim-int 2}--\ref{thm:findim-int 4}.
Indeed, by Lemma \ref{lemm:findim-int}, we have $\findim A = \sup \{ \deg_xF(x): F(x)\in\fpipset_A\}$.
If this supremum is finite, then it is an integer and is the least upper bound for the degrees of all elements of
$\fpipset_A$. Therefore, $\findim A$ is precisely the least integer satisfying any of \ref{thm:findim-int 2}--\ref{thm:findim-int 4}.
\end{remark}

\subsection{Homological dimensions vs. regular integral profiles}

For every $1\=< i \=< \dimA$, let $\pmb{t}_i = (1,\ldots,1,\underset{i\text{-th}}2,1,\ldots,1)$,
and use \[\IntProfEmb_{\II_A}A(\pmb{t}_i): =
\mathfrak{Emb}_A\bigg(\IntProf_{\II_A} \id_{\II_A(\pmb{t}_i)} A\bigg)x^0 \in\pandiring(A)[x]
\]
to represent the embedded integral of the left regular module over the restricted integration domain
$\II_A(\pmb{t}_i)=[0,1]a_1+\cdots+[0,1]a_{i-1}+[0,2]a_i+[0,1]a_{i+1}+\cdots+[0,1]a_{\dimA}$,
and call it the \defines{embedded regular integral profile} of $A$. Define
\begin{align}\label{eqdef:recring}
  \recring(A,\pmb{t}_i;1\=< i\=< \dimA)
= \operatorname{span}_{\CC} \bigg\{ \IntProfEmb_{\II_A}A(\pmb{t}_i) :
  1\=< i\=< \dimA \bigg\}
\end{align}
and call it the \defines{reconstruction ring} of $A$.
Then $\recring(A,\pmb{t}_i;1\=< i\=< \dimA)$ is a $\CC$-algebra,
and each element in the basis $\basis{\recring(A,\pmb{t}_i;1\=< i\=< \dimA)}$
$= \{\intProfEmb_{\II_A}A(\pmb{t}_i): 1\=< i\=< \dimA\}$ can be seen as a matrix in $\End_{\CC}(\pandimod^{\oplus\NN})$.


In this subsection, we write $\mathscr R =\recring(A,\pmb{t}_i;1\=<i\=<\dimA)$ and $\mathcal{J}=\rad(\mathscr R)$ for simplicity. Choose a semisimple matrix subalgebra $\mathcal{S}$ of $\mathscr R$ such that
\[\mathscr R=\mathcal{S}\oplus\mathcal{J}\]
as complex vector spaces. Thus $\mathcal{S}\cong\mathscr R/\mathcal{J}$.
All the spaces and maps introduced below are obtained from the matrices $\intProfEmb_{\II_A}A(\pmb{t}_i)$ by matrix multiplication and linear algebra inside $\End_{\CC}(\pandimod^{\oplus\NN})$.

For every $m\in\NN$, let
\[ \mathcal{H}^m_{\II_A}(A) := \Hom_{\mathcal{S}}( \mathcal{J}^{\otimes_{\mathcal{S}}m}, \mathcal{S}),\]
where $\mathcal{J}^{\otimes_{\mathcal{S}}0}=\mathcal{S}$.
For $f\in\mathcal{H}^m_{\II_A}(A)$, define $\delta_m(f)\in\mathcal{H}^{m+1}_{\II_A}(A)$ by
\begin{align}  \label{eq:regular intprof coboundary}
   \delta_m(f)(r_1\otimes\cdots\otimes r_{m+1})
:=  \sum_{k=1}^{m}(-1)^k  f(
  & r_1\otimes\cdots\otimes r_{k-1} \otimes r_kr_{k+1}\nonumber\\
  & \otimes r_{k+2}\otimes\cdots\otimes r_{m+1}),
\end{align}
where $r_1,\ldots,r_{m+1}\in\mathcal{J}$. In particular,
$\delta_0=0$. Since $\delta_{m+1}\compos\delta_m=0$, we obtain a
cochain complex, say \defines{Bar coresolution} of $\mathscr R$, induced by $A$ as follows.
\[0 \longrightarrow \mathcal{H}^0_{\II_A}(A)
    \To{\delta_0} \mathcal{H}^1_{\II_A}(A)
    \To{\delta_1} \mathcal{H}^2_{\II_A}(A)
    \longrightarrow \cdots. \]

\begin{definition}\rm
For every $m\in\NN$, we define
\[ N_m\big(\IntProfEmb_{\II_A}A\big) =\dim\mathcal{H}^m_{\II_A}(A).\]
Choose a basis of $\mathcal{H}^m_{\II_A}(A)$, and let
\[  D_m\big(\IntProfEmb_{\II_A}A\big) \]
be the matrix of $\delta_m$ with respect to the chosen bases. We call it the $m$-th \defines{homological matrix of the embedded regular integral profile}.
\end{definition}

These matrices are directly computable from the values of the regular integral profile. Indeed, if $\{r_1,\ldots,r_q\}$ is a matrix basis of $\mathcal{J}$ and
\[ r_\alpha r_\beta =\sum_{\gamma=1}^{q}c_{\alpha,\beta}^{\gamma}r_\gamma, \]
then every entry of $D_m(\intProfEmb_{\II_A}A)$ is a sum of numbers of the form $0$ or $\pm c_{\alpha,\beta}^{\gamma}$. Consequently, its rank is determined by finite matrix calculations involving only $\intProfEmb_{\II_A}A(\pmb{t}_i)$.

\begin{example}\rm
Let $A=\CC\Q/\I$ given by $1 \To{\alpha_1} 2 \To{\alpha_2} 3 \To{\alpha_3} 4 \To{\alpha_4} 5$ and $\I=\langle \alpha_1\alpha_2\alpha_3$, $\alpha_4\alpha_5 \rangle$. Then $A$ has an ordered basis
\begin{align*}
  \basis{A} & =(e_1,e_2,e_3,e_4,e_5,
  \alpha_1,\alpha_2,\alpha_3,\alpha_4,
  \alpha_1\alpha_2,\alpha_2\alpha_3 ) \\
& = (b_i)_{1\=< i\=< 11}
\end{align*}
and thus $\dimA=11$. The integral profile of $A$ is
\[
\begin{aligned}
\IntProf_{\II_A}A ={}
& \frac{\prod\limits_{j=1}^{11}t_j}{2}
(
  t_1L_{a_1}+t_2L_{a_2}+t_3L_{a_3}
  +t_4L_{a_4}+t_5L_{a_5}\\
& +t_6L_{\alpha_1}
  +t_7L_{\alpha_2}
  +t_8L_{\alpha_3}
  +t_9L_{\alpha_4}\\
& +t_{10}L_{\alpha_1\alpha_2}
  +t_{11}L_{\alpha_2\alpha_3}
).
\end{aligned}
\]
Here, for every
\begin{align*}
a =\, & c_1a_1+c_2a_2+c_3a_3+c_4a_4+c_5a_5 \\
      & +c_{11}\alpha_1+c_{12}\alpha_2+c_{13}\alpha_3+c_{14}\alpha_4 \\
      & +c_{21}\alpha_1\alpha_2+c_{22}\alpha_2\alpha_3,
\end{align*}
the matrix of $L_a: A\to A$ is
\[
L_a=
\left(\begin{smallmatrix}
c_1 & 0 & 0 & 0 & 0 & 0 & 0 & 0 & 0 & 0 & 0 \\
  0 & c_2 & 0 & 0 & 0 & 0 & 0 & 0 & 0 & 0 & 0 \\
  0 & 0 & c_3 & 0 & 0 & 0 & 0 & 0 & 0 & 0 & 0 \\
  0 & 0 & 0 & c_4 & 0 & 0 & 0 & 0 & 0 & 0 & 0 \\
  0 & 0 & 0 & 0 & c_5 & 0 & 0 & 0 & 0 & 0 & 0 \\
  0 & c_{11} & 0 & 0 & 0 & c_1 & 0 & 0 & 0 & 0 & 0 \\
  0 & 0 & c_{12} & 0 & 0 & 0 & c_2 & 0 & 0 & 0 & 0 \\
  0 & 0 & 0 & c_{13} & 0 & 0 & 0 & c_3 & 0 & 0 & 0 \\
  0 & 0 & 0 & 0 & c_{14} & 0 & 0 & 0 & c_4 & 0 & 0 \\
  0 & 0 & c_{21} & 0 & 0 & 0 & c_{11} & 0 & 0 & c_1 & 0\\
  0 & 0 & 0 & c_{22} & 0 & 0 & 0 & c_{12} & 0 & 0 & c_2
\end{smallmatrix}\right),
\]
Moreover, we have
\[\mathcal{J}=
       \CC L_{\alpha_1}
\oplus \CC L_{\alpha_2}
\oplus \CC L_{\alpha_3}
\oplus \CC L_{\alpha_4}
\oplus \CC L_{\alpha_1\alpha_2}
\oplus \CC L_{\alpha_2\alpha_3}
\quad \text{and} \quad
\mathcal{S}=\bigoplus_{i=1}^{5}\CC L_{e_i}\]
such that
\[\mathscr{R}=\mathcal{S}\oplus\mathcal{J}.\]
Under the ordered basis $\basis{A}$, let $\pmb{t}_i=(1,\ldots,1,\underset{i\text{-th}}2,1,\ldots,1)$, then
\[\IntProf_{\II_A}A(\pmb{t}_i)
= L_{b_i+\sum\limits_{j=1}^{5}e_j+
   \sum\limits_{j=1}^{4}\alpha_j+
   \alpha_1\alpha_2+\alpha_2\alpha_3}
= L_{b_i+\sum\limits_{j=1}^{11}b_j}, \]
and further
\[\IntProfEmb_{\II_A}A(\pmb{t}_i)
= \mathfrak{Emb}_A\Big(L_{b_i+\sum\limits_{j=1}^{11}b_j}\Big). \]
Moreover, we have
\[\sum_{i=1}^{11}\IntProf_{\II_A}A(\pmb{t}_i) = 12L_{\sum\limits_{k=1}^{11}b_k}.\]
Thus, as an example, we obtain that
\begin{align*}
    \mathfrak{Emb}_A(L_{\alpha_1})
& = \IntProfEmb_{\II_A}A(\pmb{t}_6)-\frac{1}{12}\sum_{i=1}^{11}\IntProfEmb_{\II_A}A(\pmb{t}_i).
\end{align*}
Next we compute $N_m( \intProfEmb_{\II_A}A)$ and $D_m( \intProfEmb_{\II_A}A)$.
\begin{itemize}
  \item For $m=0$, we have $\mathcal{J}^{\otimes_{\mathcal{S}}0}=\mathcal{S}$ and $N_0( \intProfEmb_{\II_A}A)=5$.
  \item For $m=1$, we have $\mathcal{J}^{\otimes_{\mathcal{S}}1}=\mathcal{J}$ and $N_1( \intProfEmb_{\II_A}A)=6$.
  \item For $m=2$, $\mathcal{J}\otimes_{\mathcal{S}}\mathcal{J}$ has an ordered basis
    \[\{ L_{\alpha_1}\otimes L_{\alpha_2}, \,
         L_{\alpha_1}\otimes L_{\alpha_2\alpha_3}, \,
         L_{\alpha_2}\otimes L_{\alpha_3}, \,
         L_{\alpha_3}\otimes L_{\alpha_4}, \,
         L_{\alpha_1\alpha_2}\otimes L_{\alpha_3}, \,
         L_{\alpha_2\alpha_3}\otimes L_{\alpha_4} \}, \]
    then $N_2( \intProfEmb_{\II_A}A)=6$.
  \item For $m=3$, $\mathcal{J}^{\otimes_{\mathcal{S}}3}$ has an ordered basis
    \[\{ L_{\alpha_1}\otimes L_{\alpha_2}\otimes L_{\alpha_3},
         L_{\alpha_2}\otimes L_{\alpha_3}\otimes L_{\alpha_4},
         L_{\alpha_1\alpha_2}\otimes L_{\alpha_3}\otimes L_{\alpha_4},
         L_{\alpha_1}\otimes L_{\alpha_2\alpha_3}\otimes L_{\alpha_4}\}, \]
    then $N_3( \intProfEmb_{\II_A}A)=4$.
  \item For $m=4$, $\mathcal{J}^{\otimes_{\mathcal{S}}4}$ has an ordered basis
    \[\{L_{\alpha_1}\otimes L_{\alpha_2}\otimes L_{\alpha_3}\otimes L_{\alpha_4}\},\]
    then $N_4( \intProfEmb_{\II_A}A)=1$.
  \item For $m\>=5$, we have $\mathcal{J}^{\otimes_{\mathcal{S}}m}=0$, a trivial case.
\end{itemize}
Thus, $N_0=5$, $N_1=6$, $N_2=6$, $N_3=4$, $N_4=1$, $N_{\>=5}=0$. Under the ordered bases of $\mathcal{J}^{\otimes_{\mathcal{S}}m}$ as above (to be precise, we need to consider the dual bases of these bases), we have the following facts.
\begin{itemize}
  \item $D_0(\IntProfEmb_{\II_A}A) = \mathbf{0}_{6\times 5}$.
  \item $D_1(\IntProfEmb_{\II_A}A) =
    \left(\begin{smallmatrix}
      0 & 0 & 0 & 0 &-1 & 0 \\
      0 & 0 & 0 & 0 & 0 & 0 \\
      0 & 0 & 0 & 0 & 0 &-1 \\
      0 & 0 & 0 & 0 & 0 & 0 \\
      0 & 0 & 0 & 0 & 0 & 0 \\
      0 & 0 & 0 & 0 & 0 & 0
    \end{smallmatrix}
    \right)_{6\times 6}$.
  \item $D_2(\IntProfEmb_{\II_A}A) =
    \left(\begin{smallmatrix}
      0 & 1 & 0 & 0 &-1 & 0 \\
      0 & 0 & 0 & 0 & 0 &-1 \\
      0 & 0 & 0 & 0 & 0 & 0 \\
      0 & 0 & 0 & 0 & 0 & 0
    \end{smallmatrix}\right)_{4\times 6}$.
  \item $D_3(\IntProfEmb_{\II_A}A) = \left(\begin{smallmatrix}0 & 0 & -1 & 1\end{smallmatrix}\right)_{1\times 4}$.
  \item $D_4(\IntProfEmb_{\II_A}A) =  0 $.
\end{itemize}
\end{example}

\begin{lemma}\label{lemm:regular intprof Ext rank}
For every $m\in\NN$, we have
\begin{align}\label{eq:regular intprof Ext rank}
  \dim_{\CC}\Ext_A^m(\top(A),\top(A))
  ={} & N_m\big(\IntProfEmb_{\II_A}A\big)\nonumber\\
      &-\rank D_{m-1}\big(\IntProfEmb_{\II_A}A\big)
       -\rank D_m\big(\IntProfEmb_{\II_A}A\big),
\end{align}
where the term containing $D_{-1}$ is omitted when $m=0$.
\end{lemma}

\begin{proof}
Consider the relative normalized Bar resolution of the left $\mathscr R$-module $\mathcal{S}$ whose $m$-th term is
$\mathscr R\otimes_{\mathcal{S}} \mathcal{J}^{\otimes_{\mathcal{S}}m}$.
Since $\mathcal{S}$ is semisimple, every term is a projective left $\mathscr R$-module.
Applying $\Hom_{\mathscr R}(-,\mathcal{S})$ gives the cochain complex $(\mathcal{H}^m_{\II_A}(A),\delta_m)_{m\in\NN}$.
Indeed, the first term in the Bar differential vanishes after applying
$\Hom_{\mathscr R}(-,\mathcal{S})$ because $\mathcal{S}$ is regarded as an
$\mathscr R$-module through the quotient
$\mathscr R\to\mathscr R/\mathcal{J}\cong\mathcal{S}$, and hence
$\mathcal{J}$ acts trivially on $\mathcal{S}$. The remaining terms give precisely
\eqref{eq:regular intprof coboundary}. Hence, by Fact \ref{fact:regular 0814}, we have
\[\mathrm{H}^m(\mathcal{H}^\bullet_{\II_A}(A))
  \cong \Ext_{\mathscr R}^m(\mathcal{S},\mathcal{S})
  \cong \Ext_A^m(\top(A),\top(A)). \]
Finally, by $\dim_{\CC}H^m =\dim_{\CC}\mathcal{H}^m_{\II_A}(A) - \rank\delta_{m-1}-\rank\delta_m$, we obtain this lemma.
\end{proof}

The following theorem reads the global dimension of $A$ directly from
the ranks of the homological matrices of its regular integral profile.

\begin{theorem}\label{thm:regular intprof rank gldim}
For every $m\in\NN^{>0}$, the following statements are equivalent:
\begin{enumerate}[label={\rm(\arabic*)}]
  \item $\gldim A\=<m-1$;
    \label{thm:regular intprof rank gldim 1}
  \item $\rank D_{m-1}\big(\IntProfEmb_{\II_A}A\big)
    +\rank D_m\big(\IntProfEmb_{\II_A}A\big)
    =N_m\big(\IntProfEmb_{\II_A}A\big).$
    \label{thm:regular intprof rank gldim 2}
\end{enumerate}
Consequently,
\begin{align} \label{eq:regular intprof read gldim}
  \gldim A = \min\bigg\{m\in\NN^{>0}: ~
& \rank D_{m-1}\big(\IntProfEmb_{\II_A}A\big)+\rank D_m\big(\IntProfEmb_{\II_A}A\big) \nonumber \\
& \qquad =N_m\big(\IntProfEmb_{\II_A}A\big) \bigg\}-1,
\end{align}
where the right-hand side is understood to be $\infty$ if the set is
empty.
\end{theorem}

\begin{proof}
By Lemma~\ref{lemm:regular intprof Ext rank}, \ref{thm:regular intprof rank gldim 2} holds if and only if $\Ext_A^m(\top(A),\top(A))=0$.
Let $ \cdots \to Q^2\to Q^1 \to Q^0 \to \top(A)\to 0$ be the minimal projective resolution of $\top(A)$.
Since the resolution is minimal, applying $\Hom_A(-,\top(A))$ gives zero differentials. Thus
\[
  \Ext_A^m(\top(A),\top(A)) \cong\Hom_A(Q^m,\top(A)).
\]
If $Q^m\ne0$, then $Q^m$ has a nonzero simple top, and hence $\Hom_A(Q^m,\top(A))\ne0$.
Therefore $\Ext_A^m(\top(A),\top(A))=0$ if and only if $Q^m=0$.
The latter is equivalent to $\pdim_A\top(A)\=<m-1$.
Since $S(v) \isods \top(A)$ holds for every simple left $A$-module $S(v)$ ($v\in\Q_0$),
we have $\gldim A=\pdim_A\top(A)$ by \eqref{eq:gldim-simp mod}.
This proves the equivalence and formula \eqref{eq:regular intprof read gldim}.
\end{proof}

\subsection{Two sufficient conditions for finitistic dimension to be finite}
\label{subsect:Suffi condi for fdim}

{\addtext
For a finite-dimensional algebra $A$,
keep the fixed embedding system from Assumption \ref{assu:fix emb sys}.
Define a formal differential operator \[\theta_x:=x\frac{\mathrm d}{\mathrm dx},\]
and for each element $\sum\limits_{n\>= 0}C_nx^n \in \pandiring(A)[[x]]$ ($C_n \in \pandiring(A)$),
the formal differential operator $\theta_x$ admits a correspondence
\[ \theta_x: \sum_{n\>= 0}C_nx^n \mapsto \sum_{n\>= 0}C_n\theta_x(x^n) = \sum_{n\>= 0}nC_nx^n. \]
Here, the parameter vector $\pmb t$ in each coefficient is unchanged.
Then for every $p(z)\in\CC[z]$ and $j\in\ZZ_{\>=0}$, we have
\[  x^jp(\theta_x)(C_nx^n) = x^j\cdot p(n)C_nx^n = p(n)C_nx^{n+j}. \]
We use the convention $\deg_x 0 = -\infty$ in this subsection.

Furthermore, assume that the fixed embedding system is compatible orthogonal
in the sense of Definition \ref{def:compatible orth emb sys},
and let $Q_1,Q_2,\ldots,Q_r$ be representatives of the isomorphism classes of indecomposable projective left $A$-modules.
Recall that the element
\[ \big(F_{Q_i,1}: (\RR^{\>=0})^{\times\dimA}\to \End_{\CC}(\pandimod^{\oplus\NN})\big)
   \in \pandiring(A)\subseteq\widetilde{\pandiring}(A)\]
is given in Subsection \ref{subsect:Morita inv icring}, where $\dimA$ is the dimension of $A$.
The element $F_{Q_i,1}$ is supported on the first fixed copy of $Q_i$, i.e.,
for each $1\=< i\=< r$, $F_{Q_i,1}$ is defined by
\[ F_{Q_i,1}(\pmb t) |_{X^{(\ell)}}=
 \begin{cases}
  \IntProf_{\II_A}Q_i,
     & (X,\ell)=(Q_i,1),\\
  0, & (X,\ell)\ne(Q_i,1),
 \end{cases}
\]
see \eqref{def formula: F-int}.

\begin{definition}\rm\label{def:linearized embedded profile}
Let $P$ be a finitely generated projective left $A$-module. Write
\[ P \cong \bigoplus\limits_{i=1}^r Q_i^{\oplus m_i(P)}. \]
The \defines{linearized embedded integral profile} of $P$ is
\[ (\mathcal L) \IntProf_{\II_A}P := \sum_{i=1}^r m_i(P)F_{Q_i,1}\in\pandiring(A). \]
For a left $A$-module $M$ with minimal projective resolution $\pmb P_M^\bullet$, define its
\defines{total linearized embedded projective integral profile} by
\[ (\mathcal L) \IntProfEmb_{\II_A} (\pmb P_M^\bullet;x)
 := \sum_{n\>=0} \bigg((\mathcal L) \IntProf_{\II_A}P_M^n\bigg)x^n \in \pandiring(A)[[x]]. \]
\end{definition}

For all projective $A$-modules $P$ and $P'$ and all $a,b\in\NN$, one has
\begin{equation*}
  (\mathcal L) \IntProf_{\II_A} (P^{\oplus a}\oplus {P'}^{\oplus b})
= a \cdot (\mathcal L) \IntProf_{\II_A} P + b \cdot (\mathcal L) \IntProf_{\II_A} P'\checks{.}
\end{equation*}

\begin{lemma}\label{lemm:linear prof add}
Let $P$ and $Q$ be finitely generated projective left $A$-modules. Then:
\begin{enumerate}[label={\rm(\arabic*)}]
 \item $(\mathcal L) \IntProf_{\II_A}P=0$ if and only if $P=0$;
   \label{lemm:linear prof add 1}
 \item $(\mathcal L) \IntProf_{\II_A}(P\oplus Q)
   = (\mathcal L) \IntProf_{\II_A}P
   + (\mathcal L) \IntProf_{\II_A}Q$.
   \label{lemm:linear prof add 2}
\end{enumerate}
Here, $[P]$ denotes the isomorphism class of $P$ in the split Grothendieck group $K_0^{\mathrm{split}}({_A\proj})$.
Then the map
\[ [P]\longmapsto (\mathcal L) \IntProf_{\II_A} P \]
defines an injective homomorphism from $K_0^{\mathrm{split}}({_A\proj})$ of finitely generated projective left $A$-modules to the additive group of $\pandiring(A)$.
\end{lemma}

\begin{proof}
By the Krull--Schmidt decomposition, we have $m_i(P\oplus Q)=m_i(P)+m_i(Q)$ for every $i$.
Thus \ref{lemm:linear prof add 2} holds.
Proposition \ref{prop:orth decomposition icring} shows that
$F_{Q_i,1} \in \langle F_{Q_i,1} \rangle \le_{\oplus} \pandiring(A)$ for all $1\=< i\=< r$,
and by Lemma \ref{lemm:cyclic int coeff component}, each of them is nonzero.
Therefore, $\sum\limits_{i=1}^r m_i(P)F_{Q_i,1}=0$ if and only if $m_i(P)=0$ for every $i$.
Thus, \ref{lemm:linear prof add 1} holds.
Moreover, if $(\mathcal L)\intProf_{\II_A}P=(\mathcal L)\intProf_{\II_A}Q$, then the direct-sum decomposition in Proposition \ref{prop:orth decomposition icring} gives $m_i(P)=m_i(Q)$ for every $i$. Hence $[P]=[Q]$, which proves the last assertion.
\end{proof}

For $p_0(z),\ldots, p_k(z)\in\CC[z]$, define the formal differential operator
\[ L_{p_0,p_1,\ldots, p_k} := \sum_{j=0}^{k}x^jp_j(\theta_x) = \sum_{j=0}^{k}x^jp_j(x\frac{\dd}{\dd x}). \]
We have the following result.

\begin{theorem}\label{thm:diff-findim}
Let $A$ be a nonzero finite-dimensional complex algebra.
Suppose that there are integers $k,s \>= 0$ and polynomials $p_0(z),\ldots, p_k(z)\in\CC[z]$ $(p_k\ne0)$,
such that one of the following conditions holds for every $0\ne M\in{_A\modcat}$:
\begin{enumerate}[label={\rm(\arabic*)}]
  \item $\displaystyle R_M(x) := L_{p_0,p_1,\ldots, p_k} \IntProfEmb_{\II_A}(\pmb P_M^\bullet;x) \in\pandiring(A)[x],
  \quad\text{and}\quad \deg_xR_M(x) \=< s,$
\label{eq:diff findim 1}

  \item $\displaystyle R_M^{\mathrm{lin}}(x)
 := L_{p_0,p_1,\ldots, p_k} \bigg(
 (\mathcal L) \IntProfEmb_{\II_A} (\pmb P_M^\bullet;x) \bigg)
 \in\pandiring(A)[x] \text{ and }
 \deg_xR_M^{\mathrm{lin}}(x)\=<s.$
\label{eq:diff findim 2}
\end{enumerate}
Here, $k$, $s$ and $p_0,\ldots,p_k$ are independent of $M$.
Then \[\findim A\=< \max (\{0,s-k\}\cup \{n\in\ZZ_{\>= 0}:p_k(n)=0\} ) <\infty. \]
\end{theorem}

\begin{proof}
We write $b=\max (\{0,s-k\}\cup \{n\in\ZZ_{\>= 0}:p_k(n)=0\} )$ in this proof.

\ref{eq:diff findim 1}:
Let $0\ne M\in{_A\modcat}$ satisfy $\pdim_A M=d<\infty$.
By Lemma \ref{lemm:pdim-int 0812} and the injectivity of $\mathfrak{Emb}_{P_M^d}$,
its embedded total projective integral profile is a polynomial of degree $d$.
Its leading coefficient is $\mathfrak{Emb}_{P_M^d} (\intProf_{\II_A}P_M^d)\ne 0$.
The coefficient of $x^{d+k}$ on the left-hand side of \ref{eq:diff findim 1} is
$p_k(d) \mathfrak{Emb}_{P_M^d} (\intProf_{\II_A}P_M^d)$.
Indeed, the terms with $j<k$ have degree at most $d+k-1$.
If $d+k>s$, then $\deg_xR_M(x)\=< s$ implies that the coefficient of $x^{d+k}$ in $R_M(x)$ is zero. Hence $p_k(d)=0$.
Otherwise, $d\=< s-k$. Thus, $d\=< b$ in both cases.
Taking the supremum over all modules of finite projective dimension gives $\findim A\=< b$.
Since $p_k$ is a nonzero polynomial, it has only finitely many non-negative integer roots.
Thus, $b<\infty$.

\ref{eq:diff findim 2}:
Let $0\ne M\in{_A\modcat}$ and suppose that $\pdim_AM=d<\infty$. Then
\[ (\mathcal L) \IntProfEmb_{\II_A} (\pmb P_M^\bullet;x)
= \sum_{n=0}^d \bigg((\mathcal L) \IntProf_{\II_A}P_M^n \bigg)x^n \]
is a polynomial of degree $d$ ($<\infty$). Here, $P_M^d \ne 0$,
and then, by Lemma \ref{lemm:linear prof add}, $(\mathcal L) \intProf_{\II_A} P_M^d \ne 0$.
The coefficient of $x^{d+k}$ on the left-hand side of \ref{eq:diff findim 2} is
$p_k(d) \cdot (\mathcal L) \intProf_{\II_A} P_M^d$.
If $d+k>s$, then $\deg_xR_M^{\mathrm{lin}}(x)\=<s$ implies that
the coefficient of $x^{d+k}$ in $R_M^{\mathrm{lin}}(x)$ is zero.
Thus, $p_k(d)\cdot(\mathcal L)\intProf_{\II_A}P_M^d=0$.
Since $(\mathcal L)\intProf_{\II_A}P_M^d\ne0$, we have $p_k(d)=0$.
Then we have $p_k(d)=0$ immediately.
If $d+k\=<s$, then $d\=<s-k$. Thus, $d \=< b$.
Taking the supremum over all modules of finite projective dimension gives the required bound. The set of nonnegative integral roots of $p_k \in \CC[z]$ is finite because $p_k\ne 0$.
\end{proof}

It is open that $\findim A$ is finite for every finite-dimensional algebra $A$ satisfying $\rad^4 A=0$.
Next, we give an example showing that Theorem \ref{thm:diff-findim} applies to an algebra of this kind.

\begin{example}\rm\label{ex:two sided radical four}
Let $A$ be the finite-dimensional algebra given by \eqref{eq:A 2609100913}.
Then it satisfies the following conditions:
\begin{itemize}
  \item $\rad^4(A)=0\ne\rad^3(A)$;
  \item $A$ is not self-injective;
  \item $A$ is syzygy-infinite;
  \item $A$ is connected;
  \item $A$ is non-monomial;
  \item and $A$ has an infinite global dimension,
\end{itemize}
see Section \ref{appendix 3: an alg}. Define
\[ q(x)=1+c_8x+\cdots+c_1x^8+c_0x^9 =\sum_{j=0}^9q_jx^j,\]
where $c_0$, $c_1, \ldots, c_8$, and $c_9=1$ are coefficients given in \eqref{eq:proj recur 260911}.
Choose a compatible orthogonal fixed embedding system.
By Lemma \ref{lemm:linear prof add} and \eqref{eq:proj recur 260911}, we have
\[ Q_M(x) = q(x)\cdot(\mathcal L)\IntProfEmb_{\II_A}(\pmb P_M^\bullet;x) \in\pandiring(A)[x], \quad \deg_xQ_M \=< 9. \]
Applying $\displaystyle \theta_x=x\frac{\dd}{\dd x}$ to the above equation, we obtain
\[ R_M(x) = \bigg({\sum_{j=0}^9q_jx^j(\theta_x+j)}\bigg)
 \cdot(\mathcal L)  \IntProfEmb_{\II_A} (\pmb P_M^\bullet;x) = \theta_xQ_M(x),
 \quad \deg_xR_M \=< 9. \]
Thus, we have $k=s=9$ and $p_j(z)=q_j(z+j)$ $(0 \=< j \=< 9)$.
Since $q_9=c_0=\mu_3 > 0$, the leading polynomial is $p_9(z)=\mu_3(z+9)$.
It has no nonnegative integral roots. By Theorem \ref{thm:diff-findim}, we have $\findim A=0$.
\end{example}

}

\section*{Declarations}

\paragraph{Authors' Contributions}


\paragraph{Competing Interests}

The author declare that they have no conflicts of interest as defined by the journal, nor any other interests that could be perceived as influencing the results presented in this paper.

\paragraph{Data availability}
Data sharing not applicable to this article as no datasets were generated or analysed during the current study

\paragraph{Ethical Approval}

This article does not require ethical approval.

\paragraph{Fundings}

Yu-Zhe Liu is supported by
the National Natural Science Foundation of China (Grant Nos. 12401042 and 12561008),
the Science and Technology Foundation of the Guizhou S\&T Department (Grant Nos. VZD[2026]001, ZD[2025]085 and ZK[2024]YiBan066),
and Scientific Research Foundation of Guizhou University (Grant No. [2023]16).

\newpage

\appendix

\section*{\centering Appendix}
\addcontentsline{toc}{section}{Appendix}

\section{Differential inverses of the integral profiles}\label{subsect:diff}

\textsl{For a module $M$, we now have its embedded total projective and injective integral profiles
$\intProfEmb_{\II_A} (\pmb{P}^{\bullet}_M;x)$ and $\intProfEmb_{\II_A}(\pmb{I}^M_{\bullet};x)$.
In Section~\ref{sect:hdim and indef int}, we point out that these two integrals are essentially generalisations of indefinite integrals. A natural question is:}

\begin{question}\label{quest:0814:0952}
Can one recover $M$ from $\intProfEmb_{\II_A} (\pmb{P}^{\bullet}_M;x)$ or $\intProfEmb_{\II_A}(\pmb{I}^M_{\bullet};x)$ by taking some kind of derivative?
\end{question}

\noindent
\textsl{This appendix is based on Question \ref{quest:0814:0952}.
This appendix will take the regular module $A$ as an example to answer Question \ref{quest:0814:0952}.}

\subsection{Reconstructing algebras by \texorpdfstring{$\reconstr$}{ Rec}}
\label{subsect:diff-reg mod}

Left regular module ${_AA}$ is projective whose minimal projective resolution is $\pmb{A}^{\bullet} = (A^n,d^n)_{n\>= 0}$,
where $A^0 = A$, $d^n=0$ for all $n\>= 0$, and $A^n=0$ for all $n>0$.
Then, by Definition \ref{def:tot proj int prof}, we have that the embedded projective integral profile of $A$ is
\[ \IntProfEmb_{\II_A}(\pmb A^\bullet;x)=
   \mathfrak{Emb}_A\bigg(\IntProf_{\II_A}A\bigg)x^0 \in\pandiring(A)[x].\]
We write it as
\[\IntProfEmb_{\II_A} A\]
and call it the \defines{embedded regular integral profile} of $A$ for simplicity.

For every $v\in\Q_0$, let $L_i^{(v)}$ be the matrix of the
restriction
\[L_{a_i}|_{Ae_v}:Ae_v \to Ae_v, \quad x \mapsto a_ix,\]
with respect to the previously chosen basis of $Ae_v$.
Consider the canonical left $A$-module isomorphism
\[\psi: \bigoplus_{v\in\Q_0}Ae_v \to A,
\quad (x_v)_{v\in\Q_0} \mapsto \sum_{v\in\Q_0}x_v. \]
For every $1\=< i\=< n$, we have
\[ L_{a_i} \compos \psi = \psi \compos \bigg(\bigoplus_{v\in\Q_0}L_{a_i}|_{Ae_v}\bigg). \]
Hence, if $H$ is the matrix of $\psi$ with respect to the chosen
bases, then, applying Theorem \ref{thm:mat char-proj mod} to the left regular module ${_AA}$, we obtain
\begin{align}
  H^{-1}\bigg(\IntProf_{\II_A}A(\pmb{t})\bigg)H
& = \frac{\mu_{\II_A(\pmb{t})} (\II_A(\pmb{t}))}{2} \bigoplus_{v\in\Q_0}
    \bigg(\sum_{i=1}^{\dimA}t_iL_i^{(v)}\bigg) \nonumber \\
& = \frac{\mu_{\II_A(\pmb{t})}(\II_A(\pmb{t}))}{2}
    \sum_{i=1}^{\dimA} t_i \bigg(\bigoplus_{v\in\Q_0}L_i^{(v)}\bigg) \label{eq:tot proj int prof}
\end{align}
and
\[ (L_{a_i})_{\dimA \times \dimA} = H \bigg( \bigoplus_{v\in\Q_0}L_i^{(v)} \bigg) H^{-1}.\]
It follows that
\begin{align}
   \IntProf_{\II_A}A
&\mathop{=}\limits^{\eqref{eq:tot proj int prof}}
   \frac{\mu_{\II_A(\pmb{t})}(\II_A(\pmb{t}))}{2}
     \sum_{i=1}^{\dimA} t_i\cdot H \bigg(\bigoplus_{v\in\Q_0}L_i^{(v)}\bigg) H^{-1} \nonumber \\
&= \frac{\mu_{\II_A(\pmb{t})}(\II_A(\pmb{t}))}{2}
     \sum_{i=1}^{\dimA} t_i L_{a_i} \nonumber \\
&= \frac{1}{2} \prod_{j=1}^{\dimA}t_j \cdot \sum_{i=1}^{\dimA}t_iL_{a_i} \nonumber \\
&= \frac{1}{2}\sum_{i=1}^{\dimA}t_i^2\prod_{j\ne i}t_j\,L_{a_i}. \label{eq:regular intprof}
\end{align}
For every $1\=< i\=< \dimA$, define 
\[ \reconstr_i :=
     \frac{\partial^{\dimA+1}}{\partial t_i^2\prod\limits_{j\ne i}\partial t_j} \bigg|_{\pmb{t}=0}. \]
Here the derivative at $\pmb{t}=0$ may equivalently be understood as \textbf{$i$-th algebraic coefficient extraction from the polynomial in \eqref{eq:regular intprof}}.
Immediately,
\begin{equation} \label{eq:regular coeff recovery}
  \reconstr_i \IntProf_{\II_A}A = L_{a_i} \quad (1\=< i\=< n),
\end{equation}
and therefore
\[ \reconstr_i\IntProfEmb_{\II_A} A
 = \reconstr_i\IntProfEmb_{\II_A} (\pmb{A}^{\bullet};x)
 = \mathfrak{Emb}_A\bigg(\reconstr_i\IntProf_{\II_A}A\bigg)x^0
 = \mathfrak{Emb}_A(L_{a_i})
 \in \End_{\CC}(\pandimod^{\oplus\NN}).\]
Let
\[ \recring_{\partial}(A)
:= \mathrm{span}_{\CC}
   \left\{
     \reconstr_i\IntProfEmb_{\II_A} A : 1\=< i\=< \dimA
   \right\}
   \subseteq
     \End_{\CC}(\pandimod^{\oplus\NN}). \]
The following lemma shows that $\recring_{\partial}(A)$ is a ring.

\begin{fact}\label{fact:ricring iso to A}
The set $\recring_{\partial}(A)$ is a ring, and its identity element is the idempotent $1_{\recring_{\partial}(A)}: = \mathfrak{Emb}_A(\ident_A)$.
Moreover, the map $A\to\recring_{\partial}(A)$, $a\mapsto\mathfrak{Emb}_A(L_a)$ shows that $A\cong \recring_{\partial}(A)$.
\end{fact}

\begin{proof}
By \eqref{eq:regular coeff recovery}, we have
\[ \reconstr_i \IntProfEmb_{\II_A}A = \mathfrak{Emb}_A(L_{a_i})\]
for every $1\=< i\=< \dimA$. Hence
\[ \recring_{\partial}(A) = \mathrm{span}_{\CC}\{ \mathfrak{Emb}_A(L_{a_i}): 1\=< i\=< \dimA\}.\]
By $\basis{A}=\{a_1,\ldots,a_{\dimA}\}$, every $a\in A$ can be written uniquely as
$a=\sum\limits_{i=1}^{\dimA}\alpha_i a_i$ for some $\alpha_i\in\CC$. Consequently,
$L_a = \sum\limits_{i=1}^{\dimA}\alpha_iL_{a_i}$, and therefore
\[ \recring_{\partial}(A)
= \{ \mathfrak{Emb}_A(L_a):a\in A \}
= \bigg\{ \mathfrak{Emb}_A\bigg(\sum_{i=1}^{\dimA}\alpha_iL_{a_i}\bigg): a\in A \bigg\}.\]

Let $a,b\in A$. Since $L_a\compos L_b=L_{ab}$ and $\mathfrak{Emb}_A$ preserves compositions, we obtain
\[ \mathfrak{Emb}_A(L_a)\compos\mathfrak{Emb}_A(L_b)
 = \mathfrak{Emb}_A(L_a\compos L_b)
 = \mathfrak{Emb}_A(L_{ab}) \in\recring_{\partial}(A). \]
Thus $\recring_{\partial}(A)$ is closed under multiplication.
It is closed under addition and scalar multiplication by its definition as a complex linear span.
Therefore $\recring_{\partial}(A)$ is a finite-dimensional complex algebra.

Moreover, for every $a\in A$, we have
\begin{align*}
    \mathfrak{Emb}_A(\ident_A) \compos \mathfrak{Emb}_A(L_a)
& = \mathfrak{Emb}_A(\ident_A\compos L_a)
  = \mathfrak{Emb}_A(L_a);\\
    \mathfrak{Emb}_A(L_a) \compos \mathfrak{Emb}_A(\ident_A)
& = \mathfrak{Emb}_A(L_a\compos\ident_A)
  = \mathfrak{Emb}_A(L_a).
\end{align*}
Then $1_{\recring_{\partial}(A)}=\mathfrak{Emb}_A(\ident_A)$ is the identity element of $\recring_{\partial}(A)$.

Finally, the map $A\to\recring_{\partial}(A)$, $a\longmapsto\mathfrak{Emb}_A(L_a)$ is surjective
by the description of $\recring_{\partial}(A)$ above.
And if $\mathfrak{Emb}_A(L_a)=0$, then the injectivity of $\mathfrak{Emb}_A$ implies $L_a=0$.
It follows that this map is also injective, and so it is an isomorphism.
\end{proof}

Fact \ref{fact:ricring iso to A} means that the following isomorphism
\[ A \cong \bigoplus_{i=1}^{\dimA} \CC\cdot\reconstr_i\IntProfEmb_{\II_A} A \]
of $\CC$-linear spaces holds.

Although $\intProfEmb_{\II_A} A$ belongs to $\pandiring(A)$,
its parameter derivatives $\reconstr_i\intProfEmb_{\II_A}A$ need
not belong to $\pandiring(A)$ merely by the definition of the latter as an
algebra generated by integral profiles.
Thus $\recring_{\partial}(A)$ can be understood as an algebra obtained from the $(\dimA+2)$-tuple
\[(\pandiring(A),\IntProfEmb_{\II_A} A, \reconstr_1, \ldots, \reconstr_{\dimA}). \]
This distinction is important. In this section we use the integral coefficient ring together $\pandiring(A)$
with its ``distinguished regular profile'' $\intProfEmb_{\II_A} A$ and
``coefficient recovery operations'' $(\reconstr_i)_{1\=< i \=< \dimA}$,
rather than the unmarked abstract ring $\pandiring(A)$ alone.

\subsection[Reconstructing algebras by $\bigoplus_i\int $]
  {Reconstructing algebras by
   \texorpdfstring{$\displaystyle\bigoplus\nolimits_i\protect\IntProfEmb $}{¨'¡Ò}
  }
\label{subsect:diff-reg mod 2}

From Subsection \ref{subsect:diff-reg mod}, we see that the regular module $A$, after being integrated, can be recovered back to $A$ itself by differentiation. This appendix will show that, as long as a suitable integration domain is chosen, the integral of $A$ can be directly recovered to $A$ without passing through differentiation.

Let $\pmb{t}_i:=(1,\ldots,1,\underset{i\text{-th}}{2},1,\ldots,1) \in (\RR^{\>=0})^{\times\dimA}$,
and use \[\IntProfEmb_{\II_A}A(\pmb{t}_i)\]
to represent the embedded integral of the left regular module over the restricted integration domain
\[ \II_A(\pmb{t}_i)=[0,1]a_1+\cdots+[0,1]a_{i-1}+[0,2]a_i+[0,1]a_{i+1}+\cdots+[0,1]a_{\dimA} \]

For every $v\in\Q_0$,
let $\basis{Ae_v}=\{p_{v,1},\ldots,p_{v,d_v}\}$ and $d_v=\dim_{\CC}Ae_v$.
For every $1\=< j\=< \dimA$ and $1\=< \ell\=< d_v$,
let $a_jp_{v,\ell} = \sum\limits_{k=1}^{d_v} \lambda_{j,\ell,k}^{(v)}p_{v,k}$.
Then, recall Subsection \ref{subsect:proj resol},
the matrix of the restriction $L_{a_j}|_{Ae_v}:Ae_v\to Ae_v$ with respect to the basis $\basis{Ae_v}$ is
\[
L_j^{(v)}
=
\left(
\begin{matrix}
  \lambda_{j,1,1}^{(v)}
& \lambda_{j,2,1}^{(v)}
& \cdots
& \lambda_{j,d_v,1}^{(v)} \\
  \lambda_{j,1,2}^{(v)}
& \lambda_{j,2,2}^{(v)}
& \cdots
& \lambda_{j,d_v,2}^{(v)} \\
  \vdots
& \vdots
&
& \vdots \\
  \lambda_{j,1,d_v}^{(v)}
& \lambda_{j,2,d_v}^{(v)}
& \cdots
& \lambda_{j,d_v,d_v}^{(v)}
\end{matrix}
\right).
\]
Consider the canonical left $A$-module isomorphism
\[ \psi: \bigoplus_{v\in\Q_0}Ae_v \to A,
\qquad (x_v)_{v\in\Q_0} \mapsto \sum_{v\in\Q_0}x_v, \]
and let $H$ be the matrix of $\psi$ with respect to the basis obtained by combining the bases $\basis{Ae_v}$ and the fixed basis $\basis{A}=\{a_1,a_2,\ldots,a_{\dimA}\}$.

\begin{lemma}\label{lemm:0814:0903}
$\displaystyle
  \IntProfEmb_{\II_A}A(\pmb{t}_i)
= \mathfrak{Emb}_A \bigg(\IntProf_{\II_A}A(\pmb{t}_i)\bigg)
= \mathfrak{Emb}_A ( L_{\tilde{a}_j} )$.
\end{lemma}

\begin{proof}
By Theorem~\ref{thm:mat char-proj mod} and the proof thereof, the simultaneous similarity for the left regular module is realized by the fixed matrix $H$. Hence
\begin{align} \label{eq:0814:0834}
    \IntProf_{\II_A}A(\pmb{t})
& =\frac{\mu_{\II_A(\pmb{t})}(\II_A(\pmb{t}))}{2}
  H \bigg( \bigoplus_{v\in\Q_0}
      \bigg( \sum_{j=1}^{\dimA}t_jL_j^{(v)} \bigg)
    \bigg) H^{-1}.
\end{align}
Substituting $\pmb{t}_i$ into
\eqref{eq:0814:0834}, we obtain $\mu_{\II_A(\pmb{t}_i)}(\II_A(\pmb{t}_i))=2$, and further
\begin{align}
\IntProf_{\II_A}A(\pmb{t}_i) & =
  H \bigg(
     \bigoplus_{v\in\Q_0}\bigg(
     L_1^{(v)}+\cdots+L_{i-1}^{(v)}+2L_i^{(v)}+L_{i+1}^{(v)}+\cdots+L_{\dimA}^{(v)}
   \bigg)\bigg)
  H^{-1}
\nonumber \\
& = H \bigg(
     \bigoplus_{v\in\Q_0}
       \bigg(
         \sum_{j=1}^{\dimA}L_j^{(v)}+L_i^{(v)}
     \bigg)\bigg)
  H^{-1} \label{eq:0814:0837} \\
& = 
H
\left(
\begin{matrix}
  \sum\limits_{j=1}^{\dimA}L_j^{(v_1)}+L_i^{(v_1)}
& 0
& \cdots
& 0 \\
  0
& \sum\limits_{j=1}^{\dimA}L_j^{(v_2)}+L_i^{(v_2)}
& \cdots
& 0 \\
  \vdots
& \vdots
& \ddots
& \vdots \\
  0
& 0
& \cdots
& \sum\limits_{j=1}^{\dimA}L_j^{(v_s)}+L_i^{(v_s)}
\end{matrix}
\right)
H^{-1}. \label{eq:0814:0840}
\end{align}
Here, for every diagonal block in \eqref{eq:0814:0840}, we have
\begin{align*}
   \bigg( \sum_{j=1}^{\dimA}a_j+a_i \bigg)p_{v,\ell}
&= \sum_{j=1}^{\dimA}a_jp_{v,\ell} +a_ip_{v,\ell} \\
&= \sum_{j=1}^{\dimA} \sum_{k=1}^{d_v} \lambda_{j,\ell,k}^{(v)}p_{v,k} + \sum_{k=1}^{d_v} \lambda_{i,\ell,k}^{(v)}p_{v,k} \\
&= \sum_{k=1}^{d_v}\bigg(\sum_{j=1}^{\dimA}\lambda_{j,\ell,k}^{(v)}+\lambda_{i,\ell,k}^{(v)}\bigg)p_{v,k}.
\end{align*}
Therefore, the matrix of $L_{\tilde{a}_i}\big|_{Ae_v}$, where ${\tilde{a}_i}=\sum\limits_{j=1}^{\dimA}a_j+a_i$,
with respect to the basis $\basis{Ae_v}$ is
\begin{align*}
\sum\limits_{j=1}^{\dimA}L_j^{(v)}+L_i^{(v)}=
\left( \sum\limits_{j=1}^{\dimA}\lambda_{j,\ell,k}^{(v)} +\lambda_{i,\ell,k}^{(v)} \right)_{k,\ell}
\end{align*}
It follows that $\bigoplus\limits_{v\in\Q_0}\bigg(\sum\limits_{j=1}^{\dimA}L_j^{(v)}+L_i^{(v)}\bigg)$ is the matrix of $\bigoplus\limits_{v\in\Q_0} L_{\tilde{a}_i}\big|_{Ae_v}$ with respect to the ordered basis $\bigcup\limits_{v\in\Q_0}\basis{Ae_v}$ of $A$ where a fixed ordering of $\Q_0$ is understood.
Thus,
\[ L_{\tilde{a}_i} \compos\psi
= \psi\compos \bigg( \bigoplus_{v\in\Q_0} L_{\tilde{a}_i}\big|_{Ae_v}\bigg).\]
Consequently,
\[L_{\tilde{a}_i}= H \bigg( \bigoplus_{v\in\Q_0} \bigg(\sum_{j=1}^{\dimA}L_j^{(v)}+L_i^{(v)}\bigg)\bigg)H^{-1}.\]
Comparing this equality with
\eqref{eq:0814:0837}, we obtain $\IntProf_{\II_A}A(\pmb{t}_i) = L_{\tilde{a}_i}$, and then this lemma holds.
\end{proof}

Define the reconstruction ring of $A$ is
\begin{align*} 
    \recring(A, \pmb{t}_i; 1\=< i \=<\dimA)
  = \mathrm{span}_{\CC}\bigg\{ \IntProfEmb_{\II_A}A(\pmb{t}_i): 1\=< i\=< \dimA \bigg\},
\end{align*}
see \eqref{eqdef:recring} The following result holds.

\begin{fact}\label{fact:regular 0814}
We have
\begin{align*} 
    \recring(A, \pmb{t}_i; 1\=< i \=<\dimA)=\{\mathfrak{Emb}_A(L_a):a\in A\}.
\end{align*}
In particular, the left-hand side is closed under matrix
multiplication and is isomorphic to $A$ as a complex algebra.
\end{fact}

\begin{proof}
By equation~\eqref{eq:regular intprof}, for every $1\=< i\=< \dimA$, we have
\begin{align*}
  \IntProfEmb_{\II_A} A (\pmb{t}_i) = \mathfrak{Emb}_A (L_{\tilde{a}_i})
\end{align*}
by Lemma \ref{lemm:0814:0903}. Notice that $\basis{A}':=\Big\{ \sum\limits_{j=1}^{n}a_j+a_i: 1\=< i\=< n \Big\}$ is a basis of $A$.
Indeed, its transition matrix with respect to the basis $\{a_1,\ldots,a_n\}$ is
\[
\left(
\begin{matrix}
2 & 1 & \cdots & 1\\
1 & 2 & \cdots & 1\\
\vdots & \vdots & \ddots & \vdots\\
1 & 1 & \cdots & 2
\end{matrix}
\right)_{\dimA\times \dimA},
\]
and the determinant of the above matrix is $\dimA+1\ne 0$.
By Fact \ref{fact:ricring iso to A}, we have an isomorphism $A\to \recring_{\partial}(A)$, $a\mapsto\mathfrak{Emb}_A(L_a)$. Thus,
\[\mathfrak{Emb}_A(\basis{A}') = \Big\{\mathfrak{Emb}_A(L_{\tilde{a}_i}):1\=< i\=< n\Big\}\]
is a basis of $\recring_{\partial}(A)$. It follows that
\begin{align*}
  \recring(A, \pmb{t}_i; 1\=< i \=<\dimA) = \recring_{\partial}(A) = \{\mathfrak{Emb}_A(L_a):a\in A\} ~ (\cong A),
\end{align*}
which proves this proposition.
\end{proof}

Fact \ref{fact:regular 0814} means that the following isomorphism
\[ A \cong \bigoplus_{i=1}^{\dimA} \CC\cdot\IntProfEmb_{\II_A} A(\pmb{t}_i)  \]
of $\CC$-linear spaces holds.

\section{Applications}\label{sect:examp}

\subsection{Integral profiles over hereditary Nakayama algebras}
\label{subsect:example-heredi Nakayama}
\textsl{
Let $\Q=\overrightarrow{\mathbb{A}}_n := 1\xrightarrow{\alpha_1}2 \xrightarrow{\alpha_2} \cdots \xrightarrow{\alpha_{n-1}}n$
and let $\overrightarrow{A}=\CC\overrightarrow{\mathbb{A}}_n$ in this section. Then $\overrightarrow{A}$ is a \defines{hereditary Nakayama algebra}.
For every $1\=< i\=< j\=< n$, let
\[ p_{ij} := \begin{cases}
    e_i, & i=j,\\
    \alpha_i\alpha_{i+1}\cdots\alpha_{j-1}, & i<j.
  \end{cases}\]
Thus $p_{ij}$ is the unique path from $i$ to $j$. Under our assumption for the multiplication of paths, we have
$p_{ij}p_{k\ell}= p_{i\ell}$ if $j=k$ and $p_{ij}p_{k\ell} =0 $ for otherwise.
The set $\basis{\overrightarrow{A}} =\{p_{ij}:1\=< i\=< j\=< n\}$ is a basis of $\overrightarrow{A}$, then $\dimA:=\dim \overrightarrow{A}=\frac{n(n+1)}{2}$.
Let $\pmb{t}=(t_{ij})_{1\=< i\=< j\=< n} \in(\RR^{\>= 0})^{\times\dimA}$.
Then $\II_{\overrightarrow{A}}(\pmb{t}) = \prod\limits_{1\=< i\=< j\=< n}[0,t_{ij}]p_{ij}$, and
$\mu_{\II_{\overrightarrow{A}}(\pmb{t})}(\II_{\overrightarrow{A}}(\pmb{t})) = \prod\limits_{1\=< i\=< j\=< n}t_{ij}$. }

\subsubsection{Integral profiles of projective modules}

By Theorem \ref{thm:param-int-module}, for every left $\overrightarrow{A}$-module $M$,
\begin{equation}\label{eq:Nakayama-general-profile}
\IntProf_{\II_{\overrightarrow{A}}}M(\pmb{t}) =
\frac{1}{2}\prod\limits_{1\=< i\=< j\=< n}t_{ij}
\sum_{1\=< i\=< j\=< n} t_{ij}\rho_M(p_{ij}).
\end{equation}
For every $1\=< v\=< n$, the indecomposable projective left $\overrightarrow{A}$-module corresponding to $v$ is $P(v):=\overrightarrow{A}e_v$.
It has the path basis $\basis{P(v)} = \{p_{1v},p_{2v},\ldots,p_{vv}\}$.
Consequently, if $j\=< v$, then $\rho_{P(v)}(p_{ij})$ sends $p_{jv}$ to $p_{iv}$ and sends all the other basis elements to zero. If $j>v$, then $\rho_{P(v)}(p_{ij})=0$.
Therefore, with respect to the ordered basis $(p_{1v},p_{2v},\ldots,p_{vv})$, the integral profile of $P(v)$ is
\begin{equation}\label{eq:Nakayama proj prof}
\IntProf_{\II_{\overrightarrow{A}}}P(v)(\pmb{t})
=
\frac{1}{2} \prod\limits_{1\=< i\=< j\=< n}t_{ij}
\left(
\begin{matrix}
t_{11} & t_{12} & t_{13} & \cdots & t_{1v}\\
   & t_{22} & t_{23} & \cdots & t_{2v}\\
   &   & t_{33} & \cdots & t_{3v}\\
   &   &   & \ddots & \vdots\\
   &   &   &   & t_{vv}
\end{matrix}
\right)_{v\times v}.
\end{equation}
In particular, for the probability integration domain determined by
$t_{ij}=1$ for all $1\=< i\=< j\=< n$, we obtain
\[ \int_{\II_{\overrightarrow{A}}}P(v)
= \frac{1}{2}
\left(\begin{matrix}
 1 &\cdots& 1 \\
   &\ddots&\vdots\\
   &   & 1  \\
\end{matrix}\right)_{v\times v}.\]

\subsubsection{Integral profiles of injective modules}

For every $1\=< v\=< n$, the corresponding indecomposable injective left $\overrightarrow{A}$-module is
\[I(v):=D(e_v\overrightarrow{A})=\Hom_{\CC}(e_v\overrightarrow{A},\CC). \]
The right $\overrightarrow{A}$-module $e_v\overrightarrow{A}$ has the path basis $\basis{e_v\overrightarrow{A}} = \{p_{vv},p_{v,v+1},\ldots,p_{vn}\}$.
Let $\basis{e_v\overrightarrow{A}}^{*}=\{p_{vv}^{*},p_{v,v+1}^{*},\ldots,p_{vn}^{*}\}$ be the corresponding dual basis of $I(v)$.
Recall that the left $\overrightarrow{A}$-action on $I(v)$ is given by $(a\cdot f)(x):=f(xa)$, $a\in \overrightarrow{A},\ f\in I(v),\ x\in e_v\overrightarrow{A}$.
For $v\=< k,\ell\=< n$ and $1\=< i\=< j\=< n$, we have
\begin{align*}
   \bigl(p_{ij}\cdot p_{v\ell}^{*}\bigr)(p_{vk})
 = p_{v\ell}^{*}(p_{vk}p_{ij})
 = \begin{cases}
     1, & k=i\text{ and } \ell = j; \\
     0, & \text{otherwise}.
   \end{cases}
\end{align*}
It follows that
\[ p_{ij}\cdot p_{v\ell}^{*}
 = \begin{cases}
     p_{vi}^{*}, & v\=< i\=< j\=< n \text{ and } \ell=j,\\
     0, & \text{otherwise}.
   \end{cases}\]
Therefore, with respect to the ordered dual basis $(p_{vv}^{*},p_{v,v+1}^{*},\ldots,p_{vn}^{*})$,
the integral profile of $I(v)$ is
\begin{equation}\label{eq:Nakayama-injective-profile}
   \IntProf_{\II_{\overrightarrow{A}}}I(v)(\pmb{t})
 = \frac{1}{2} \prod\limits_{1\=< i\=< j\=< n}t_{ij}
   \left(\begin{matrix}
     t_{vv} & t_{v,v+1} & t_{v,v+2} & \cdots & t_{vn}   \\
       & t_{v+1,v+1} & t_{v+1,v+2} & \cdots & t_{v+1,n}\\
       &  & t_{v+2,v+2} & \cdots & t_{v+2,n} \\
       &  &  &\ddots&\vdots \\
       &  &  & &t_{nn}
   \end{matrix}\right)_{(n-v+1)\times(n-v+1)}.
\end{equation}
In particular, if $t_{ij}=1$ for every $1\=< i\=< j\=< n$, then
\[ \int_{\II_{\overrightarrow{A}}}I(v)
= \frac{1}{2}
\left(\begin{matrix}
 1 &\cdots& 1 \\
   &\ddots&\vdots\\
   &   & 1  \\
\end{matrix}\right)_{(n-v+1)\times(n-v+1)}.\]

\begin{remark}\rm
Every finite-dimensional projective left $\overrightarrow{A}$-module is of the form $P\cong \bigoplus\limits_{v=1}^{n} P(v)^{\oplus m_v}$ for some $m_v\in\NN$. Hence
\[ \IntProf_{\II_{\overrightarrow{A}}}P \sim \bigoplus_{v=1}^{n}
   \bigg(\IntProf_{\II_{\overrightarrow{A}}}P(v)\bigg)^{\oplus m_v}, \]
where each block $\IntProf_{\II_{\overrightarrow{A}}}P(v)$ is given by \eqref{eq:Nakayama proj prof}.
Similarly, every finite-dimensional injective left $\overrightarrow{A}$-module is of the form $I\cong \bigoplus\limits_{v=1}^{n} I(v)^{\oplus r_v}$ for some $r_v\in\NN$, and then
\[ \IntProf_{\II_{\overrightarrow{A}}}I \sim \bigoplus_{v=1}^{n}
   \bigg(\IntProf_{\II_{\overrightarrow{A}}}I(v)\bigg)^{\oplus r_v}, \]
where each block $\IntProf_{\II_{\overrightarrow{A}}}I(v)$ is given by \eqref{eq:Nakayama-injective-profile}.
\end{remark}

\subsubsection{Integral coefficient rings}
In this example, we take the fixed embedding system from Assumption \ref{assu:fix emb sys} to be the following compatible orthogonal embedding system.
The pairwise non-isomorphic indecomposable modules occurring in ${_{\overrightarrow{A}}\pandi}$ are
$P(1)$, $\ldots$, $P(n)$, $I(2)$, $\ldots$, $I(n)$, where $P(n)\cong I(1)$. For every such indecomposable module $X$,
we fix mutually orthogonal copies $X^{(1)}$, $X^{(2)}$, $X^{(3)}$, $\ldots$.
If \[ M\cong \bigoplus_{v=1}^{n}P(v)^{\oplus m_v} \oplus \bigoplus_{v=2}^{n}I(v)^{\oplus r_v},\]
we choose $\sigma_M$ so that the summands of $M$ are embedded into the first $m_v$ copies of $P(v)$ and the first $r_v$ copies of
$I(v)$, respectively.

For every indecomposable module
\[ X \in \{P(1),\ldots,P(n),I(2),\ldots,I(n)\} \]
and every $r \>= 1$, let $F_{X,r} \in \ol{\pandiring}(\overrightarrow{A})$ be the function whose restriction to $X^{(r)}$ is
$\IntProf_{\II_{\overrightarrow{A}}}X$ and whose restriction to all the other fixed summands is zero.
Since the fixed summands are mutually orthogonal, we have
\[ F_{X,r}F_{Y,s}=0 \quad \text{whenever }(X,r)\ne(Y,s).\]
For each $X$ and $r$, the element $F_{X,r}$ satisfies no nonzero polynomial relation over $\CC$. Indeed, one of its diagonal entries is a nonconstant polynomial of the form $\frac{1}{2}\big( \prod\limits_{1 \=< i \=< j \=< n}t_{ij} \big)t_{kk}$.
Consequently,
\[ \langle F_{X,r}\rangle = \bigoplus_{m \>= 1}\CC F_{X,r}^{m} \cong x\CC[x] \]
as non-unital complex algebras.

For every $r \>= 1$, the difference between the embedded integral
profiles of $X^{\oplus r}$ and $X^{\oplus(r-1)}$ is precisely
$F_{X,r}$. Hence every $F_{X,r}$ belongs to $\pandiring(\overrightarrow{A})$.
Conversely, by the compatibility of the fixed embedding system,
the embedded integral profile of every module in ${_{\overrightarrow{A}}\pandi}$ is
a finite sum of elements of the form $F_{X,r}$. We therefore obtain
\[ \pandiring(\overrightarrow{A})
 = \bigoplus_{\substack{ X\in\{P(1),\ldots,P(n),I(2),\ldots,I(n)\} \\ r \>= 1}}
   \langle F_{X,r}\rangle
 \cong \bigoplus_{r \>= 1} \left(x\CC[x]\right)^{\oplus(2n-1)} \]
as non-unital complex algebras.
If the subalgebra generated by the integral profiles is required to contain the identity of $\ol{\pandiring}(\overrightarrow{A})$, then $\pandiring(\overrightarrow{A})$ is the unitization of the above algebra. Namely,
\[ \pandiring(\overrightarrow{A}) \cong \CC~\ident_{\pandimod^{\oplus\NN}} \oplus \bigoplus_{r \>= 1} \left(x\CC[x]\right)^{\oplus(2n-1)}.\]

\subsection{Application: homological dimensions of gentle algebras} \label{subsect:gentle hdim}

\textsl{Gentle algebras were introduced by Assem and Skowro\'nski and form an important class of finite-dimensional special biserial algebras; see \cite{AS1987}. The combinatorial conditions in their definition make it possible to describe their projective modules, injective modules and homological dimensions explicitly in terms of paths in their bound quivers.}

\subsubsection{Gentle algebras} \label{subsubsect:gentle}

Throughout this section, we follow the convention of \cite{ASS2006} for the multiplication of paths.
Thus, if $\alpha:i\to j$ is an arrow, then $e_i\alpha=\alpha=\alpha e_j$.
If $\alpha:i\to j$ and $\beta:j\to k$ are arrows, then their composition is written as $\alpha\beta:i\to k$.

Let $A=\CC\Q/\I$ be a finite-dimensional bound quiver algebra.
Recall that $A$ is called a \defines{gentle algebra} if the
following conditions are satisfied:
\begin{enumerate}[label={\rm(G\arabic*)}]
  \item For every $v\in\Q_0$, there are at most two arrows
    starting at $v$ and at most two arrows ending at $v$.
    \label{G1}
  \item For every $\alpha\in\Q_1$,
    there is at most one arrow $\beta\in\Q_1$ such that
    $\fctt(\alpha)=\fcts(\beta)$ and $\alpha\beta\notin\I$,
    and there is at most one arrow $\beta'\in\Q_1$ such that
    $\fctt(\alpha)=\fcts(\beta')$ and $\alpha\beta'\in\I$.\label{G2}
  \item For every $\alpha\in\Q_1$,
    there is at most one arrow $\gamma\in\Q_1$ such that
    $\fctt(\gamma)=\fcts(\alpha)$ and $\gamma\alpha\notin\I$,
    and there is at most one arrow $\gamma'\in\Q_1$ such that
    $\fctt(\gamma')=\fcts(\alpha)$ and $\gamma'\alpha\in\I$.\label{G3}
  \item The ideal $\I$ is generated by paths of length two.
    \label{G4}
\end{enumerate}

Since $\I$ is generated by paths, the residue classes of all paths which do not belong to $\I$ form a complex basis of $A$.
For every $v\in\Q_0$, the indecomposable projective left $A$-module $Ae_v$ has a basis consisting of all nonzero paths ending at $v$.
Similarly, the right ideal $e_vA$ has a basis consisting of all nonzero paths starting at $v$, and the corresponding indecomposable injective left $A$-module is $D(e_vA)=\Hom_{\CC}(e_vA,\CC)$.

In particular, the projective and injective integral profiles over a gentle algebra can be calculated explicitly by using these path bases.
Moreover, the gentle conditions \ref{G1}--\ref{G4} determine the successive terms of the minimal projective resolutions of the simple modules, which will allow us to relate their total projective integral profiles to the global dimension of $A$, see Theorem \ref{thm:gentle gldim} in Subsection \ref{subsubsect:gldim gentle}.

\begin{remark}\rm
In \cite{BD2017}, the authors show that every gentle algebra is a quotient of a direct product of hereditary Nakayama algebras.
Therefore, gentle algebras can be studied by hereditary Nakayama algebras. The construction of the graded geometric models of (graded) gentle algebras in \cite[etc]{OPS2018,QZZ2022} relies on this fact.
\end{remark}

\subsubsection{Global dimensions of gentle algebras} \label{subsubsect:gldim gentle}

For every $v\in \Q_0$ and $m \>= 1$, let
\[ \mathcal{F}_A(v,m) :=
\left\{
  \alpha_1\cdots\alpha_m:
    \begin{array}{l}
      \alpha_1,\ldots,\alpha_m\in \Q_1, ~ \fctt(\alpha_m)=v,\\
      \alpha_i\alpha_{i+1}\in \I \text{ for every }1\=< i < m
    \end{array}
\right\}.\]
The elements of $\mathcal{F}_A(v,m)$ are not required to be maximal. For convenience, put
\[ \mathcal{F}_A(v,0):=\{e_v\}. \]
Moreover, for each polynomial $f(x)=\sum_ta_tx^t$, we use $[x^m]f(x)$ to represent the coefficient $a_m$ of the term $a_mx^m$ of a polynomial to the power of $m$.

\begin{lemma}\label{lemm:gentle corresp}
Let $A=\CC Q/I$ be a gentle algebra and let $v\in Q_0$ and $\pmb{P}_{S_v}^{\bullet} = (P_{S_v}^m,d_{S_v}^m)_{m \>= 0}$
be the minimal projective resolution of the simple left $A$-module $S_v$. Then the following statements hold.
\begin{enumerate}[label={\rm(\arabic*)}]
  \item For every $m \>= 0$, we have
    \[ P_{S_v}^m \cong \bigoplus_{\alpha_1\cdots\alpha_m\in\mathcal F_A(v,m)} Ae_{s(\alpha_1)}; \]
    furthermore, we have
    \[ \IntProf_{\II_A}P_{S_v}^m \sim
       \bigoplus_{\alpha_1\cdots\alpha_m\in\mathcal F_A(v,m)}\IntProf_{\II_A}Ae_{s(\alpha_1)}. \]
    Here, $P_{S_v}^0 = Ae_v$.
  \item The diagonal blocks of the coefficient
    \[[x^m]\IntProfEmb_{\II_A} (\pmb{P}_{S_v}^{\bullet};x) =
        \mathfrak{Emb}_{P_{S_v}^m}
        \left(
        \begin{matrix}
          \displaystyle\IntProf_{\II_A}Ae_{s(\alpha_{1,1})}
        & 0
        & \cdots
        & 0\\
          0
        &  \displaystyle\IntProf_{\II_A}Ae_{s(\alpha_{2,1})}
        & \cdots
        & 0\\
          \vdots
        & \vdots
        & 
        & \vdots\\
          0
        & 0
        & \cdots
        & \displaystyle\IntProf_{\II_A}Ae_{s(\alpha_{r,1})}
        \end{matrix}\right)\]
    of $\intProfEmb_{\II_A} (\pmb{P}_{S_v}^{\bullet};x)$ are in one-to-one correspondence with the forbidden paths in $\mathcal F_A(v,m)$.
\end{enumerate}
\end{lemma}

\begin{proof}
First, we have $P_{S_v}^0=Ae_v$. Since the nontrivial paths ending at $v$ form a path basis of $\rad(Ae_v)$, we have
\[\rad(Ae_v) = \bigoplus_{\substack{\alpha\in \Q_1\\ \fctt(\alpha)=v}}A\alpha,\]
and then
\[P_{S_v}^1 \cong \bigoplus_{\substack{\alpha\in \Q_1\\ \fctt(\alpha)=v}} Ae_{s(\alpha)}
= \bigoplus_{\alpha\in\mathcal F_A(v,1)} Ae_{s(\alpha)}. \]
More precisely, for every arrow $\alpha:u\to v$, right multiplication by $\alpha$ defines a projective cover $Ae_u \to A\alpha$, $x\mapsto x\alpha.$
By \ref{G4}, a nonzero path $p$ ending at $u$ satisfies $p\alpha=0$ if and only if the final
arrow $\beta$ of $p$ satisfies $\beta\alpha\in\I$.
The conditions \ref{G1}--\ref{G3} imply that there is at most one such arrow $\beta$. Hence
\[ \ker(-\compos\alpha) =
\begin{cases}
  A\beta, & \text{if there is an arrow $\beta$ such that } \beta\alpha = 0_A;\\
  0, & \text{otherwise}.
\end{cases}\]
Inductively, every sequence $\alpha_1\cdots\alpha_m\in\mathcal F_A(v,m)$
provides precisely one projective summand $Ae_{s(\alpha_1)}$ of $P_{S_v}^m$. Then
\[ P_{S_v}^m \cong \bigoplus_{\alpha_1\cdots\alpha_m\in\mathcal F_A(v,m)} Ae_{s(\alpha_1)}.\]
By Theorem \ref{thm:param-int-module iso} and Corollary \ref{coro:param-int-dir sum}, we have
\[ \IntProf_{\II_A}P_{S_v}^m \sim \bigoplus_{\alpha_1\cdots\alpha_m\in\mathcal F_A(v,m)}
\IntProf_{\II_A}Ae_{s(\alpha_1)}. \]
Thus the diagonal blocks of $[x^m]\intProfEmb_{\II_A} (\pmb P_{S_v}^{\bullet};x)$ are in one-to-one correspondence with the forbidden paths of length $m$ ending at $v\in\Q_0$.
\end{proof}

For any gentle algebra $A=\CC\Q/\I$ and arbitrary two arrows $\alpha$ and $\beta$ in $\Q$, let
\[ g_{\alpha,\beta} =
\begin{cases}
1, & t(\alpha)=s(\beta) \text{ and } \alpha\beta=0_A \\
0, & \text{otherwise}.
\end{cases} \]
Then the matrix $G_A :=(g_{\alpha,\beta})_{(\alpha,\beta)\in\Q_1\times\Q_1}$ lies in $\pandiring(A)$.

\begin{lemma}\label{lemm:gentle polynomial}
Keep the notations from Lemma \ref{lemm:gentle corresp}. Then
\[ \IntProf (\pmb{P}^{\bullet}_{S_v};x) \widetilde{\in} \pandiring(A)[x]
\text{ if and only if the lengths all paths } \wp \in \bigcup_{m\in\NN}\mathcal{F}_A(v,m) \text{ are bounded. }  \]
If these equivalent conditions hold, then
\[ \deg_x \IntProfEmb_{\II_A} (\pmb P_{S_v}^{\bullet};x) = \max \{ m\>= 0: \mathcal F_A(v,m)\ne\varnothing \}.\]
\end{lemma}

\begin{proof}
By Lemma \ref{lemm:gentle corresp}, we obtain that:
\begin{itemize}
  \item the coefficient $[x^m]\intProfEmb_{\II_A} (\pmb P_{S_v}^{\bullet};x)$
    is a block-diagonal matrix whose diagonal blocks are indexed by $\mathcal{F}_A(v,m)$;
  \item every diagonal block is the embedded integral profile of a nonzero projective module $Ae_{\fcts(\alpha_1)}$,
    and its integral profile is nonzero.
\end{itemize}
Thus, we have
\[[x^m]\IntProfEmb_{\II_A} (\pmb P_{S_v}^{\bullet};x)\ne0
\text{ if and only if } \mathcal{F}_A(v,m)\ne\varnothing.\]
It follows that the total projective integral profile is a polynomial if and only if
$\mathcal{F}_A(v,m)$ is empty for all sufficiently large $m$.
Equivalently, the lengths of the forbidden paths ending at $v$ are bounded.
In this case, its degree is the maximal integer $m$ satisfying $\mathcal{F}_A(v,m)\ne\varnothing$.
\end{proof}

\begin{proposition}\label{prop:gldim gent}
Let $A=\CC\Q/\I$ be a gentle algebra. Then the following statements hold.
\begin{enumerate}[label={\rm(\arabic*)}]
  \item If $\Q_1=\varnothing$, then $\gldim A=0$.
    \label{prop:gldim gent 1}
  \item If $\Q_1\ne\varnothing$. Then
    \[ \gldim A<\infty \text{  if and only if }
     \IntProf_{\II_A} (\pmb P_{S_v}^{\bullet};x) \widetilde{\in} \pandiring(A)[x] ~ (\forall v\in Q_0) \]
    \label{prop:gldim gent 2}
  \item If $\gldim A<\infty$, then
    \begin{align*}
       \gldim A
     & = \max_{v\in Q_0} \deg_x \IntProfEmb_{\II_A} (\pmb P_{S_v}^{\bullet};x) \\
     & = \max \{ m\>= 1: \mathcal{F}_A(v,m)\ne\varnothing \text{ for some } v\in Q_0 \}
    \end{align*}
    \label{prop:gldim gent 3}
\end{enumerate}
\end{proposition}

\begin{proof}
If $\Q_1=\varnothing$, then $A$ is semisimple, and hence $\gldim A=0$, thus \ref{prop:gldim gent 1} holds.
Now assume $\Q_1\ne\varnothing$. By Lemma \ref{lemm:gentle corresp}, for every $v\in \Q_0$,
$\IntProf(\pmb{P}^{\bullet}_{S_v};x) \widetilde{\in} \pandiring(A)[x]$
if and only if the lengths of all paths $\wp \in \bigcup_{m\in\NN}\mathcal{F}_A(v,m)$ are bounded.
By Theorem \ref{thm:gldim-int}, we obtain \ref{prop:gldim gent 3} immediately.

Moreover, Lemma \ref{lemm:pdim-int 0812} provides the minimal projective resolution of $S_v$, and hence
\[ \pdim_A S_v = \deg_x \IntProfEmb_{\II_A} (\pmb P_{S_v}^{\bullet};x). \]
Then
\[\gldim A = \max_{v\in \Q_0} \deg_x\IntProfEmb_{\II_A}(\pmb P_{S_v}^{\bullet};x)\]
holds by using \eqref{eq:gldim-simp mod}.
\end{proof}

\begin{theorem}\label{thm:gentle gldim}
Let $A$ be a gentle algebra. Then the following statements hold.
\begin{enumerate}[label={\rm(\arabic*)}]
  \item If $\Q_1=\varnothing$, then $\gldim A=0$.
    \label{thm:gentle gldim 1}
  \item If $\Q_1\ne\varnothing$ and $G_A$ is nilpotent, then $\gldim(A) = \min\{m\>= 1: G_A^m=0\}$.
    \label{thm:gentle gldim 2}
  \item If $G_A$ is not nilpotent, then $\gldim A=\infty$.
    \label{thm:gentle gldim 3}
\end{enumerate}
\end{theorem}

\begin{proof}
It is clearly that \ref{thm:gentle gldim 1} holds by Proposition \ref{prop:gldim gent}.

Now assume that $\Q_1\ne\varnothing$. For every $m \>= 2$ and
$\alpha_1,\alpha_m\in\Q_1$, the
$(\alpha_1,\alpha_m)$-entry of $G_A^{m-1}$ is
\[\sum_{\alpha_2,\ldots,\alpha_{m-1}\in\Q_1} g_{\alpha_1,\alpha_2} g_{\alpha_2,\alpha_3} \cdots g_{\alpha_{m-1},\alpha_m}.\]
Since every $g_{\alpha,\beta}$ is either zero or one,
a summand $g_{\alpha_1,\alpha_2}g_{\alpha_2,\alpha_3} \cdots g_{\alpha_{m-1},\alpha_m}$
is nonzero if and only if $\alpha_i\alpha_{i+1}=0_A$ for every $1\=< i<m$.
In this case, $\alpha_1\cdots\alpha_m \in \mathcal{F}_A(\fctt(\alpha_m),m)$.
Conversely, every element $\alpha_1\cdots\alpha_m \in \mathcal{F}_A(v,m)$ provides a nonzero summand
in the $(\alpha_1,\alpha_m)$-entry of $G_A^{m-1}$.
Therefore,
\begin{equation}\label{eq:gentle matrix path}
G_A^{m-1}\ne 0 \text{~if and only if~} \mathcal{F}_A(v,m) \ne \varnothing \text{ for some } v\in\Q_0.
\end{equation}
For $m=1$, this equivalence still holds because $G_A^0$ is the identity matrix and $\Q_1\ne\varnothing$.

Suppose that $G_A$ is nilpotent, and let $d = \min\{m \>= 1:G_A^m=0\}$.
By \eqref{eq:gentle matrix path}, there exists $v\in\Q_0$ such that $\mathcal{F}_A(v,d)\ne\varnothing$,
whereas $\mathcal{F}_A(v,m)=\varnothing$ for every $v\in\Q_0$ and every $m>d$.
By Lemma~\ref{lemm:gentle corresp}, $[x^m]\intProfEmb_{\II_A}(\pmb P_{S_v}^{\bullet};x)\ne 0$ if and only if $\mathcal{F}_A(v,m)\ne\varnothing$.
Them all $\intProfEmb_{\II_A}(\pmb P_{S_v}^{\bullet};x)$ belong to $\pandiring(A)[x]$, and
\[\max_{v\in\Q_0}\deg_x \IntProfEmb_{\II_A}(\pmb P_{S_v}^{\bullet};x)=d.\]
It follows from Proposition \ref{prop:gldim gent} that $\gldim A=d = \min\{m \>= 1:G_A^m=0\}$. Thus, \ref{thm:gentle gldim 2} holds.

Finally, suppose that $G_A$ is not nilpotent. Then $G_A^m\ne 0$ for all $m\in\NN^+$.
By \eqref{eq:gentle matrix path}, for every $m \>= 1$, there exists
$v\in\Q_0$ such that $\mathcal{F}_A(v,m)\ne\varnothing$.
Consequently, the coefficients of arbitrarily large powers of $x$ occur among
the total projective integral profiles $\intProfEmb_{\II_A} (\pmb P_{S_v}^{\bullet};x)$.
i.e., at least one of these total projective integral profiles lies in
$\pandiring(A)[[x]]\backslash \pandiring(A)[x]$.
It follows from Proposition \ref{prop:gldim gent} that $\gldim A=\infty$. Thus, {\rm(3)} holds.
\end{proof}

\subsubsection{Self-injective dimensions of gentle algebras} \label{subsubsect:findim gentle}

For every arrow $\alpha\in\Q_1$, we say that $\alpha$ belongs to the \defines{cyclic part} of $G_A$ if
$(G_A^m)_{\alpha,\alpha} \ne 0$ for some $m \>= 1$. Let
\[\Q_1^{\mathrm{fin}} = \{ \alpha\in\Q_1: (G_A^m)_{\alpha,\alpha}=0 \text{ for every }m\>= 1 \},\]
and let
\[G_A^{\mathrm{fin}} = (g_{\alpha,\beta})_ {\alpha,\beta\in\Q_1^{\mathrm{fin}}} \in \pandiring(A)\]
be the \defines{principal submatrix} of $G_A$ indexed by $\Q_1^{\mathrm{fin}}$.

\begin{theorem}\label{thm:gentle findim}
Let $A=\CC\Q/\I$ be a gentle algebra. Then $G_A^{\mathrm{fin}}$ is nilpotent and
\[ \idim A =
\begin{cases}
0, & \Q_1^{\mathrm{fin}}=\varnothing, \\
\min\{ m\>= 1: (G_A^{\mathrm{fin}})^m=0 \},
& \Q_1^{\mathrm{fin}}\ne\varnothing.
\end{cases}\]
\end{theorem}

\begin{proof}
By the definition of $G_A$, for every $m\>=2$ and $\alpha_1,\alpha_m\in\Q_1$,
the $(\alpha_1,\alpha_m)$-entry of $G_A^{m-1}$ is
\[\sum_{\alpha_2,\ldots,\alpha_{m-1}\in\Q_1} g_{\alpha_1,\alpha_2} g_{\alpha_2,\alpha_3} \cdots g_{\alpha_{m-1},\alpha_m}.\]
A summand in the above expression is nonzero if and only if $\alpha_i\alpha_{i+1}=0_A$ for every $1\=< i<m$.
The definition of $\Q_1^{\mathrm{fin}}$ removes precisely those
arrows $\alpha$ for which a sequence $\alpha=\alpha_1$, $\alpha_2$, $\ldots$, $\alpha_m=\alpha$
satisfying $\alpha_i\alpha_{i+1}=0_A$ $(1\=< i<m)$ exists.
Therefore, no such sequence can return to its initial arrow when all its arrows belong to $\Q_1^{\mathrm{fin}}$.
It follows that $G_A^{\mathrm{fin}}$ is nilpotent.

Suppose first that $\Q_1^{\mathrm{fin}}\ne\varnothing$, and put $d:=\min\{m\>= 1: (G_A^{\mathrm{fin}})^m=0\}$.
There exist arrows $\alpha_1$, $\ldots$, $\alpha_d\in\Q_1^{\mathrm{fin}}$ such that $\alpha_i\alpha_{i+1}=0_A$ $(1\=< i<d)$,
whereas no such sequence consisting of $d+1$ arrows in $\Q_1^{\mathrm{fin}}$ exists.
Since $A$ is gentle, every row and every column of $G_A$ contains at most one nonzero entry.
Hence every sequence of arrows satisfying the above zero relations either eventually returns to a previously occurring arrow or terminates after finitely many steps.
The former arrows are exactly those removed in the definition of $\Q_1^{\mathrm{fin}}$. Thus $d$ is the maximal length of all finite maximal sequences $\alpha_1$, $\ldots$, $\alpha_d$ satisfying $\alpha_i\alpha_{i+1}=0_A$ $(\forall 1\=< i<d)$.
Every gentle algebra is Gorenstein \cite{GR2005}. Moreover, its left and right self-injective dimensions are equal to this maximal length; see, for example, \cite[Theorem 1.1 (2)]{LGH2024}. Therefore, $\idim {_A A}=\idim A_A=d$.

Finally, suppose that $\Q_1^{\mathrm{fin}}=\varnothing$. Then every arrow belongs to the cyclic part of $G_A$.
If $\Q_1=\varnothing$, then $A$ is semisimple, while otherwise each relevant component consists entirely of arrows returning cyclically under the zero relations. In this case $A$ is self-injective on these components, and $\idim {_A A} = \idim A_A = 0$.
Since a self-injective finite-dimensional algebra has no non-projective modules of finite projective dimension, we obtain $\idim A=0$.
\end{proof}


\section{An algebra to be syzygy infinite, non self-injective, non monomial and Loewy length of 4}
\label{appendix 3: an alg}
{\addtext

We construct a connected non-monomial algebra $A$ which is finite-dimensional, and we show that:
\begin{itemize}
  \item $\rad^4(A)=0\ne\rad^3(A)$;
  \item $A$ is not self-injective;
  \item $A$ is syzygy-infinite;
  \item $A$ is connected;
  \item $A$ is non-monomial;
  \item and $A$ has an infinite global dimension.
\end{itemize}

Let $C=\CC\Q/\I$, where $\Q$ is the quiver
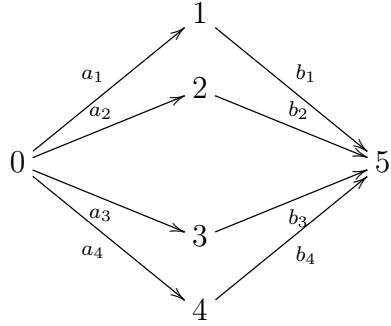
\[
\xymatrix@C=2cm@R=0.5cm{
 & 1 \ar[rdd]^{b_1} & \\
 & 2 \ar[rd]^{b_2} & \\
0 \ar[ruu]^{a_1}
  \ar[ru]^{a_2}
  \ar[rd]_{a_3}
  \ar[rdd]_{a_4}
 & & 5 \\
 & 3 \ar[ru]_{b_3} & \\
 & 4 \ar[ruu]_{b_4} &
}
\]
and
\begin{equation}\label{eq:canonical relations two sided}
 I=\langle
 a_3b_3-a_1b_1-a_2b_2,\,
 a_4b_4-a_1b_1-2a_2b_2
 \rangle.
\end{equation}
We compose paths from left to right.
Thus $C$ is a canonical algebra of weight type $(2,2,2,2)$, and it has dimension $16$.
Write $R=\rad(C)$, then $R^3=0\ne R^2$. Define
\[ B=T(C):=C\ltimes D(C), \]
where $D=\Hom_{\CC}(-,\CC)$, and $D(C)$ has its natural $(C,C)$-bimodule structure $(cfd)(x)=f(dxc)$.
Thus, we have $(c,f)(d,g)=(cd,cg+fd)$ in $B$. The algebra $B$ is symmetric and has dimension $32$.
Moreover,
\begin{equation}\label{eq:B period}
 \Omega^4_{B\otimes_{\CC}B^{\mathrm{op}}}(B)\cong B.
\end{equation}
Indeed, \cite[Example 8.12 (a)]{CDIM2025periodic} shows that $C$ has a fractional Calabi--Yau relation with $(m,\ell)=(2,2)$.
By \cite[Theorem 6.1(a)]{CDIM2025periodic}, we have
$\Omega^{m+\ell}_{B\otimes_{\CC}B^{\mathrm{op}}}(B) \cong {}_1B_\varphi$,
$\varphi(c,f)=(c,(-1)^{m+\ell}f)$.
Since $m+\ell=4$, we have $\varphi=\mathrm{id}_B$, and \eqref{eq:B period} holds.
Next, let $S$ be the simple left $B$-module at the vertex $0$,
and define $U=S^{\oplus 2}$. For $0\=<i\=<3$, write $U_i=\Omega_B^i(U)$,
and extend the subscripts periodically modulo $4$.
This notation is justified by \eqref{eq:B period}.
Choose a primitive idempotent $e_j$ of $B$ such that $Be_j \le_{\oplus} P_{U_3}^0$,
i.e., $Be_j$ occurs in the projective cover of $U_3$.
Let $W$ be the simple right $B$-module corresponding to $e_j$.
Let $E=\CC^{1\times2}$ be the natural right $M_2(\CC)$-module and
let $E'=\CC^{2\times1}$ be the natural left $M_2(\CC)$-module. Set
\[ V=U\otimes_{\CC}E, \quad\text{and}\quad V'=E'\otimes_{\CC}W.\]
Then $V$ is a $(B,M_2(\CC))$-bimodule and $V'$ is an $(M_2(\CC),B)$-bimodule.
Now, define
\begin{equation}\label{eq:A 2609100913}
 A=
 \begin{pmatrix}
  B & V\\
  V'& M_2(\CC)
 \end{pmatrix},
\end{equation}
where the two bimodule homomorphisms $V\otimes_{M_2(\CC)}V'\to B$ and $V'\otimes_BV\to M_2(\CC)$,
called the Morita-context pairings (see, for example, \cite{TLZ2014Morita} for the definition), are zero.
Thus, the multiplication in $A$ is given by
\[ \begin{pmatrix}b&v\\v'&d\end{pmatrix}
   \begin{pmatrix}c&w\\w'&e\end{pmatrix}
 = \begin{pmatrix}
     bc&bw+ve\\
     v'c+dw'&de
   \end{pmatrix}.\]
Here, we have $\dim A=32+4+2+4=42$.

\begin{itemize}
  \item We first show that $A$ has \textbf{Loewy length $4$, that is, $\rad^4(A)=0\ne\rad^3(A)$}. Let $J_B=\rad(B)$.
  Since $J_BU=0$ and $WJ_B=0$, the definition of the multiplication provides
\[
 \rad(A)=
 \begin{pmatrix}J_B&V\\V'&0\end{pmatrix},
 \text{ and }
 \rad^n(A)=
 \begin{pmatrix}J_B^n&0\\0&0\end{pmatrix}
 (n\>= 2).
\]
Furthermore,
\[ J_B^n=R^n+ \sum_{p+q=n-1}R^pD(C)R^q.\]
Since $R^3=0$, we have $J_B^4=0$. If $0\ne c\in R^2$ and $f\in D(C)$ satisfies $f(c)\ne0$, then $cf \ne 0$.
Hence $J_B^3 \ne 0$. Therefore, $\rad^4(A)=0\ne\rad^3(A)$.

\item The algebra $A$ is \textbf{connected}: Indeed, $B$ is connected, and either one of the nonzero bimodules $V$ and $V'$ connects the simple $M_2(\CC)$-block to the simple blocks of $B$.

\item We next show that $A$ is \textbf{non-monomial} in the Morita-invariant sense:
Let $E_{11}$ be a matrix unit and define $e=\left(\begin{smallmatrix}1_B&0\\0&E_{11}\end{smallmatrix}\right)$.
This is a full idempotent. In the ordinary quiver of $eAe$, consider only the three branches $0\xrightarrow{a_i}i\xrightarrow{b_i}5$ $(i \in \{1,2,3\})$ of quiver $Q$.
The six arrow spaces are one-dimensional.
In the second radical layer,
the three distinct nonzero paths $q_i:=a_ib_i$ satisfy that $q_3=q_1+q_2$ and $q_1,q_2$ are linearly independent.
This relation cannot be removed by changing arrow representatives.
Indeed, such a change does not change their products modulo the third radical power,
except for nonzero scalar factors. Distinct surviving paths are linearly independent in a monomial algebra.
Therefore, $eAe$ admits no monomial presentation.
It follows that $A$ is not Morita equivalent to a monomial algebra.

\item The algebra $A$ is \textbf{not self-injective}:
Let $\varepsilon= \left(\begin{smallmatrix}0&0\\0&E_{11}\end{smallmatrix}\right)$.
Then the module $A\varepsilon$ is indecomposable and projective,
and we have $\mathrm{soc}(A\varepsilon)\cong S^{\oplus2}$,
which is not simple, and so $A\varepsilon$ is not injective.
Since $A\varepsilon$ is a direct summand of the left regular module ${}_AA$, it would be injective if $A$ were self-injective. Therefore, $A$ is not self-injective.

\item We now \textbf{compute the syzygies}. If an $A$-module has zero $M_2(\CC)$-component,
we write it as $(X,0)$, where $X$ is its $B$-component.
Let $L$ be the simple left $A$-module with zero $B$-component and with the natural simple $M_2(\CC)$-module as its $M_2(\CC)$-component.
Let $e_i$ be a primitive idempotent of $B$ and let $\widetilde e_i= \left( \begin{smallmatrix}e_i&0 \\ 0&0\end{smallmatrix}\right)$.
The indecomposable projective modules coming from $B$ are $A\widetilde e_i$.
The relations $V'J_B=0$ and the zero Morita-context pairing $V\otimes_{M_2(\CC)}V'\to B$ give
\begin{equation}\label{eq:rad old proj}
 \rad(A\widetilde e_i)
 \cong (J_Be_i,0)\oplus(0,V'e_i).
\end{equation}
Since $V'=E'\otimes_{\CC}W$ and $W$ is simple, the second summand in \eqref{eq:rad old proj} is either zero or a copy of $L$.
For the projective module $A\varepsilon$ coming from the $M_2(\CC)$-block,
we have \begin{equation}\label{eq:rad new proj}
 \rad(A\varepsilon)\cong(U,0).
\end{equation}
It follows from \eqref{eq:rad old proj} and \eqref{eq:rad new proj} that the radical of every projective $A$-module has the form
\begin{equation}\label{eq:rad arbit proj}
 (Z,0)\oplus L^{\oplus q}, \quad  Z\in\add (\{J_Be_i:i\}\cup\{U\}).
\end{equation}
Here, although $V$ and $V'$ are nonzero bimodules, their induced actions on the radical of every projective $A$-module are zero.
Indeed, the action of $V$ on each summand $V'e_i$ is zero because $V\otimes_{M_2(\CC)}V' \to B$ is zero;
and moreover, $V'J_B=0$, and the action of $V'$ on $U$ is zero because $V'\otimes_BV\to M_2(\CC)$ is zero.
Thus, the $B$-component and the $M_2(\CC)$-component of the radical are $A$-submodules.
Therefore, we have $\rad(P)\cong (Z,0)\oplus L^{\oplus q}$ as an $A$-module.
Therefore, every submodule of it is the direct sum of its $B$-component and its $M_2(\CC)$-component.
The latter component is a direct sum of copies of $L$, because $M_2(\CC)$ is semisimple.

Let $P_M\to M$ be a projective cover. Its kernel is contained in $\rad(P_M)$, and
\begin{equation}\label{eq:first syzygy form two sided}
 \Omega_A^1(M) \cong (Y,0)\oplus L^{\oplus q}
\end{equation}
for some $q\>= 0$ and some $B$-module $Y$.
Moreover, $Y$ is contained in a direct sum of modules $J_Be_i$ and copies of $U$.
In this sense, $Y$ has no projective direct summand.
Indeed, suppose that $Y$ has a nonzero projective direct summand $P$,
then, by $B$ to be self-injective, $P$ is injective,
and then the inclusion $P\hookrightarrow Z$ splits. It follows that $P$ is a direct summand of $Z$.
This is impossible because neither $J_Be_i$ nor $U$ has a projective direct summand.
For a left $B$-module $X$, let $\nu(X)$ be the multiplicity of $Be_j$ in its projective cover.
Lifting a projective cover of $Y$ over $B$ to a projective cover of $(Y,0)$ over $A$,
and using \eqref{eq:rad old proj}, provides
\begin{equation}\label{eq:syzygy B part two sided}
 \Omega_A^1(Y,0) \cong(\Omega_B^1(Y),0)\oplus L^{\oplus\nu(Y)}.
\end{equation}
Similarly, \eqref{eq:rad new proj} provides
\begin{equation}\label{eq:syzygy new simple two sided}
 \Omega_A^1(L)\cong(U,0).
\end{equation}
Let $\mu_i=\nu(U_i)$, $0\=< i\=< 3$.  Our choice of $e_j$ gives $\mu_3>0$.
On the fixed modules $U_0,U_1,U_2,U_3,L$, the syzygy operation is described by the matrix
\begin{equation}\label{eq:fixed transition matrix two sided}
 T=
 \begin{pmatrix}
 0&0&0&1&1\\
 1&0&0&0&0\\
 0&1&0&0&0\\
 0&0&1&0&0\\
 \mu_0&\mu_1&\mu_2&\mu_3&0
 \end{pmatrix}.
\end{equation}
For the module $Y$ in \eqref{eq:first syzygy form two sided}, the
four modules $Y,\ \Omega_B^1(Y),\ \Omega_B^2(Y)$ and $\Omega_B^3(Y)$ are cyclically permuted.
The multiplicities with which they produce $L$ may depend on $Y$,
but they occur only in an off-diagonal block.
Thus the transition matrix for all the summands of the successive
syzygies has block triangular form
\[
 \mathcal T_Y=
 \begin{pmatrix}C_4&0\\
 H_Y&T\end{pmatrix},
\]
where $C_4=
\left(\begin{smallmatrix}
0&0&0&1\\
1&0&0&0\\
0&1&0&0\\
0&0&1&0
\end{smallmatrix}\right)$.
In particular, its characteristic polynomial is independent of $Y$ and equals
\begin{equation}\label{eq:uniform characteristic polynomial two sided}
 \chi(z)=(z^4-1)
 \bigl(z^5-\mu_0z^3-\mu_1z^2-(\mu_2+1)z-\mu_3\bigr).
\end{equation}
Write $\chi(z)=z^9+c_8z^8+\cdots+c_1z+c_0$. Since $\mu_3>0$, we have $c_0=\mu_3\ne0$.
By the Cayley--Hamilton theorem and the additivity of projective covers,
\begin{equation}\label{eq:proj recur 260911}
  [P_M^{n+9}]+c_8[P_M^{n+8}]+\cdots+c_1[P_M^{n+1}]+c_0[P_M^n]=0
\end{equation}
in the split Grothendieck group of projective $A$-modules for every $n\>= 1$.
Notice that the coefficients in this recurrence are independent of $M$.

  \item The algebra $A$ is \textbf{syzygy-infinite}. To prove this, for each $t\in\CC$, define a right $C$-module $L_t$ as follows. The vector space at each vertex of $Q$ is $\CC$. Right multiplication by every $a_i$ is the identity map, while right multiplication by $b_1,b_2,b_3,b_4$ is multiplication by $1$, $t$, $1+t$, $1+2t$, respectively.
These maps satisfy the relations given in \eqref{eq:canonical relations two sided}.
Thus $L_t$ is a well-defined right $C$-module.
Since the maps corresponding to all $a_i$ and to $b_1$ are nonzero, every
endomorphism of $L_t$ is given by the same scalar at all vertices.
Then $\End_C(L_t)\cong\CC$, i.e., $L_t$ is indecomposable.
The same argument shows that $L_t\cong L_s$ implies $t=s$.

Let $\pi:B=T(C)=C\ltimes D(C)\to C$, $\pi(c,f)=c$ be the canonical algebra epimorphism,
and let $X_t=D(L_t)=\Hom_{\CC}(L_t,\CC)$. Then $X_t$ is a left $C$-module.
We can regard it as a left $B$-module by $b x=\pi(b)x$ $(b\in B,\ x\in X_t)$.
Equivalently, the ideal $D(C)=\ker\pi$ acts trivially on $X_t$.
The modules $X_t$ are pairwise non-isomorphic and indecomposable.
Since $B$ has only finitely many isomorphism classes of indecomposable projective modules,
at most finitely many of the $X_t$ are projective.
We remove these modules and keep the notation $X_t$ for the remaining infinite family.

We claim that, for every $n\>= 0$,
\[ (\Omega_B^n(X_t),0) \le_{\oplus} \Omega_A^n(X_t,0). \]
This is clear for $n=0$. If it holds for some $n$, then \eqref{eq:syzygy B part two sided} shows that the syzygy of the displayed summand contains $(\Omega_B^{n+1}(X_t),0)$ as a direct summand. The claim therefore follows by induction.
Since $B$ is self-injective and $X_t$ is nonprojective, $\Omega_B^n(X_t)$ is nonprojective for every $n$.
Moreover, by \eqref{eq:B period}, we have $\Omega_B^4(X_t) \cong X_t$ in the stable module category.
Since both sides have no projective direct summands, they are isomorphic as $B$-modules.
Consequently, for each fixed $n$, the modules $\Omega_B^n(X_t)$ form an infinite family of pairwise non-isomorphic indecomposable modules.
Thus, for every $n$, infinitely many pairwise non-isomorphic indecomposable modules occur as direct summands of $n$-th syzygies over $A$.
Hence $A$ is syzygy-infinite.

\item Finally, for every $n$, we have $(\Omega_B^n(X_t),0)\le_{\oplus}\Omega_A^n(X_t,0)$.
Therefore $\pdim (X_t,0) = \infty$, and then we obtain $\gldim A=\infty.$
\end{itemize}

}

\newpage


\addcontentsline{toc}{section}{References}



\begin{thebibliography}{10}

\bibitem{Alex1932}
P.~Alexandroff.
\newblock Dimensionstheorie: Ein beitrag zur geometrie der abgeschlossenen
  mengen.
\newblock {\em Math. Ann.}, 106(1):161--238, 1932.
\newblock \href{https://doi.org/10.1007/BF01455884}{DOI:10.1007/BF01455884}.

\bibitem{Alex1968}
P.~Alexandroff.
\newblock On the fundamental theorem of homological dimension theory.
\newblock {\em Dokl. Akad. Nauk SSSR}, 180(3):519--522, 1968.
\newblock \url{http://mi.mathnet.ru/dan33859},
  \href{https://mathscinet.ams.org/mathscinet-getitem?mr=0242150}{MR:0242150},
  \href{https://zbmath.org/?q=an:0185.26503}{zbMATH:0185.26503}.

\bibitem{APL2021diff}
M.~Alvarez-Picallo and J.-S.~P. Lemay.
\newblock Cartesian difference categories.
\newblock {\em Log. Methods Comput. Sci.}, 17(3):paper no. 23, 48, 2021.
\newblock
  \href{https://doi.org/10.46298/lmcs-17(3:23)2021}{DOI:10.46298/lmcs-17(3:23)2021}.

\bibitem{AF1992}
F.~W. Anderson and K.~R. Fuller.
\newblock {\em Rings and Categories of Modules}, volume~13 of {\em Graduate
  Texts in Mathematics}.
\newblock Springer-Verlag, New York, 2 edition, 1992.
\newblock
  \href{https://doi.org/10.1007/978-1-4612-4418-9}{DOI:10.1007/978-1-4612-4418-9}.

\bibitem{ASS2006}
I.~Assem, D.~Simson, and A.~Skowro\'{n}ski.
\newblock {\em Elements of the Representation Theory of Associative Algebras,
  Volume 1 Techniques of Representation Theory}.
\newblock Cambridge University Press, Cambridge, 2006.
\newblock
  \href{https://doi.org/10.1017/CBO9780511614309}{DOI:10.1017/CBO9780511614309}.

\bibitem{AS1987}
I.~Assem and A.~Skowro\'{n}ski.
\newblock Iterated tilted algebras of type {$\widetilde{\mathbb{A}}_n$}.
\newblock {\em Math. Z.}, 195(2):269--290, 1987.
\newblock
  \href{https://link.springer.com/article/10.1007/BF01166463}{DOI:10.1007/bf01166463}.

\bibitem{Aus1955}
M.~Auslander.
\newblock On the dimension of modules and algebras {III}: Global dimension.
\newblock {\em Nagoya Math. J.}, 9:67--77, 1955.
\newblock
  \href{https://doi.org/10.1017/S0027763000023291}{DOI:10.1017/S0027763000023291}.

\bibitem{ARS1995}
M.~Auslander, I.~Reiten, and S.~O. Smal{\o}.
\newblock {\em Representation Theory of Artin Algebras}, volume~36 of {\em
  Cambridge Studies in Advanced Mathematics}.
\newblock Cambridge University Press, Cambridge, 1995.
\newblock
  \href{https://doi.org/10.1017/CBO9780511623608}{DOI:10.1017/CBO9780511623608}.

\bibitem{Bass1960}
H.~Bass.
\newblock Finitistic dimension and a homological generalization of semi-primary
  rings.
\newblock {\em T. Am. Math. Soc.}, 95(3):466--488, 1960.
\newblock
  \href{http://doi.org/10.1090/S0002-9947-1960-0157984-8}{DOI:10.1090/S0002-9947-1960-0157984-8}.

\bibitem{Bel2011}
A.~Beligiannis.
\newblock On algebras of finite {C}ohen-{M}acaulay type.
\newblock {\em Adv. Math.}, 226(1), 2011.
\newblock
  \href{https://dx.doi.org/10.1016/j.aim.2010.09.006}{DOI:10.1016/j.aim.2010.09.006}.

\bibitem{BCS2006diff}
R.~F. Blute, J.~R.~B. Cockett, and R.~A.~G. Seely.
\newblock Differential categories.
\newblock {\em Math. Struct. Comput. Sci.}, 16(06):1049--1083, 2006.
\newblock \href{https://doi.org/10.1017/S0960129506005676}
  {DOI:10.1017/S0960129506005676}.

\bibitem{BCS2015diff}
R.~F. Blute, J.~R.~B. Cockett, and R.~A.~G. Seely.
\newblock \href{http://emis.icm.edu.pl/journals/TAC/volumes/22/23/22-23.pdf}
  {Cartesian differential storage categories}.
\newblock {\em Theor. Appl. Categ.}, 30(18):620--687, 2015.
\newblock
  \href{https://mathscinet.ams.org/mathscinet/article?mr=3346979}{AMS:MR3346979}.

\bibitem{BP2020findim}
D.~Bravo and C.~Paquette.
\newblock Idempotent reduction for the finitistic dimension conjecture.
\newblock {\em P. Am. Math. Soc.}, 148(5):1891--1900, 2020.
\newblock \href{https://doi.org/10.1090/proc/14945}{DOI:10.1090/proc/14945},
  \href{https://arxiv.org/abs/1902.00317}{arXiv:1902.00317}.

{\addtext
\bibitem{BBK2002almostKoszul}
S.~Brenner, M.~C.~R. Butler, and A.~D. King.
\newblock Periodic algebras which are almost Koszul.
\newblock {\em Algebr. Represent. Theory}, 5(4):331--368, 2002.
\newblock \href{https://doi.org/10.1023/A:1020146502185}{DOI:10.1023/A:1020146502185}.
}

\bibitem{BD2017}
I.~Burban and Y.~Drozd.
\newblock On the derived categories of gentle and skew-gentle algebras:
  homological algebra and matrix problems.
\newblock \href{http://arxiv.org/abs/1706.08358}{arXiv:1706.08358}, 2017.

\bibitem{CE1956}
H.~Cartan and S.~Eilenberg.
\newblock {\em Homological algebra}.
\newblock Princeton University Press, 1956.

{\addtext\bibitem{CDIM2025periodic}
A.~Chan, E.~Darp\"o, O.~Iyama and R.~Marczinzik,
\newblock Periodic trivial extension algebras and fractionally Calabi--Yau algebras,
\newblock \emph{Ann. Sci. \'{E}c. Norm. Sup\'{e}r.}, 58(2), 463--510, 2025.
}

\bibitem{ChX2017Finidim}
H.~Chen and C.~Xi.
\newblock Recollements of derived categories {III}: {Finitistic} dimensions.
\newblock {\em J. Lond. Math. Soc., II. Ser.}, 95(2):633--658, 2017.
\newblock \href{https://doi.org/10.1112/jlms.12026}{DOI:10.1112/jlms.12026}.



\bibitem{CL2018Cart-int}
J.~R.~B. Cockett and J.-S.~P. Lemay.
\newblock Cartesian integral categories and contextual integral categories.
\newblock In {\em Proceedings of the {T}hirty-{F}ourth {C}onference on the
  {M}athematical {F}oundations of {P}rogramming {S}emantics ({MFPS} {XXXIV})},
  volume 341 of {\em Electron. Notes Theor. Comput. Sci.}, pages 45--72.
  Elsevier Sci. B. V., Amsterdam, 2018.
\newblock
  \href{https://doi.org/10.1016/j.entcs.2018.11.004}{DOI:10.1016/j.entcs.2018.11.004}.

\bibitem{CL2018int}
J.~R.~B. Cockett and J.-S.~P. Lemay.
\newblock Integral categories and calculus categories.
\newblock {\em Math. Struct. Comp. Sci.}, 29(2):243--308, 2018.
\newblock
  \href{https://doi.org/10.1017/S0960129518000014}{DOI:10.1017/S0960129518000014}.

\bibitem{Dan1918}
P.~J. Daniell.
\newblock A general form of integral.
\newblock {\em Ann. Math.}, 19(4):279--294, 1918.
\newblock \href{https://doi.org/10.2307/1967495}{DOI:10.2307/1967495}.

\bibitem{Dan1919-II}
P.~J. Daniell.
\newblock Functions of limited variation in an infinite number of dimension.
\newblock {\em Ann. Math.}, 21(1):30--38, 1919.
\newblock \href{https://doi.org/10.2307/2007133}{DOI:10.2307/2007133}.

\bibitem{Dan1919-I}
P.~J. Daniell.
\newblock Integrals in an infinite number of dimensions.
\newblock {\em Ann. Math.}, 20(4):281--288, 1919.
\newblock \href{https://doi.org/10.2307/1967122}{DOI:10.2307/1967122}.

\bibitem{DunfordSchwartz1958}
N.~Dunford and J.~T. Schwartz.
\newblock {\em Linear Operators. Part I: General Theory}, volume~7 of {\em Pure
  and Applied Mathematics}.
\newblock Interscience Publishers, New York, 1958.

{\addtext
\bibitem{Dugas2010periodic}
A.~S. Dugas.
\newblock Periodic resolutions and self-injective algebras of finite type.
\newblock {\em J. Pure Appl. Algebra}, 214(6):990--1000, 2010.
\newblock \href{https://doi.org/10.1016/j.jpaa.2009.09.012}{DOI:10.1016/j.jpaa.2009.09.012}.
}

\bibitem{Folland1999}
G.~B. Folland.
\newblock {\em Real Analysis: Modern Techniques and Their Applications}.
\newblock Pure and Applied Mathematics. John Wiley \& Sons, New York, 2
  edition, 1999.

\bibitem{FOZ2025}
S.~Ford, A.~Oswald, and J.~J. Zhang.
\newblock Homological conditions on locally gentle algebras.
\newblock {\em Journal of Pure and Applied Algebra}, page 108343, 2026.
\newblock
  \href{https://doi.org/10.1016/j.jpaa.2026.108343}{DOI:j.jpaa.2026.108343}
  \href{http://arxiv.org/abs/2409.08333}{arXiv:2409.08333}.

\bibitem{Gab1972}
P.~Gabriel.
\newblock Indecomposable representations {I}.
\newblock {\em Manuscripta Math.}, 6:71--109, 1972.

\bibitem{Gab1973}
P.~Gabriel.
\newblock \href{https://cir.nii.ac.jp/crid/1571417125042123520}{Indecomposable
  representations II}.
\newblock {\em Symposia Math.}, 11:81--104, 1973.

\bibitem{Gab2006Repr}
P.~Gabriel.
\newblock Repr{\'e}sentations ind{\'e}composables.
\newblock In {\em S{\'e}minaire Bourbaki vol. 1973/74 Expos{\'e}s 436--452},
  pages 143--169. Springer, 2006.
\newblock \href{https://doi.org/10.1007/BFb0066369} {DOI:10.1007/BFb0066369}.

\bibitem{GR2005}
C.~Gei{\ss} and I.~Reiten.
\newblock Gentle algebras are {G}orenstein, {\rm in {{\it {r}epresentations of
  {a}lgebras and {r}elated {t}opics}, {p}roceedings from the 10th
  {i}nternational {c}onference, {t}oronto, {c}anada, {j}uly 15 -- {a}ugust 10,
  2002}}.
\newblock {\em Amer. Math. Soc., Providence, RI, Fields Institute Comm.},
  1(45):129--133, 2005.
\newblock \href{http://dx.doi.org/10.1090/fic/045}{DOI:10.1090/fic/045}.

\bibitem{Gel2022depth}
V.~G{\'e}linas.
\newblock The depth, the delooping level and the finitistic dimension.
\newblock {\em Adv. Math}, 394:108052, 2022.
\newblock
  \href{https://doi.org/10.1016/j.aim.2021.108052}{DOI:10.1016/j.aim.2021.108052},
  \href{https://arxiv.org/abs/2004.04828}{arXiv:2004.04828}.

{\addtext
\bibitem{GLS2005}
C.~Gei\ss, B.~Leclerc, and J.~Schr\"{o}er.
\newblock Semicanonical bases and preprojective algebras.
\newblock {\em Ann. Sci. \'{E}cole Norm. Sup. (4)}, 38(2):193--253, 2005.
\newblock \href{https://doi.org/10.1016/j.ansens.2004.12.001}{DOI:10.1016/j.ansens.2004.12.001}.
}

\bibitem{GP2020}
N.~Gigli and E.~Pasqualetto.
\newblock The theory of normed modules.
\newblock In {\em Lectures on Nonsmooth Differential Geometry, SISSA Springer
  Series (SISSASS)}, pages 67--69, Cham, 2020. Springer.
\newblock
  \href{https://doi.org/10.1007/978-3-030-38613-9_3}{DOI:10.1007/978-3-030-38613-9\_3}.

\bibitem{GI2025findim}
R.~Guo and K.~Igusa.
\newblock Derived delooping levels and finitistic dimension.
\newblock {\em Adv. Math.}, 464:110152, 2025.
\newblock
  \href{https://doi.org/10.1016/j.aim.2025.110152}{DOI:10.1016/j.aim.2025.110152}
  \href{https://arxiv.org/abs/2311.00661}{arXiv:2311.00661}.

\bibitem{Guo2013}
T.~Guo.
\newblock On some basic theorems of continuous module homomorphisms between
  random normed modules.
\newblock {\em J. Funct. Space.}, 54(1):115--130, 2013.
\newblock \href{https://doi.org/10.1155/2013/989102} {DOI:
  10.1155/2013/989102}.

\bibitem{Huis1995}
B.~Huisgen-Zimmermann.
\newblock The finitistic dimension conjectures---a tale of 3.5 decades.
\newblock In A.~Facchini and C.~Menini, editors, {\em Abelian Groups and
  Modules (Padova, 1994)}, volume 343 of {\em Mathematics and Its
  Applications}, pages 501--517. Kluwer Academic Publishers, Dordrecht, 1995.

\bibitem{IT2005}
K.~Igusa and G.~Todorov.
\newblock On the finitistic global dimension conjecture for {Artin} algebras.
\newblock In {\em Representations of Algebras and Related Topics}, volume~45 of
  {\em Fields Institute Communications}, pages 201--204. American Mathematical
  Society, Providence, RI, 2005.

\bibitem{IL2023ana}
S.~Ikonicoff and J.-S.~P. Lemay.
\newblock
  \href{https://cahierstgdc.com/wp-content/uploads/2023/04/IKONICOFF-LEMAY-LXIV-2.pdf}
  {Cartesian differential comonads and new models of {C}artesian differential
  categories}.
\newblock {\em Cah. Topol. G\'{e}om. Diff\'{e}r. Cat\'{e}g.}, 64(2):198--239,
  2023.
\newblock \href{https://mathscinet.ams.org/mathscinet/article?mr=4605864}
  {AMS:MR4605864}.

\bibitem{Jen1966}
C.~U. Jensen.
\newblock On homological dimensions of rings with countably generated ideals.
\newblock {\em Math. Scand.}, 18(2):97--105, 1966.
\newblock \url{https://www.jstor.org/stable/24490078}.

\bibitem{Krause2015}
H.~Krause.
\newblock Krull--schmidt categories and projective covers.
\newblock {\em Expositiones Mathematicae}, 33(4):535--549, 2015.
\newblock
  \href{https://doi.org/10.1016/j.exmath.2015.10.001}{DOI:10.1016/j.exmath.2015.10.001}.

\bibitem{Lei2023FA}
T.~Leinster.
\newblock A categorical derivation of {L}ebesgue integration.
\newblock {\em J. Lond. Math. Soc.}, 107(6):1959--1982, 2023.
\newblock \href{https://doi.org/10.1112/jlms.12730}{DOI:10.1112/jlms.12730},
  \href{https://arxiv.org/abs/2011.00412}{arXiv:2011.00412}.

\bibitem{Lemay2019}
J.-S.~P. Lemay.
\newblock Differential algebras in codifferential categories.
\newblock {\em J. Pure. Appl. Algebra}, 223(10):4191--4225, 2019.
\newblock \href{https://doi.org/10.1016/j.jpaa.2019.01.005}
  {DOI:10.1016/j.jpaa.2019.01.005}.

\bibitem{Lemay2023}
J.-S.~P. Lemay.
\newblock {Cartesian Differential {K}leisli Categories}.
\newblock {\em ENTICS}, 3(23):1--13, 2023.
\newblock
  \href{https://doi.org/10.46298/entics.12278}{DOI:10.46298/entics.12278}.

\bibitem{LLL2025}
M.~Liu, S.~Liu, and Y.-Z. Liu.
\newblock Normed modules, integral sequences, and integrals with variable upper
  limits.
\newblock {\em B. Iran. Math. Soc.}, 51(79):1--49, 2025.
\newblock \href{https://doi.org/10.1007/s41980-025-00990-4}
  {DOI:10.1007/s41980-025-00990-4}.

\bibitem{Liu2025pre}
Y.-Z. Liu.
\newblock Normed representations of weighted quivers.
\newblock \href{https://arxiv.org/abs/2507.06962}{arXiv:2507.06962}, 2025.

\bibitem{LGH2024}
Y.-Z. Liu, H.~Gao, and Z.~Huang.
\newblock Homological dimensions of gentle algebras via geometric models.
\newblock {\em Sci. China Math.}, 67(4):733--766, 2023.
\newblock
  \href{https://www.sciengine.com/doi/10.1007/s11425-022-2120-8}{DOI:10.1007/s11425-022-2120-8}.

\bibitem{LLHZ2025}
Y.-Z. Liu, S.~Liu, Z.~Huang, and P.~Zhou.
\newblock Normed modules and the categorification of integrations, series
  expansions, and differentiations.
\newblock {\em Sci. China Math.}, 69(3):571--614, 2025.
\newblock \href{https://www.sciengine.com/SCM/doi/10.1007/s11425-024-2418-3}
  {DOI:10.1007/s11425-024-2418-3},
  \href{http://arxiv.org/abs/2405.02777}{arXiv:2405.02777}.

\bibitem{MN2024findim}
J.~MacQuarrie and F.~Naves.
\newblock Quotient bifinite extensions and the finitistic dimension conjecture.
\newblock {\em P. Am. Math. Soc.}, 152(02):585--590, 2024.
\newblock \href{https://doi.org/10.1090/proc/17754}{DOI:10.1090/proc/17754},
  \href{https://arxiv.org/abs/2306.02539}{arXiv:2306.02539}.

\bibitem{MT2015}
N.~Mahdou and M.~Tamekkante.
\newblock On {G}orenstein global dimension of tensor product of algebras over a
  field.
\newblock {\em Gulf Journal of Mathematics}, 3:30--37, 2015.
\newblock
  \href{https://doi.org/10.56947/gjom.v3i2.159}{DOI:10.56947/gjom.v3i2.159}.

\bibitem{MS2022findim}
P.~Moradifar and J.~{\v{S}}aroch.
\newblock Finitistic dimension conjectures via {G}orenstein projective
  dimension.
\newblock {\em J. Algebra}, 591:15--35, 2022.
\newblock
  \href{https://doi.org/10.1016/j.jalgebra.2021.10.026}{DOI:10.1016/j.jalgebra.2021.10.026},
  \href{https://arxiv.org/abs/2006.02182}{arXiv:2006.02182}.

\bibitem{Munkres2000}
J.~R. Munkres.
\newblock {\em Topology}.
\newblock Prentice Hall, Upper Saddle River, NJ, 2 edition, 2000.

\bibitem{OPS2018}
S.~Opper, P.-G. Plamondon, and S.~Schroll.
\newblock A geometric model for the derived category of gentle algebras.
\newblock \href{http://arxiv.org/abs/1801.09659}{arXiv:1801.09659}, 2018.

\bibitem{QZZ2022}
Y.~Qiu, Y.~Zhou, and C.~Zhang.
\newblock Two geometric models for graded skew-gentle algebras.
\newblock \href{http://arxiv.org/abs/2212.10369}{arXiv:2212.10369}, 2022.

\bibitem{Rogers1999}
C.~A. Rogers.
\newblock {\em Hausdorff Measures}.
\newblock Cambridge University Press, Cambridge, England, 2nd edition, 1999.

\bibitem{R1979}
J.~J. Rotman.
\newblock {\em An Introduction to Homological Algebra (Second Edition)}.
\newblock Academic Press, 1979.
\newblock \href{https://doi.org/10.1007/b98977}{DOI:10.1007/b98977}.

\bibitem{Ryan2002}
R.~A. Ryan.
\newblock The projective tensor product.
\newblock In {\em \href{https://doi.org/10.1007/978-1-4471-3903-4}
  {Introduction to Tensor Products of Banach Spaces}}, pages 15--43, London,
  2002. Springer London.
\newblock
  \href{https://doi.org/10.1007/978-1-4471-3903-4_2}{DOI:10.1007/978-1-4471-3903-4\_2}.

\bibitem{Sag1965ana}
I.~Segal.
\newblock Algebraic integration theory.
\newblock {\em Bull. Amer. Math. Soc.}, 71(3):419--489, 1965.
\newblock
  \href{https://community.ams.org/journals/bull/1965-71-03/S0002-9904-1965-11284-8}
  {DOI:S0002-9904-1965-11284-8}.

{\addtext
\bibitem{TLZ2014Morita}
G.~Tang, C.~Li and Y.~Zhou,
\newblock Study of Morita contexts,
\newblock \emph{Comm. Algebra}, 42(4): 1668--1681, 2014.
\newblock
\href{https://www.tandfonline.com/doi/abs/10.1080/00927872.2012.748327}{DOI:10.1080/00927872.2012.748327}~/~
\href{https://zbmath.org/1292.16020}{zbMath:1292.16020}.}

\bibitem{Vcech1933}
E.~\v{C}ech.
\newblock P{\v{r}}{\'\i}sp{\v{e}}vek k teorii dimense.
\newblock {\em {\v{C}}as. Mat. Fys.}, 62(7):277--291, 1933.
\newblock
  \href{https://doi.org/10.21136/CPMF.1933.121898}{DOI:10.21136/CPMF.1933.121898}.

\bibitem{Wei2009}
J.~Wei.
\newblock Finitistic dimension and {Igusa--Todorov} algebras.
\newblock {\em Adv. Math.}, 222(6):2215--2226, 2009.
\newblock
  \href{https://doi.org/10.1016/j.aim.2009.07.008}{10.1016/j.aim.2009.07.008}.

\bibitem{XinZhang2026}
Z.~Xin and L.~Zhang.
\newblock Global dimension of a string algebra.
\newblock {\em Arch. Math.}, 2026.
\newblock
  \href{https://doi.org/10.1007/s00013-026-02262-x}{DOI:10.1007/s00013-026-02262-x},
  \href{https://arxiv.org/abs/2607.11552}{arXiv:2607.11552}.

\bibitem{Z2017globalcohomo}
C.~Zhang.
\newblock On global cohomological width of {A}rtin algebras.
\newblock {\em Colloq. Math.}, 146:31--46, 2017.
\newblock
  \href{https://doi.org/10.4064/cm6714-10-2015}{DOI:10.4064/cm6714-10-2015}.

\bibitem{ZZ2020}
H.~Zhang and X.~Zhu. 
\newblock {G}orenstein global dimension of recollements of abelian categories.
\newblock {\em Commun. Algebra}, 48(2):467--483, 2020.
\newblock
  \href{https://doi.org/10.1080/00927872.2019.1648650}{DOI:10.1080/00927872.2019.1648650}.

\bibitem{ZH2020findim}
J.~Zheng and Z.~Huang.
\newblock The finitistic dimension and chain conditions on ideals.
\newblock {\em Glasgow Math. J.}, 64(1):37--44, 2022.
\newblock
  \href{https://doi.org/10.1017/S001708952000052X}{DOI:10.1017/S001708952000052X}.

\end{thebibliography}

\def\cprime{$'$}

\end{document}